\documentclass{imanum_NoJournalName}

\received{March 31, 2016}

\usepackage{graphicx}
\usepackage{amsmath}
\usepackage{amsthm}
\usepackage{amssymb}
\usepackage{color}
\usepackage{float}
\usepackage{url}

\newtheorem{thm}{Theorem}
\newtheorem{lem}[thm]{Lemma}
\newtheorem{prop}[thm]{Proposition}

\theoremstyle{definition}

\newtheorem{assump}{Assumption}
\newtheorem{ex}{Example}
\newtheorem{rem}{Remark}
\newtheorem{prob}{Problem}

\numberwithin{equation}{section}
\numberwithin{thm}{section}
\numberwithin{rem}{section}
\numberwithin{ex}{section}

\def\e{\mathrm{e}}
\def\i{\mathrm{i}}
\def\d{\mathrm{d}}

\allowdisplaybreaks[4]

\begin{document}

\title{
Potential theoretic approach to design of accurate formulas for function approximation in symmetric weighted Hardy spaces}
\shorttitle{POTENTIAL THEORETIC APPROACH TO DESIGN OF ACCURATE FORMULAS}

\author{
{\sc Ken'ichiro Tanaka$^{\ast}$} \\
Department of Mathematical Engineering, Faculty of Engineering, Musashino University \\
3-3-3, Ariake, Koto-ku, Tokyo 135-8181, Japan \\
$^{\ast}${\rm Corresponding author: ketanaka@musashino-u.ac.jp} \\[3pt]
{\sc Tomoaki Okayama} \\
Department of Systems Engineering, Graduate School of Information Sciences, Hiroshima City University \\
3-4-1, Ozuka-higashi, Asaminami-ku, Hiroshima 731-3194, Japan \\[3pt]
{\sc Masaaki Sugihara} \\
Department of Physics and Mathematics, College of Science and Engineering, Aoyama Gakuin University \\
5-10-1, Fuchinobe, Chuo-ku, Sagamihara-shi, Kanagawa 252-5258, Japan
}
\shortauthorlist{K.~Tanaka, T.~Okayama, M.~Sugihara}

\maketitle

\begin{abstract}
% Body of abstract:
{
We propose a method for designing accurate interpolation formulas on the real axis
for the purpose of function approximation in weighted Hardy spaces. 
%As a weighted Hardy space, 
In particular, 
we consider the Hardy space of functions that are analytic in a strip region around the real axis, 
being characterized by a weight function $w$ that determines
the decay rate of its elements in the neighborhood of infinity. 
Such a space is considered as a set of functions that are transformed
by variable transformations that realize a certain decay rate at infinity. 
Popular examples of such transformations are given by the
single exponential (SE) and double exponential (DE) transformations 
for the SE-Sinc and DE-Sinc formulas, 
which are very accurate owing to the accuracy of sinc interpolation 
in the weighted Hardy spaces with single and double exponential weights $w$, respectively. 
However, it is not guaranteed that 
the sinc formulas are optimal in weighted Hardy spaces, 
although Sugihara has demonstrated that they are near optimal.
An explicit form for an optimal approximation formula 
has only been given in weighted Hardy spaces with SE weights of a certain type. 
In general cases, explicit forms for optimal formulas have not been provided so far. 
We adopt a potential theoretic approach 
to obtain almost optimal formulas in weighted Hardy spaces 
in the case of general weight functions $w$. 
We formulate 
the problem of designing an optimal formula in each space
as an optimization problem written in terms of a Green potential with an external field. 
By solving the optimization problem numerically, 
we obtain an almost optimal formula in each space.  
Furthermore, some numerical results demonstrate the validity of this method. 
In particular, for the case of a DE weight, 
the formula designed by our method outperforms the DE-Sinc formula. }
%Keywords:
{weighted Hardy space;
function approximation; 
potential theory;
Green potential}
\end{abstract}

%------------------------------------------------------------
\section{Introduction}
\label{sec:intro}

We propose a method for designing accurate interpolation formulas on $\mathbf{R}$ 
for the purpose of function approximation in weighted Hardy spaces, which are defined by 
\begin{align}
\label{eq:def_weighted_Hardy}
\boldsymbol{H}^{\infty}(\mathcal{D}_{d}, w)
:=
\left\{
f : \mathcal{D}_{d} \to \mathbf{C} 
\ \left| \
f \text{ is analytic in } \mathcal{D}_{d} \text{ and } 
\| f \| < \infty 
\right.
\right\}, 
\end{align}
where $d>0$, $\mathcal{D}_{d} := \{ z \in \mathbf{C} \mid | \mathrm{Im} z | < d \}$, 
$w$ is a weight function satisfying $w(z) \neq 0$ for any $z \in \mathcal{D}_{d}$, and
\begin{align}
\label{eq:def_weighted_Hardy_norm}
\| f \| := \sup_{z \in \mathcal{D}_{d}} \left| \frac{f(z)}{w(z)} \right|. 
\end{align} 
Each of these is a space of functions that are analytic in the strip region $\mathcal{D}_{d}$, 
being characterized by the decay rate of its elements (functions) in the neighborhood of infinity. 

At first sight, 
each of the spaces $\boldsymbol{H}^{\infty}(\mathcal{D}_{d}, w)$ may appear to be limited as a space of functions 
that we approximate for numerical computations in practical applications. 
However, by using some variable transformations, 
we can transform reasonably wide ranges of functions to those 
in the space $\boldsymbol{H}^{\infty}(\mathcal{D}_{d}, w)$ for some $d$ and $w$. 
Popular examples of such transformations are given by {\it sinc numerical methods} \citep{bib:StengerSinc1993, bib:StengerSinc2011}, 
which are numerical methods based on the sinc function 
$\mathop{\mathrm{sinc}}(x) = \sin(\pi x)/(\pi x)$
and some useful transformations. 
Typical transformations used in the sinc numerical methods are single-exponential (SE) transformations and 
double-exponential (DE) transformations, which map some regions to the strip regions $\mathcal{D}_{d}$ for some $d > 0$
and transform given functions to those on $\mathcal{D}_{d}$ with SE decay or DE decay at infinity. 
For example, the map $\psi: \mathcal{D}_{\pi/4} \to \mathbf{C}$ of SE type given by
\begin{align}
\psi(z) := \tanh(z)
\end{align}
transforms a function $g: \{ \zeta \in \mathbf{C} \mid |\zeta| < 1 \} \to \mathbf{C}$ 
to $f := (g \circ \psi) : \mathcal{D}_{\pi/4} \to \mathbf{C}$. 
If the function $g$ satisfies $g(\zeta) = \mathrm{O}((1-\zeta^{2})^{\beta/2})$ as $\zeta \to \pm 1$ for some $\beta > 0$,  
then the transformed function $f$ has SE decay at infinity. 
Then, for a positive integer $N$ and an appropriately selected real $h>0$, the truncated sinc interpolation 
\begin{align}
f(x) \approx \sum_{k = -N}^{N} f(kh)\, \mathop{\mathrm{sinc}}(x/h - k)
\label{eq:ex_f_sinc}
\end{align}
can approximate $f$ accurately, 
if it is analytic and bounded with respect to the norm in \eqref{eq:def_weighted_Hardy_norm} on $\mathcal{D}_{\pi/4}$. 
In particular, if $f \in \boldsymbol{H}^{\infty}(\mathcal{D}_{\pi/4}, \mathop{\mathrm{sech}}^{\beta}(z))$, 
we can derive a theoretical error estimate of the approximation given by \eqref{eq:ex_f_sinc}. 
The sinc numerical methods involving SE and DE transformations are collectively called SE-Sinc and DE-Sinc methods, respectively. 
For further details about these, 
see \cite{bib:StengerSinc1993, bib:StengerSinc2011, bib:Sugihara_NearOpt_2003, bib:SugiharaMatsuo2004, bib:TanaSugiMuro2009}.

%Sugihara 
%\cite{bib:Sugihara_NearOpt_2003} 
\cite{bib:Sugihara_NearOpt_2003} 
not only derived upper bounds for the errors of the SE-Sinc and DE-Sinc approximations
but also demonstrated that the SE-Sinc and DE-Sinc approximations are ``near optimal''
in $\boldsymbol{H}^{\infty}(\mathcal{D}_{d}, w)$ with $w$ of SE and DE decay types, respectively. 
The term ``near optimality'' means that the upper bounds of the errors of the sinc approximations 
are bounded from below by some lower bounds of $E_{N}^{\mathrm{min}}(\boldsymbol{H}^{\infty}(\mathcal{D}_{d}, w))$, 
where $E_{N}^{\mathrm{min}}(\boldsymbol{H}^{\infty}(\mathcal{D}_{d}, w))$ 
is the minimum worst error of the $(2N+1)$-point interpolation formulas in $\boldsymbol{H}^{\infty}(\mathcal{D}_{d}, w)$. 
For a precise definition of $E_{N}^{\mathrm{min}}(\boldsymbol{H}^{\infty}(\mathcal{D}_{d}, w))$, see Section~\ref{sec:DefOptApprox}.
However, 
an explicit optimal formula attaining $E_{N}^{\mathrm{min}}(\boldsymbol{H}^{\infty}(\mathcal{D}_{d}, w))$
is only known in the limited case that $d = \pi/4$ and $w(z) = \mathop{\mathrm{sech}}^{\beta} (z)$ for some $\beta > 0$. 
In this case, an optimal formula is provided by the results of 
%Ganelius 
%\cite{bib:Ganelius1976} 
\cite{bib:Ganelius1976} 
and 
%Jang and Haber 
%\cite{bib:JangHaber2001}. 
\cite{bib:JangHaber2001}. 
In the other cases, explicit forms of optimal formulas have not yet been derived. 

Therefore, we proceed with the aim of finding
optimal formulas with explicit forms for 
$\boldsymbol{H}^{\infty}(\mathcal{D}_{d}, w)$ in the case of general weight functions $w$. 
We regard an approximation formula as optimal 
if it attains the minimum worst error $E_{N}^{\mathrm{min}}(\boldsymbol{H}^{\infty}(\mathcal{D}_{d}, w))$. 
Using the fact that
\begin{align}
E_{N}^{\mathrm{min}}(\boldsymbol{H}^{\infty}(\mathcal{D}_{d}, w)) 
= 
\inf_{a_{\ell} \in \mathbf{R}} 
\left\{ 
\sup_{x \in \mathbf{R}}
\left|
w(x) \prod_{\ell=-N}^{N} \tanh \left( \frac{\pi}{4d} (x - a_{\ell}) \right)
\right|
\right\}
\label{eq:keyfact_in_intro}
\end{align}
which was shown by 
%Sugihara 
%\cite{bib:Sugihara_NearOpt_2003}, 
\cite{bib:Sugihara_NearOpt_2003}, 
we reduce the problem of finding an optimal formula 
to the minimization problem~\eqref{eq:keyfact_in_intro}.
By taking the logarithm of the objective function in problem~\eqref{eq:keyfact_in_intro}, 
we consider the equivalent problem written in terms of a Green potential with an external field, 
which is presented as Problem~\ref{prob:Original} in Section~\ref{sec:ReducChar}. 
However, Problem~\ref{prob:Original} is not easily tractable owing to the lack of convexity.
Therefore, 
by introducing some approximations and using a potential theoretic approach, 
we arrive at Problem~\ref{prob:2Steps} in Section~\ref{sec:Proc}, 
which yields an approximate solution of Problem~\ref{prob:Original}. 
See the diagram in Figure \ref{fig:diagram_Problems}
for the process we apply to reduce Problem~\ref{prob:Original} to Problem~\ref{prob:2Steps}. 
The full details of this process are presented in Sections~\ref{sec:Reduction} and~\ref{sec:Proc}. 
Finally, we propose a numerical method for solving Problem~\ref{prob:2Steps} and 
find an ``almost optimal'' formula for $\boldsymbol{H}^{\infty}(\mathcal{D}_{d}, w)$. 
\begin{figure}[h]
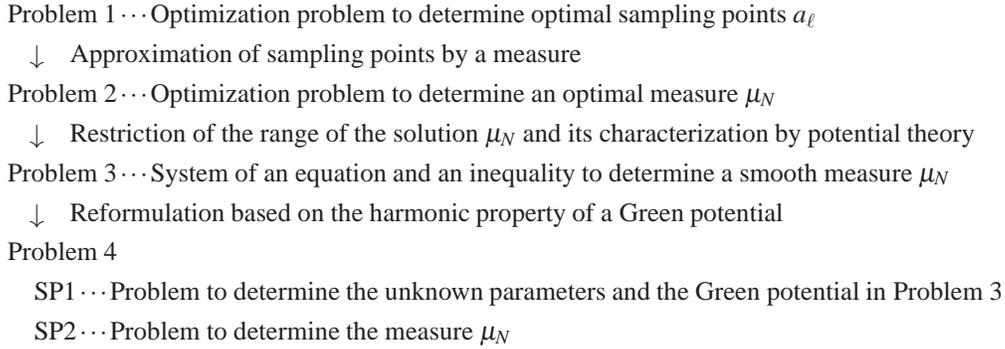

\begin{center}
\begin{align*}
& \text{Problem~\ref{prob:Original}} \cdots \text{Optimization problem to determine optimal sampling points $a_{\ell}$} \\
& \quad \rotatebox[origin=c]{90}{$\leftarrow$}\quad \text{Approximation of sampling points by a measure} \\
& \text{Problem~\ref{prob:ContMeas}} \cdots \text{Optimization problem to determine an optimal measure $\mu_{N}$} \\
& \quad \rotatebox[origin=c]{90}{$\leftarrow$}\quad \text{Restriction of the range of the solution $\mu_{N}$ and its characterization by potential theory} \\
& \text{Problem~\ref{prob:IE_IINEQ}} \cdots \text{System of an equation and an inequality to determine a smooth measure $\mu_{N}$} \\
& \quad \rotatebox[origin=c]{90}{$\leftarrow$}\quad \text{Reformulation based on the harmonic property of a Green potential} \\
& \text{Problem~\ref{prob:2Steps}} \\
& \quad \text{SP1} \cdots \text{Problem to determine the unknown parameters and the Green potential in Problem~\ref{prob:IE_IINEQ}} \\
& \quad \text{SP2} \cdots \text{Problem to determine the measure $\mu_{N}$} 
\end{align*}
\end{center}
\caption{Diagram showing the relations of Problems~\ref{prob:Original}--\ref{prob:2Steps}}
\label{fig:diagram_Problems}
\end{figure}

The rest of this paper is organized as follows. 
In Section \ref{sec:MathPre}, we present some mathematical preliminaries 
concerning assumptions for weight functions $w$, 
the formulation of the optimality of approximations on $\boldsymbol{H}^{\infty}(\mathcal{D}_{d}, w)$, 
and some fundamental facts relating to potential theory. 
Subsequently,   in Section \ref{sec:Reduction}
we approximate the problem (Problem~\ref{prob:Original}) for finding an optimal formula for each $w$
as a potential problem (Problem~\ref{prob:ContMeas}) and characterize its solution. 
We then approximately solve this through Problems~\ref{prob:IE_IINEQ} and \ref{prob:2Steps}
to present a procedure for designing an accurate formula in Section \ref{sec:Proc}. 
In Section \ref{sec:Err}, 
we estimate the error of the formula under certain assumptions. 
We postpone the lengthy proofs of certain lemmas, and present these in Section \ref{sec:Proofs}. 
In Section \ref{sec:NumEx}, 
we present some numerical results supporting the validity of our method. 
The programs used for the numerical experiments are available on the web page \cite{bib:TanaMatlab2015}.
The results show that 
in the case of a DE weight function $w$, 
the formula designed by our method outperforms the DE-Sinc formula. 
Finally, we conclude this paper in Section \ref{sec:Concl}, 
in which we mention some considerations 
about the computational complexity and the numerical stability of our method.

%------------------------------
\section{Mathematical preliminaries}
\label{sec:MathPre}

\subsection{Weight functions and weighted Hardy spaces}

Let $d$ be a positive real number, and let $\mathcal{D}_{d}$ be a strip region defined by 
$\mathcal{D}_{d} := \{ z \in \mathbf{C} \mid |\mathop{\mathrm{Im}} z | < d \}$. 
In order to specify the analytical property of a weight function $w$ on $\mathcal{D}_{d}$, 
we use a function space $B(\mathcal{D}_{d})$ of all functions $\zeta$ that are analytic in $\mathcal{D}_{d}$
such that
\begin{align}
\lim_{x \to \pm \infty} \int_{-d}^{d} |\zeta(x+\i y)|\, \d y = 0
\end{align} 
and 
\begin{align}
\lim_{y \to d-0}
\int_{-\infty}^{\infty} ( |\zeta(x+\i y)| + |\zeta(x-\i y)| )\, \d x < \infty.
\end{align} 
Let $w$ be a complex valued function on $\mathcal{D}_{d}$. 
We regard $w$ as a weight function on $\mathcal{D}_{d}$
if $w$ satisfies the following assumption. 
\begin{assump}
\label{assump:w}
The function $w$ belongs to $B(\mathcal{D}_{d})$, 
does not vanish at any point in $\mathcal{D}_{d}$, 
and takes positive real values on the real axis.
\end{assump}

For a weight function $w$ that satisfies Assumption \ref{assump:w}, 
we define a weighted Hardy space on $\mathcal{D}_{d}$ by \eqref{eq:def_weighted_Hardy}, i.e.,
\begin{align}
\label{eq:def_weighted_Hardy_rev}
\boldsymbol{H}^{\infty}(\mathcal{D}_{d}, w)
:=
\left\{
f : \mathcal{D}_{d} \to \mathbf{C} 
\ \left| \
f \text{ is analytic in } \mathcal{D}_{d} \text{ and } 
\| f \| < \infty 
\right.
\right\}, 
\end{align}
where
\begin{align}
\label{eq:def_weighted_Hardy_norm_rev}
\| f \| := \sup_{z \in \mathcal{D}_{d}} \left| \frac{f(z)}{w(z)} \right|. 
\end{align}

In this paper, for the simplicity of our mathematical arguments, 
we apply the following additional assumptions for a weight function $w$. 

\begin{assump}
\label{assump:w_even}
The function $w$ is even on $\mathbf{R}$. 
\end{assump}

\begin{assump}
\label{assump:w_convex}
The function $\log w$ is concave on $\mathbf{R}$.
\end{assump}

\subsection{Optimal approximation}
\label{sec:DefOptApprox}

We provide a mathematical formulation for the optimality of approximation formulas 
in the weighted Hardy spaces $\boldsymbol{H}^{\infty}(\mathcal{D}_{d}, w)$
with weight functions satisfying Assumptions \ref{assump:w} and \ref{assump:w_even}. 
In this regard, for a given positive integer $N$, 
we first consider all possible $(2N+1)$-point interpolation formulas on $\mathbf{R}$ 
that can be applied to any functions $f \in \boldsymbol{H}^{\infty}(\mathcal{D}_{d}, w)$. 
Then, we choose one of the formulas such that it gives the minimum worst error in $\boldsymbol{H}^{\infty}(\mathcal{D}_{d}, w)$. 
The precise definition of the minimum worst error, denoted by 
$E_{N}^{\mathrm{min}}(\boldsymbol{H}^{\infty}(\mathcal{D}_{d}, w))$, is given by 
\begin{align}
& E_{N}^{\mathrm{min}}(\boldsymbol{H}^{\infty}(\mathcal{D}_{d}, w)) \notag \\ 
& :=
\inf_{1 \leq l \leq N} 
\inf_{\begin{subarray}{c} m_{-l}, \ldots , m_{l} \\ m_{-l}+\cdots+m_{l} = 2N+1 \end{subarray}}
\inf_{\begin{subarray}{c} a_{j} \in \mathcal{D}_{d} \\ \text{distinct} \end{subarray}}
\inf_{\phi_{jk}}
\left[
\sup_{\| f \| \leq 1}
\sup_{x \in \mathbf{R}}
\left|
f(x) - \sum_{j = -l}^{l} \sum_{k = 0}^{m_{j} - 1} f^{(k)}(a_{j})\, \phi_{jk}(x)
\right|
\right],
\label{eq:def_E_min} 
\end{align}
where the $\phi_{jk}$'s are functions that are analytic in $\mathcal{D}_{d}$. 
Here, 
we restrict ourselves without loss of generality to sequences of sampling points that are symmetric with respect to the origin. 
That is, they have the form $a_{-l}, a_{-l+1}, \ldots, a_{l-1}, a_{l}$ with $a_{-k} = - a_{k}\ (k=1,\ldots, l)$. 
This is because we consider even weight functions $w$ according to Assumption~\ref{assump:w_even}. 
Owing to this reason, our definition \eqref{eq:def_E_min} of $E_{N}^{\mathrm{min}}(\boldsymbol{H}^{\infty}(\mathcal{D}_{d}, w))$ 
is slightly different from that given in 
\cite{bib:Sugihara_NearOpt_2003}. 

In the case that $d = \pi/4$ and $w(z) = \mathop{\mathrm{sech}}^{\beta} (z)$ for some $\beta > 0$, 
the exact order of the minimum worst error, 
according to 
\cite{bib:AnderssonQuadHp1980, 
bib:Ganelius1976, 
bib:Sugihara_NearOpt_2003}, 
is given by
\begin{align}
C_{\beta}\, 
\exp \left( 
-\pi \sqrt{\frac{\beta\, N}{2}}
\right)
\leq
E_{N}^{\mathrm{min}}(\boldsymbol{H}^{\infty}(\mathcal{D}_{\pi/4}, \mathrm{sech}^{\beta} (z)))
\leq
C_{\beta}'\, 
\exp \left( 
-\pi \sqrt{\frac{\beta\, N}{2}}
\right), 
\label{eq:ExactEminSech}
\end{align}
where $C_{\beta}$ and $C_{\beta}'$ are positive constants depending on $\beta$. 
In particular, the upper estimate in \eqref{eq:ExactEminSech} is based on the following lemma in \cite{bib:Sugihara_NearOpt_2003}, 
which concerns the transformation of results on the interval $[-1,1]$ 
from {\cite{bib:AnderssonQuadHp1980, bib:Ganelius1976}} to corresponding results on $\mathbf{R}$. 

\begin{lem}[{\cite{bib:AnderssonQuadHp1980, bib:Ganelius1976, bib:Sugihara_NearOpt_2003}}]
\label{lem:Ganelius}
For a positive integer $n$,
there exist $s_{1}', s_{2}', \ldots, s_{n}' \in \mathbf{R}$ such that
\begin{align}
\sup_{x \in \mathbf{R}} 
\left| 
\mathrm{sech}^{\beta}(x) 
\prod_{k=1}^{n} \tanh(x - s_{k}') 
\right| 
\leq
C_{\beta}'\, 
\exp \left( 
-\frac{\pi}{2} \sqrt{\beta\, n}
\right), 
\label{eq:Ganelius}
\end{align} 
where $C_{\beta}'$ is a positive constant depending on $\beta$. 
\end{lem}
\noindent
\cite{bib:Ganelius1976} presents a sequence $\{ s_{k}' \}$ that satisfies inequality~\eqref{eq:Ganelius}. 
However, this is not suitable as a set of sampling points for approximating functions 
because some of the members of $\{ s_{k}' \}$ coincide. 
\cite{bib:JangHaber2001} modify $\{ s_{k}' \}$ 
so that its members are mutually distinct and suited for such approximations. 
In the following, we describe the construction of the modified sequence. 
First, suppose that $n=2N$ is a positive even integer and define $\{ u_{k} \} \subset (0,1)$ by 
\begin{align}
u_{k} & =
\begin{cases}
\varphi(k-1)/\varphi(N_{o}) & (1 \leq k \leq N_{o}), \\
\varphi(k-3/2)/\varphi(N_{o}) & (k = N_{o} + 1), \\
1 - \frac{k-N_{o} - 1}{5(N - N_{o} - 1)} & (N_{o} + 2 \leq k \leq N), 
\end{cases} 
\label{eq:Ganelius_sample} 
\end{align}
where 
$\varphi(x) = \exp(\pi \sqrt{2x/\beta})$ and 
$N_{o} = N - \left \lceil \frac{\pi}{4} \sqrt{N \beta/2} \right \rceil$. 
Next, define $\{ t_{k} \} \subset (-1,1)$ by
\begin{align}
t_{k} = \sqrt{\frac{1-u_{k}}{1+u_{k}}}, \quad t_{-k} = -t_{k} \quad (k=1,2,\ldots, N). 
\end{align}
Finally, define $\{ s_{k} \}$ by
\begin{align}
s_{k} = \mathop{\mathrm{arctanh}} \left( t_{k} \right) \quad (k=\pm 1, \pm 2, \ldots , \pm N).
\end{align}
Then, it follows that 
\begin{align}
\sup_{x \in \mathbf{R}} 
\left| 
\mathrm{sech}^{\beta}(x) 
\prod_{\begin{subarray}{c} k=-N \\ k \neq 0 \end{subarray}}^{N} \tanh(x - s_{k}) 
\right| 
\leq
C_{\beta}'\, 
\exp \left( 
- \pi \sqrt{\frac{\beta\, N}{2}}
\right). 
\label{eq:Ganelius_explicit_sample}
\end{align} 
By using the sequence $\{ s_{k} \}$ and the functions
\begin{align}
T_{d}(x) &= \tanh \left( \frac{\pi}{4d} x \right), \label{eq:tanh_Gan} \\
B_{N}(x; \{s_{\ell}\}, \mathcal{D}_{d}) & = \prod_{k=-N}^{N} \tanh \left( \frac{\pi}{4d} (x - s_{k}) \right), \label{eq:B_Gan} \\
B_{N;k}(x; \{s_{\ell}\}, \mathcal{D}_{d}) & = \prod_{\begin{subarray}{c} -N\leq m \leq N,\\ m \neq k \end{subarray} } 
\tanh \left( \frac{\pi}{4d} (x - s_{m}) \right), \label{eq:B_k_Gan} 
\end{align}
we can obtain the optimal approximation formula 
in $\boldsymbol{H}^{\infty}(\mathcal{D}_{\pi/4}, \mathrm{sech}^{\beta}(z))$ as 
\begin{align}
f(x)
& \approx
\sum_{\begin{subarray}{c} k=-N \\ k \neq 0 \end{subarray}}^{N} f(s_{k})
\frac{ B_{N;k}(x; \{s_{\ell}\}, \mathcal{D}_{\pi/4})\, \mathrm{sech}^{\beta}(x) }
{B_{N;k}(s_{k}; \{s_{\ell}\}, \mathcal{D}_{\pi/4})\, \mathrm{sech}^{\beta}(s_{k})} T_{\pi/4}'(s_{k} - x) \notag \\
& = 
\sum_{\begin{subarray}{c} k=-N \\ k \neq 0 \end{subarray}}^{N} f(s_{k}) \, 
\frac{ 
\displaystyle
\mathrm{sech}^{\beta}(x) 
\prod_{\begin{subarray}{c} -N\leq m \leq N,\\ m \neq k \end{subarray} } 
\tanh \left( x - s_{m} \right)}
{\displaystyle
\mathrm{sech}^{\beta}(s_{k})
\prod_{\begin{subarray}{c} -N\leq m \leq N,\\ m \neq k \end{subarray} } 
\tanh \left( s_{k} - s_{m} \right)}
\, \mathrm{sech}^{2}(s_{k} - x).
\label{eq:Ganelius_formula}
\end{align}
%Formulas \eqref{eq:Ganelius_formula} and \eqref{eq:tanh_Gan}--\eqref{eq:B_k_Gan} are special forms of 
%the general formulas given by \eqref{eq:GeneralOptimalFomula} and \eqref{eq:tanh}--\eqref{eq:B_k} respectively, in Section~\ref{sec:ReducChar}.
In this paper, we call formula \eqref{eq:Ganelius_formula} Ganelius's formula. 

\subsection{Fundamentals in potential theory}
\label{sec:prelim_pot_th}

In reference to the study 
%\cite{bib:SaffTotik_LogPotExtField_1997}, 
\cite{bib:SaffTotik_LogPotExtField_1997}, 
we now describe some fundamental facts relating to potential theory on the complex plane $\mathbf{C}$. 

First, we present some facts concerning logarithmic potentials on $\mathbf{C}$. 
Let $\varSigma \subset \mathbf{C}$ be a compact subset of the complex plane,
and let $\mathcal{M}(\varSigma)$ be the collection of all positive unit Borel measures with support in $\varSigma$. 
The logarithmic potential $U^{\mu}(x)$ for $\mu \in \mathcal{M}(\varSigma)$ is defined by 
\begin{align}
U^{\mu}(x) :=
- \int_{\varSigma} \log |x-z|\, \d \mu(z), 
\label{eq:LogPot}
\end{align}
and the logarithmic energy $I(\mu)$ of $\mu \in \mathcal{M}(\varSigma)$ is defined by
\begin{align}
I(\mu) 
:= \int_{\varSigma} U^{\mu}(x)\, \d \mu(x)
= - \int_{\varSigma} \int_{\varSigma} \log |x-z|\, \d \mu(z) \d \mu(x).  
\end{align}
The energy $W$ of $\varSigma$ is defined by 
\begin{align}
W := \inf_{\mu \in \mathcal{M}(\varSigma)} I(\mu),
\label{eq:energy}
\end{align}
which is either finite or $+\infty$. In the finite case, there is a unique measure $\mu = \mu_{\varSigma}$
that attains the infimum in \eqref{eq:energy}. 
Then, the measure $\mu_{\varSigma}$ is called the equilibrium measure of $\varSigma$. 
Further, the quantity
\begin{align}
\mathop{\mathrm{cap}} (\varSigma) := \exp(-W)
\end{align}
is called the logarithmic capacity of $\varSigma$, and the capacity of an arbitrary Borel set $S$ is defined by
\begin{align}
\mathop{\mathrm{cap}} (S) := \sup \{ \mathop{\mathrm{cap}} (\varSigma) \mid \varSigma \subset S, \ \varSigma \text{ is compact} \}.
\end{align}
A property is said to hold quasi everywhere (q.e.) on a set $S$ if the set of exceptional points 
where the property does not hold is of capacity zero. 

Next, we describe some facts about Green potentials on a region in $\mathbf{C}$. 
Let $G \subset \mathbf{C}$ be a region and let $E \subset G$ be a closed set. 
Moreover, 
let $M$ be a positive real constant, and 
let $\mathcal{M}(E, M)$ be the collection of all positive Borel measures $\mu$ on $E$ with $\mu(E) = M$. 
If the region $G$ has the Green function $g_{G}(x,z)$, then
a Green potential $U_{G}^{\mu}(x)$ with respect to the measure $\mu \in \mathcal{M}(E, M)$ can be defined by 
\begin{align}
U_{G}^{\mu}(x) :=
\int_{E} g_{G}(x,z)\, \d \mu(z). 
\end{align}
Furthermore, if a function $w: E \to \mathbf{C}$ has some appropriate properties, then
the functions $w$ and $Q(x) = -\log w(x)$ can be regarded as a weight and an external field on $E$, 
respectively, 
and the total energy $I^{G}_{w}(\mu)$ is given by 
\begin{align}
I^{G}_{w}(\mu) 
& :=
\int_{E} \left( U_{G}^{\mu}(x) + 2 Q(x) \right) \d \mu(x) \notag \\
& =
\int_{E} \int_{E} g_{G}(x,z)\, \d \mu(z) \d \mu(x) - 2 \int_{E} \log w(x) \, \d \mu(x) \label{eq:Green_energy_pre} \\
& = 
\int_{E} \int_{E} g_{G}(x,z)\, \d \mu(z) \d \mu(x) - \int_{E} \log w(x) \, \d \mu(x) - \int_{E} \log w(z) \, \d \mu(z) \notag \\
& = 
\int_{E} \int_{E} g_{G}(x,z)\, \d \mu(z) \d \mu(x) - 
\frac{1}{M} \left( \int_{E} \int_{E} \log w(x) \, \d \mu(x) \d \mu(z) + \int_{E} \int_{E} \log w(z) \, \d \mu(z) \d \mu(x) \right) \notag \\
& =
\int_{E} \int_{E} \left( g_{G}(x,z) - \log(w(x)w(z))^{1/M} \right) \d \mu(z) \d \mu(x). 
\label{eq:Green_energy}
\end{align}
Furthermore, the first term in \eqref{eq:Green_energy_pre} is called the Green energy. 
%For example, 
%if $G$ is the unit disk, i.e., $G=\{ z \in \mathbf{C} \mid |z| < 1 \}$, then
%its Green function $g_{G}$ is given by $g_{G}(x,z) = -\log |x-z|$ and 
%its Green potential $U_{G}^{\mu}(x)$ is the logarithmic potential on $G$. 
If a measure $\mu \in \mathcal{M}(E,M)$ minimizes the Green energy with the external field, 
which means that the infimum
\begin{align}
\inf_{\mu \in \mathcal{M}(E, M)} I^{G}_{w}(\mu)
\end{align}
is finite and attained by the measure $\mu$, then it is called an equilibrium measure. 

\begin{rem}
\label{rem:total_meas}
In 
%\cite{bib:SaffTotik_LogPotExtField_1997}, 
\cite{bib:SaffTotik_LogPotExtField_1997}, 
only the case $M = 1$ is considered. 
However, in this paper we consider general constants $M > 0$ 
to fix the number of the sampling points used in the formula proposed in Sections \ref{sec:Reduction} and \ref{sec:Proc}.
For further details relating to this, 
see condition \eqref{eq:all_meas_b}.  
\end{rem}

In this paper, we consider the special case
$G = \mathcal{D}_{d}$ and $E = \mathbf{R}$. 
Furthermore, 
we assume that the weight function $w: \mathbf{R} \to \mathbf{C}$ satisfies 
Assumptions \ref{assump:w}--\ref{assump:w_convex}.
Note that $w$ satisfies
\begin{align}
\lim_{\begin{subarray}{c} x \to \pm \infty \\ x \in \mathbf{R} \end{subarray}} w(x) = 0
\label{eq:w_infty}
\end{align}
due to Assumption \ref{assump:w}. 
In this case, 
we have $g_{\mathcal{D}_{d}}(x,z) = -\log|\tanh((\pi/(4d)) (x-z))|$ for $z \in \mathbf{R}$,  
\begin{align}
U_{\mathcal{D}_{d}}^{\mu}(x) = 
-\int_{\mathbf{R}} \log \left| \tanh \left( \frac{\pi}{4d} (x-z) \right) \right| \d \mu(z), 
\label{eq:GreenPotDd}
\end{align}
and 
\begin{align}
I^{\mathcal{D}_{d}}_{w}(\mu) 
& = 
\int_{\mathbf{R}} \int_{\mathbf{R}} \left( g_{\mathcal{D}_{d}}(x,z) - \log (w(x) w(z))^{1/M} \right) \d \mu(z) \d \mu(x) 
\label{eq:WeightedGreenEnergyDd_pre} \\
& = 
\int_{\mathbf{R}} \int_{\mathbf{R}} \left( 
- \log \left| \tanh \left( \frac{\pi}{4d} (x-z) \right) \right| 
- \log (w(x) w(z))^{1/M}
\right)
\d \mu(z) \d \mu(x).
\label{eq:WeightedGreenEnergyDd}
\end{align}

\begin{rem}
\label{rem:admissible}
In 
%\cite{bib:SaffTotik_LogPotExtField_1997},
\cite{bib:SaffTotik_LogPotExtField_1997},
a weight function $w$ is called admissible 
if (i) $w$ is upper semi-continuous on $E$, 
(ii) the set $\{ x \in E \mid w(x) > 0 \}$ has a positive capacity, and 
(iii) $\lim_{x \to \pm \infty} |x|\, w(x) = 0$ in the case that $E$ is unbounded. 
These conditions are used to prove the existence and uniqueness of an equilibrium measure with compact support 
in the case of logarithmic potentials \eqref{eq:LogPot}. 
However, 
in the case that the above setting $G = \mathcal{D}_{d}$ and $E = \mathbf{R}$ is assumed, 
 we do not assume the admissibility of $w$ 
because conditions (i) and (ii) 
are fulfilled on the basis of Assumption \ref{assump:w}, 
and condition (iii) can be substituted for the weaker condition \eqref{eq:w_infty}
as will be demonstrated in the proof of Theorem \ref{thm:UniqueEquiMeas}.
\end{rem}

The following theorem,
which is a slight modification of 
%{\cite[Theorem II.5.10]{bib:SaffTotik_LogPotExtField_1997}},
{Theorem II.5.10 in \cite{bib:SaffTotik_LogPotExtField_1997}},
shows that a unique equilibrium measure exists given certain assumptions. 

\begin{thm}
\label{thm:UniqueEquiMeas}
Let $w$ be a weight function satisfying Assumptions \ref{assump:w}--\ref{assump:w_convex}, 
let $M$ be a positive real constant, and set 
\begin{align}
V_{w}^{\mathcal{D}_{d}} := 
\inf_{\mu \in \mathcal{M}(\mathbf{R}, M)} I_{w}^{\mathcal{D}_{d}}(\mu), 
\end{align}
where $\mathcal{M}(\mathbf{R}, M)$ is the set of Borel measures on $\mathbf{R}$ 
with total measure $M$. Then the following hold. 
\begin{enumerate}
\item $V_{w}^{\mathcal{D}_{d}}$ is finite. 
\item There is a unique measure $\mu_{w}^{\mathcal{D}_{d}} \in \mathcal{M}(\mathbf{R}, M)$ such that
\begin{align}
I_{w}^{\mathcal{D}_{d}}(\mu_{w}^{\mathcal{D}_{d}}) = V_{w}^{\mathcal{D}_{d}}. 
\label{eq:GE_min_Dd}
\end{align}
Moreover, $\mu_{w}^{\mathcal{D}_{d}}$ has compact support and finite Green energy. 
\end{enumerate}
\end{thm}

\begin{proof}
It follows from \eqref{eq:WeightedGreenEnergyDd_pre} that
\begin{align}
& 
\inf_{\mu \in \mathcal{M}(\mathbf{R}, M)} I_{w}^{\mathcal{D}_{d}}(\mu) \notag \\
& = 
\inf_{\mu \in \mathcal{M}(\mathbf{R}, M)} 
\left[
M^{2}
\int_{\mathbf{R}} \int_{\mathbf{R}} \frac{g_{\mathcal{D}_{d}}(x,z) - \log(w(x)w(z))^{1/M}}{M^{2}} \, \d \mu(z) \d \mu(x) 
\right] \notag \\
& =
M^{2}
\inf_{\nu \in \mathcal{M}(\mathbf{R}, 1)} 
\left[
\int_{\mathbf{R}} \int_{\mathbf{R}} \left( g_{\mathcal{D}_{d}}(x,z) - \log \left( w(x)^{1/M} w(z)^{1/M} \right) \right) \d \nu(z) \d \nu(x) 
\right] \notag \\
& = 
M^{2}
\inf_{\nu \in \mathcal{M}(\mathbf{R}, 1)} 
I_{w^{1/M}}^{\mathcal{D}_{d}}(\nu). 
\label{eq:GE_min_Dd_prob_meas}
\end{align}
Then, we only have to consider the minimization problem of $I_{w^{1/M}}^{\mathcal{D}_{d}}(\nu)$
over the probability measures $\nu$ on $\mathbf{R}$. 
The unique existence of a solution to this problem is guaranteed by 
%{\cite[Theorem II.5.10]{bib:SaffTotik_LogPotExtField_1997}}
{Theorem II.5.10 in \cite{bib:SaffTotik_LogPotExtField_1997}}, 
provided that the weight $w^{1/M}$ is admissible in the sense explained in Remark \ref{rem:admissible}.
However, we can only use condition \eqref{eq:w_infty} 
instead of condition (iii) in Remark \ref{rem:admissible}. 

In fact, in 
%{\cite[Theorem II.5.10]{bib:SaffTotik_LogPotExtField_1997}}, 
{Theorem II.5.10 in \cite{bib:SaffTotik_LogPotExtField_1997}}, 
condition (iii) is necessary only to show that 
\begin{align}
g_{G}(x,z) - \log (w(x) w(z)) > C
\label{eq:bnd_gGww}
\end{align}
for a certain constant $C > -\infty$, and
\begin{align}
\lim_{n \to \infty} 
\left[
g_{G}(x_{n},z_{n}) - \log (w(x_{n}) w(z_{n})) 
\right] = \infty
\label{eq:lim_gGww}
\end{align}
if $\{ (x_{n}, z_{n}) \} \subset E \times E$ is a sequence with
\begin{align}
\lim_{n \to \infty} \min \{ w(x_{n}), w(z_{n}) \} \to 0.
\end{align}
These conditions guarantee 
the finiteness of $V_{w}^{G}$,  
the existence of the equilibrium measure $\mu_{w}^{G}$, and 
the compactness of the support of $\mu_{w}^{G}$. 

Therefore, to prove Theorem \ref{thm:UniqueEquiMeas}, 
we only have to show that the weight function $w^{1/M}$ satisfies \eqref{eq:bnd_gGww} and \eqref{eq:lim_gGww}
in the case that $G = \mathcal{D}_{d}$ and $E = \mathbf{R}$, using condition \eqref{eq:w_infty}. 
Because we have that 
\begin{align}
g_{\mathcal{D}_{d}}(x,z) = - \log \left| \tanh \left( \frac{\pi}{4d} (x-z) \right) \right| \geq 0
\end{align}
for any $x,z \in \mathbf{R}$, it suffices to consider the function $\log(w(x)w(z))^{-1/M}$. 
First, it follows from Assumption \ref{assump:w} that $w$ is bounded above on $\mathbf{R}$, 
hence $\log(w(x)w(z))^{-1/M}$ is bounded from below. 
Therefore, we have that \eqref{eq:bnd_gGww} holds for $w^{1/M}$. 
Next, if $\{ (x_{n}, z_{n}) \} \subset \mathbf{R} \times \mathbf{R}$ is a sequence with
\begin{align}
\lim_{n \to \infty} \min\{ w(x_{n})^{1/M}, w(z_{n})^{1/M} \} \to 0, 
\end{align}
then we have that $x_{n} \to \pm \infty$ or $z_{n} \to \pm \infty$ by Assumption \ref{assump:w} and \eqref{eq:w_infty}. 
Then, we have that
\begin{align}
\lim_{n \to \infty} \log (w(x_{n}) w(z_{n}))^{-1/M} \to \infty, 
\label{eq:lim_gGww_Dd}
\end{align}
which implies \eqref{eq:lim_gGww} for $w^{1/M}$. 
Thus, we reduce Theorem \ref{thm:UniqueEquiMeas} to 
%{\cite[Theorem II.5.10]{bib:SaffTotik_LogPotExtField_1997}}
{Theorem II.5.10 in \cite{bib:SaffTotik_LogPotExtField_1997}}, 
and the proof is concluded. 
\end{proof}

%------------------------------
\section{A basic idea for designing accurate formulas based on potential theory}
\label{sec:Reduction}

Let $N$ be a positive integer. 
As mentioned in Section \ref{sec:DefOptApprox}, 
we can restrict ourselves without loss of generality 
to sequences of sampling points that are symmetric with respect to the origin. 

%--------------------
\subsection{Reduction of the characterization problem to a problem with a continuous variable}
\label{sec:ReducChar}

We begin with the characterization of 
$E_{N}^{\mathrm{min}}(\boldsymbol{H}^{\infty}(\mathcal{D}_{d}, w))$ by explicit formulas. 
These are given by the following proposition, 
in which we restrict 
%{\cite[Lemma 4.3]{bib:Sugihara_NearOpt_2003}} 
{Lemma 4.3 in \cite{bib:Sugihara_NearOpt_2003}} 
to the case of even weight functions and the sampling points stated above. 
For readers' convenience, 
we describe the sketch of the proof of this proposition 
in Section \ref{sec:proof_opt_err_norm}. 

\begin{prop}[{Lemma 4.3 in \cite{bib:Sugihara_NearOpt_2003}}]
\label{prop:Sugihara2003}
\begin{align}
& E_{N}^{\mathrm{min}}(\boldsymbol{H}^{\infty}(\mathcal{D}_{d}, w)) \notag \\
& =
\inf_{a_{\ell} \in \mathbf{R}}
\left[
\sup_{\| f \| \leq 1}
\sup_{x \in \mathbf{R}} 
\left|
f(x) - \hspace{-1mm} \sum_{k = -N}^{N} \hspace{-1mm} f(a_{k}) 
\frac{ B_{N;k}(x; \{a_{\ell}\}, \mathcal{D}_{d}) w(x) }{B_{N;k}(a_{k}; \{a_{\ell}\}, \mathcal{D}_{d}) w(a_{k})} \frac{4d}{\pi} T_{d}'(a_{k} - x)
\right|
\right]
\notag \\
& = \inf_{a_{\ell} \in \mathbf{R}} 
\left[
\sup_{x \in \mathbf{R}} 
\left|
B_{N}(x; \{a_{\ell}\} , \mathcal{D}_{d})\, w(x)
\right|
\right], 
\label{eq:MinErrBlaschke} 
\end{align}
where 
$T_{d}(x)$, 
$B_{N}(x; \{a_{\ell}\}, \mathcal{D}_{d})$, and
$B_{N;k}(x; \{a_{\ell}\}, \mathcal{D}_{d})$
are defined in \eqref{eq:tanh_Gan}, \eqref{eq:B_Gan}, \eqref{eq:B_k_Gan}, respectively. 
%\(
%T(x) &= \tanh \left( \frac{\pi}{4d} x \right), \label{eq:tanh} \\
%B_{N}(x; \{a_{\ell}\}, \mathcal{D}_{d}) & = \prod_{k=-N}^{N} \tanh \left( \frac{\pi}{4d} (x - a_{k}) \right), \label{eq:B} \\
%B_{N;k}(x; \{a_{\ell}\}, \mathcal{D}_{d}) & = \prod_{\begin{subarray}{c} -N\leq m \leq N,\\ m \neq k \end{subarray} } 
%\tanh \left( \frac{\pi}{4d} (x - a_{m}) \right). \label{eq:B_k}
%\)
\end{prop}

\noindent
The first equality in \eqref{eq:MinErrBlaschke} gives the explicit form of the basis functions\footnote{
In 
%\cite[Lemma 4.3]{bib:Sugihara_NearOpt_2003}, 
Lemma 4.3 in \cite{bib:Sugihara_NearOpt_2003}, 
the form of the basis functions is wrong. 
Here, we have corrected the form by inserting the factor $4d/\pi$ in front of $T'$. 
}. 
That is, if we can obtain sampling points $a_{\ell} \in \mathbf{R}$ that attain the infimum in \eqref{eq:MinErrBlaschke}, 
then the interpolation formula
\begin{align}
\tilde{f}_{N}(x) 
& =
\sum_{k = -N}^{N} f(a_{k})
\frac{ B_{N;k}(x; \{a_{\ell}\}, \mathcal{D}_{d}) w(x) }{B_{N;k}(a_{k}; \{a_{\ell}\}, \mathcal{D}_{d}) w(a_{k})} \frac{4d}{\pi} T_{d}'(a_{k} - x) \notag \\
& =
\sum_{k = -N}^{N} f(a_{k})\,
\frac{ w(x) 
\displaystyle
\prod_{\begin{subarray}{c} -N\leq m \leq N,\\ m \neq k \end{subarray} } 
\tanh \left( \frac{\pi}{4d} (x - a_{m}) \right)}
{\displaystyle
w(a_{k})
\prod_{\begin{subarray}{c} -N\leq m \leq N,\\ m \neq k \end{subarray} } 
\tanh \left( \frac{\pi}{4d} (a_{k} - a_{m}) \right)} 
\, \mathrm{sech}^{2} \left( \frac{\pi}{4d} (a_{k} - x) \right) 
\label{eq:GeneralOptimalFomula}
\end{align}
gives an optimal approximation in $\boldsymbol{H}^{\infty}(\mathcal{D}_{d}, w)$. 
Then, what remains is to determine the sampling points $a_{\ell} \in \mathbf{R}$. 
One criterion for determining these is given by the second equality in \eqref{eq:MinErrBlaschke}. 
Therefore, in the remainder of this paper, we will focus on the optimization problem
\begin{align}
\label{eq:main_opt}
\inf_{a_{\ell} \in \mathbf{R}} 
\left[
\sup_{x \in \mathbf{R}} 
\left|
B_{N}(x; \{a_{\ell}\} , \mathcal{D}_{d})\, w(x)
\right|
\right].
\end{align}
Because the logarithm is a monotonically increasing function on $(0,\infty)$, 
we consider the following problem which is equivalent to \eqref{eq:main_opt}.
\begin{prob}
\label{prob:Original}
Find a sequence $\{ a_{\ell} \}$ of sampling points that attains 
\begin{align}
\inf_{a_{\ell} \in \mathbf{R}} 
\left[
\sup_{x \in \mathbf{R}} 
\left(
V_{\mathcal{D}_{d}}^{\sigma_{a}}(x) + \log w(x) 
\right)
\right], 
\label{eq:main_opt_equiv_log} 
\end{align}
where
\begin{align}
V_{\mathcal{D}_{d}}^{\sigma_{a}}(x) & = \sum_{k = -N}^{N} \log \left| \tanh \left( \frac{\pi}{4d} (x - a_{k}) \right) \right|. \label{eq:minus_G_pot}
\end{align}
\end{prob}
\noindent
The function $U_{\mathcal{D}_{d}}^{\sigma_{a}} = -V_{\mathcal{D}_{d}}^{\sigma_{a}}$ is the Green potential on $\mathcal{D}_{d}$ 
given by the discrete measure $\sigma_{a}$ associated with the sampling points $\{a_{\ell} \}$. That is,
\begin{align}
\sigma_{a}(A) = \sum_{\ell:\, a_{\ell} \in A} 1, \qquad A \subset \mathbf{R}.
\label{eq:disc_meas_a}
\end{align}
In fact, the Green function of the region $\mathcal{D}_{d}$ is given by
\begin{align}
g_{\mathcal{D}_{d}}(x,z) = - \log \left| \tanh \left( \frac{\pi}{4d} (x - z) \right) \right|
\end{align}
if $z \in \mathbf{R}$. 
By applying the measure $\sigma_{a}$ from \eqref{eq:disc_meas_a}, we can rewrite 
$V_{\mathcal{D}_{d}}^{\sigma_{a}}(x)$ in \eqref{eq:minus_G_pot} as
\begin{align}
V_{\mathcal{D}_{d}}^{\sigma_{a}}(x)
=
\int_{-a_{N}}^{a_{N}} \log \left| \tanh \left( \frac{\pi}{4d} (x - z) \right) \right| \, \mathrm{d}\sigma_{a}(z). 
\label{eq:minus_G_pot_rewrited}
\end{align}

For the sake of analytical tractability, 
we replace the discrete measure $\sigma_{a}$ in \eqref{eq:minus_G_pot_rewrited} with a general measure $\mu_{N}$ on $\mathbf{R}$
to consider the approximation of $V_{\mathcal{D}_{d}}^{\sigma_{a}}(x)$ given by
\begin{align}
V_{\mathcal{D}_{d}}^{\mu_{N}}(x)
=
\int_{-\alpha_{N}}^{\alpha_{N}} \log \left| \tanh \left( \frac{\pi}{4d} (x - z) \right) \right| \, \mathrm{d}\mu_{N}(z). 
\label{eq:minus_G_pot_approx}
\end{align}
The real number $\alpha_{N}$ in \eqref{eq:minus_G_pot_approx}
is an unknown that determines the support of the measure\footnote{
The support of a measure $\mu$ is defined by the collection of all points $z$ satisfying
\(
\mu(\{ y \mid |y - z| < r \}) \neq 0
\)
for all $r > 0$.
} 
$\mu_{N}$ as $\mathop{\mathrm{supp}} \mu_{N} = [-\alpha_{N}, \alpha_{N}]$. 
In order to impose a condition concerning the number of sampling points on $\mu_{N}$, 
we assume that 
\begin{align}
\mu_{N}([-\alpha_{N}, \alpha_{N}]) = \int_{-\alpha_{N}}^{\alpha_{N}} \mathrm{d}\mu_{N}(z) = 2(N+1). 
\label{eq:all_meas_b}
\end{align}
This value is different from $2N+1$ because this will provide a technical advantage
when we estimate the difference between $V_{\mathcal{D}_{d}}^{\mu_{N}}$ and 
$V_{\mathcal{D}_{d}}^{\sigma_{a}}$ in Lemma \ref{lem:GrPotDiscError_mod}. 
Now, we consider the following problem as an approximation of Problem \ref{prob:Original}. 
\begin{prob}
\label{prob:ContMeas}
Find a positive real number $\alpha_{N}$ and 
a measure $\mu_{N} \in \mathcal{M}(\mathbf{R}, 2(N+1))$ that attain
\begin{align}
\inf_{
\begin{subarray}{c}
\mu_{N} \in \mathcal{M}(\mathbf{R}, 2(N+1)) \\
\mathop{\mathrm{supp}} \mu_{N} = [-\alpha_{N}, \alpha_{N}] 
\end{subarray}
}
\left[
\sup_{x \in \mathbf{R}}
\left(
V_{\mathcal{D}_{d}}^{\mu_{N}}(x) + \log w(x)
\right)
\right]. 
\label{eq:main_opt_approx} 
\end{align}
\end{prob}
\noindent
If Problem~\ref{prob:ContMeas} has a solution $\mu_{N}$ with a density function $\nu_{N}$ 
such that $\d \mu_{N}(x) = \nu_{N}(x)\, \d x$, 
then we generate sampling points $a_{\ell}$ as
\begin{align}
a_{\ell} = (I[\nu_{N}])^{-1}(\ell)\quad (\ell = -N, -N+1, \ldots, N-1, N), 
\label{eq:gen_sampling_points}
\end{align}
where 
\begin{align}
I[\nu_{N}](x) = \int_{0}^{x}  \nu_{N}(t)\, \d t.
\notag
\end{align}
We expect that the sequence $\{ a_{\ell} \}$ will provide a good approximate solution of Problem~\ref{prob:Original}.

%--------------------
\subsection{Characterization of solutions of the problem \eqref{eq:main_opt_approx} using potential theory}

In this section, we characterize solutions $\mu_{N}$ of the problem \eqref{eq:main_opt_approx} 
using some fundamental facts relating to potential theory. 
The characterization of $\mu_{N}$ is provided by the following theorem. 

\begin{thm}
\label{thm:char_opt_nu_N}
Let $w$ be a weight function satisfying Assumptions \ref{assump:w}--\ref{assump:w_convex}. 
Then, the value $\alpha_{N}$ and the measure $\mu_{N} \in \mathcal{M}(\mathbf{R}, 2(N+1))$ are the solutions of the problem \eqref{eq:main_opt_approx} 
if $\mathop{\mathrm{supp}} \mu_{N} = [-\alpha_{N}, \alpha_{N}]$ and 
there exists a constant $K_{N}$ depending on $N$ such that
\begin{align}
& \label{eq:char_opt_eq} V_{\mathcal{D}_{d}}^{\mu_{N}}(x) + \log w(x) = -K_{N} \quad \text{for q.e. }x \in [-\alpha_{N}, \alpha_{N}], \\
& \label{eq:char_opt_ineq} V_{\mathcal{D}_{d}}^{\mu_{N}}(x) + \log w(x) \leq -K_{N} \quad \text{for any }x \in \mathbf{R} \setminus [-\alpha_{N}, \alpha_{N}]. 
\end{align}
\end{thm}

In order to prove this theorem,  we apply some fundamental facts from potential theory that are described below. 
The study
%{\cite{bib:SaffTotik_LogPotExtField_1997}}, 
{\cite{bib:SaffTotik_LogPotExtField_1997}} presents 
the original forms of these facts in a more general setting, i.e., 
in terms of more general $G$, $E$ and $w$ introduced in Section \ref{sec:prelim_pot_th}. 
However, for simplicity
we concern ourselves only with the case that $G=\mathcal{D}_{d}, E = \mathbf{R}$ 
and $w$ satisfy Assumptions \ref{assump:w}--\ref{assump:w_convex}. 
Furthermore, 
we only use $V_{\mathcal{D}_{d}}^{\mu_{N}}$, 
rather than
$U_{\mathcal{D}_{d}}^{\mu_{N}}\ (=-V_{\mathcal{D}_{d}}^{\mu_{N}})$. 

First, we consider the integral
\begin{align}
J_{w}^{\mathcal{D}_{d}}(\mu_{N}) 
& = \int_{-\alpha_{N}}^{\alpha_{N}} (V_{\mathcal{D}_{d}}^{\mu_{N}}(x) + 2 \log w(x))\, \mathrm{d}\mu_{N}(x) \notag \\
& = \int_{-\alpha_{N}}^{\alpha_{N}} \int_{-\alpha_{N}}^{\alpha_{N}} \log \left| \tanh \left( \frac{\pi}{4d} (x - z) \right) \right|\, \mathrm{d}\mu_{N}(z) \mathrm{d}\mu_{N}(x)
+ 2 \int_{-\alpha_{N}}^{\alpha_{N}} \log w(x)\, \mathrm{d}\mu_{N}(x). 
\end{align} 
Note that $J_{w}^{\mathcal{D}_{d}}(\mu_{N}) = -I_{w}^{\mathcal{D}_{d}}(\mu_{N})$, 
where $I_{w}^{\mathcal{D}_{d}}(\mu_{N})$ is the weighted Green energy given in \eqref{eq:Green_energy}. 
Then, according to Theorem \ref{thm:UniqueEquiMeas}, 
there is a unique maximizer $\mu_{N}^{\ast} \in \mathcal{M}(\mathbf{R}, 2(N+1))$ of $J_{w}^{\mathcal{D}_{d}}$. 
Supposing $\mathop{\mathrm{supp}} \mu_{N}^{\ast} \in [-\alpha_{N}^{\ast}, \alpha_{N}^{\ast}]$, 
we define a constant $K_{N}^{\ast}$ depending on $N$ by
\begin{align}
K_{N}^{\ast} = J_{w}^{\mathcal{D}_{d}}(\mu_{N}^{\ast}) - \int_{-\alpha_{N}^{\ast}}^{\alpha_{N}^{\ast}} \log w(x)\, \mathrm{d}x. 
\end{align}
Then, we obtain the following characterization of the optimal measure $\mu_{N}^{\ast}$. 

\begin{prop}
\label{prop:optimal_meas}
Let $w$ be a weight function satisfying Assumptions \ref{assump:w}--\ref{assump:w_convex}. 
Furthermore, suppose that $\mu_{N} \in \mathcal{M}(\mathbf{R}, 2(N+1))$ has compact support and a finite Green energy. 
If there exists a constant $K_{N}$ such that 
\begin{align}
&V_{\mathcal{D}_{d}}^{\mu_{N}}(x) + \log w(x) = -K_{N} \quad \text{for q.e. } x \in \mathop{\mathrm{supp}} \mu_{N}, \label{eq:opt_meas_1} \\
&V_{\mathcal{D}_{d}}^{\mu_{N}}(x) + \log w(x) \leq -K_{N} \quad \text{for any } x \in \mathbf{R}, \label{eq:opt_meas_2}
\end{align}
then we have that $\mu_{N} = \mu_{N}^{\ast}$ and $K_{N} = K_{N}^{\ast}$. 
\end{prop}

\begin{proof}
By dividing both sides of \eqref{eq:opt_meas_1} and \eqref{eq:opt_meas_2} by $M=2(N+1)$, 
we can reduce this theorem to 
%{\cite[Theorem II.5.12]{bib:SaffTotik_LogPotExtField_1997}} 
{Theorem II.5.12 in \cite{bib:SaffTotik_LogPotExtField_1997}}, 
provided that the weight $w^{1/M}$ is admissible. 
However, because the admissibility is only necessary for the same purpose as in 
Theorem \ref{thm:UniqueEquiMeas}, 
Assumptions \ref{assump:w}--\ref{assump:w_convex} on $w$ are sufficient for proving this theorem
as shown in the proof of Theorem \ref{thm:UniqueEquiMeas}. 
Thus we can obtain the conclusion by referring to 
%{\cite[Theorem II.5.12]{bib:SaffTotik_LogPotExtField_1997}}. 
{Theorem II.5.12 in \cite{bib:SaffTotik_LogPotExtField_1997}}. 
\end{proof}

\vskip5pt

\noindent
In addition to Proposition \ref{prop:optimal_meas}, 
we require the following proposition to demonstrate that the maximizer $\mu_{N}^{\ast}$ of $J$ 
is also a solution of the optimization problem \eqref{eq:main_opt_approx}. 
Following \cite{bib:SaffTotik_LogPotExtField_1997}, 
for a real function $h$ on $\mathbf{R}$, let $``\sup_{x \in \mathbf{R}}\text{''}\, h(x)$ 
denote the smallest number $U$ such that $h$ takes values larger than $U$ only on a set of zero capacity.  

\begin{prop}
\label{prop:pot_weight_gen_meas}
Let $w$ be a weight function satisfying Assumptions \ref{assump:w}--\ref{assump:w_convex}. 
Then, for any $\mu_{N} \in \mathcal{M}(\mathbf{R}, 2(N+1))$ with compact support, we have
\begin{align}
``\sup_{x \in \mathbf{R}}\text{''} \left( V_{\mathcal{D}_{d}}^{\mu_{N}}(x) + \log w(x) \right) & \geq -K_{N}^{\ast}, \\
\inf_{x \in \mathop{\mathrm{supp}} \mu_{N}} \left( V_{\mathcal{D}_{d}}^{\mu_{N}}(x) + \log w(x) \right) & \leq -K_{N}^{\ast}.
\end{align}
\end{prop}

\begin{proof}
This theorem is an analogue of 
%{\cite[Theorem I.3.1]{bib:SaffTotik_LogPotExtField_1997}}, 
{Theorem I.3.1 in \cite{bib:SaffTotik_LogPotExtField_1997}}, 
which states an analogous fact for the case of logarithmic potentials. 
By dividing $\mu_{N} \in \mathcal{M}(\mathbf{R}, 2(N+1))$ by $M = 2(N+1)$ and 
replacing the admissibility assumption in 
%{\cite[Theorem I.3.1]{bib:SaffTotik_LogPotExtField_1997}} 
{Theorem I.3.1 in \cite{bib:SaffTotik_LogPotExtField_1997}} 
by Assumptions \ref{assump:w}--\ref{assump:w_convex}, 
we can prove this theorem in almost the same manner as 
%{\cite[Theorem I.3.1]{bib:SaffTotik_LogPotExtField_1997}}. 
{Theorem I.3.1 in \cite{bib:SaffTotik_LogPotExtField_1997}}. 
\end{proof}

\vskip5pt

\noindent
By combining Propositions \ref{prop:optimal_meas} and \ref{prop:pot_weight_gen_meas}, 
we can prove Theorem \ref{thm:char_opt_nu_N}. 

\vskip5pt

\noindent
\renewcommand{\proof}{{\it Proof of Theorem \ref{thm:char_opt_nu_N}}.}
\begin{proof}
Suppose that the conditions \eqref{eq:char_opt_eq} and \eqref{eq:char_opt_ineq} are satisfied. 
Then, it follows from Proposition \ref{prop:optimal_meas} that $\mu_{N} = \mu_{N}^{\ast}$, $K_{N} = K_{N}^{\ast}$, and
\begin{align}
\sup_{x \in \mathbf{R}} \left( V_{\mathcal{D}_{d}}^{\mu_{N}^{\ast}}(x) + \log w(x) \right) = -K_{N}^{\ast}. 
\end{align}
Furthermore, let $\mu_{N} \in \mathcal{M}(\mathbf{R}, 2(N+1))$ be a measure with compact support. 
According to Proposition \ref{prop:pot_weight_gen_meas}, we have that
\begin{align}
&\sup_{x \in \mathbf{R}} \left( V_{\mathcal{D}_{d}}^{\mu_{N}}(x) + \log w(x) \right) 
\geq 
``\sup_{x \in \mathbf{R}}\text{''} \left( V_{\mathcal{D}_{d}}^{\mu_{N}}(x) + \log w(x) \right) \notag \\
& \geq -K_{N}^{\ast} 
= \sup_{x \in \mathbf{R}} \left( V_{\mathcal{D}_{d}}^{\mu_{N}^{\ast}}(x) + \log w(x) \right). 
\end{align}
Then, $\mu_{N}^{\ast}$ is a solution of the optimization problem \eqref{eq:main_opt_approx}.
\end{proof}
\renewcommand{\proof}{{\it Proof}.} 
%\renewcommand{\proofname}{{\it Proof}.}

%------------------------------
\section{Procedure for designing accurate formulas}
\label{sec:Proc}

In order to generate sampling points $a_{\ell}$ using Theorem \ref{thm:char_opt_nu_N}, 
we need to obtain the solution $\mu_{N}^{\ast}$ or its approximation of the optimality condition 
given by the integral equation \eqref{eq:char_opt_eq} and the integral inequality \eqref{eq:char_opt_ineq} 
with unknown parameters $\alpha_{N}$ and $K_{N}$. 

In order to achieve analytical tractability, 
we seek an approximation of $\mu_{N}^{\ast}$ in 
the set of the measures in $\mathcal{M}(\mathbf{R}, 2(N+1))$
with continuously differentiable density functions. 
For this purpose, we define 
\begin{align}
& \mathcal{C}_{\mathrm{d}}^{1}(\mathbf{R}, 2(N+1)) \notag \\
& =
\left\{
\nu_{N}: \mathbf{R} \to \mathbf{R}_{\geq 0} 
\left| \, 
\nu_{N} \text{ is continuously differentiable on } \mathop{\mathrm{supp}} \nu_{N}, \ 
\int_{-\infty}^{\infty} \nu_{N}(x) \, \d x = 2(N+1)
\right.
\right\}
\label{eq:CondDiffDens}
\end{align}
and let $\mu[\nu_{N}]$ be the measure such that 
$\d \mu[\nu_{N}](x) = \nu_{N}(x)\, \d x$ for $\nu_{N} \in \mathcal{C}_{\mathrm{d}}^{1}(\mathbf{R}, 2(N+1))$. 
Then we seek an approximation of $\mu_{N}^{\ast}$ in the set 
$\{ \mu[\nu_{N}] \mid \nu_{N} \in \mathcal{C}_{\mathrm{d}}^{1}(\mathbf{R}, 2(N+1)) \}$, 
which is a subset of $\mathcal{M}(\mathbf{R}, 2(N+1))$.  
By using some fundamental properties of singular integrals, 
we can show that the following smoothness property of $V_{\mathcal{D}_{d}}^{\mu[\nu_{N}]}$ holds. 
The proof is presented in Section \ref{sec:proof_smooth_V}. 

\begin{prop}
\label{prop:smooth_V}
Let $\nu_{N} \in \mathcal{C}_{\mathrm{d}}^{1}(\mathbf{R}, 2(N+1))$ be a density function 
with $\mathop{\mathrm{supp}} \nu_{N} = [-\alpha_{N}, \alpha_{N}]$ and $\nu_{N}(\pm \alpha_{N}) = 0$. 
Then, the function $V_{\mathcal{D}_{d}}^{\mu[\nu_{N}]}$ given by \eqref{eq:minus_G_pot_approx} 
is differentiable on $\mathbf{R} \setminus \{ \pm \alpha_{N} \}$
and its derivative $(V_{\mathcal{D}_{d}}^{\mu[\nu_{N}]})'$ satisfies
\begin{align}
(V_{\mathcal{D}_{d}}^{\mu[\nu_{N}]})'(\pm \alpha_{N} - 0) = (V_{\mathcal{D}_{d}}^{\mu[\nu_{N}]})'(\pm \alpha_{N} + 0). 
\label{eq:smooth_V}
\end{align}
\end{prop}
\noindent
Therefore, 
in the remainder of this section 
we use condition \eqref{eq:char_opt_eq} from Theorem \ref{thm:char_opt_nu_N} 
replacing ``q.e.'' by ``for any''. 
Consequently, we consider the following problem. 

\begin{prob}
\label{prob:IE_IINEQ}
Find real numbers $\alpha_{N}$ and $K_{N}$, and a density function 
$\nu_{N} \in \mathcal{C}_{\mathrm{d}}^{1}(\mathbf{R}, 2(N+1))$ 
that satisfy $\mathop{\mathrm{supp}} \nu_{N} = [-\alpha_{N}, \alpha_{N}]$ and 
\begin{align}
& \label{eq:char_opt_eq_Prob} V_{\mathcal{D}_{d}}^{\mu[\nu_{N}]}(x) + \log w(x) = -K_{N} \quad \text{for any }x \in [-\alpha_{N}, \alpha_{N}], \\
& \label{eq:char_opt_ineq_Prob} V_{\mathcal{D}_{d}}^{\mu[\nu_{N}]}(x) + \log w(x) \leq -K_{N} \quad \text{for any }x \in \mathbf{R} \setminus [-\alpha_{N}, \alpha_{N}]. 
\end{align}
\end{prob}

Because the system in Problem~\ref{prob:IE_IINEQ} contains inequality \eqref{eq:char_opt_ineq_Prob}, 
it seems difficult to obtain the explicit form of $V_{\mathcal{D}_{d}}^{\mu[\nu_{N}]}$ outside of $[-\alpha_{N}, \alpha_{N}]$. 
However, we can in fact obtain it using the fact that the Green potential $U_{\mathcal{D}_{d}}^{\mu[\nu_{N}]}$ 
is harmonic on $\mathcal{D}_{d} \setminus \mathop{\mathrm{supp}} \nu_{N}$ and so is $V_{\mathcal{D}_{d}}^{\mu[\nu_{N}]}$. 
Then, we consider the following problem in order to obtain a solution for Problem~\ref{prob:IE_IINEQ}.

\begin{prob}
\label{prob:2Steps}
Find solutions to the following two subproblems. 
\begin{description}
\item[SP1]
Determine the explicit forms of 
$\alpha_{N}$, $K_{N}$ and 
$V_{\mathcal{D}_{d}}^{\mu[\nu_{N}]}$ on $\mathbf{R} \setminus [-\alpha_{N}, \alpha_{N}]$
under the constraints
$\nu_{N} \in \mathcal{C}_{\mathrm{d}}^{1}(\mathbf{R}, 2(N+1))$, 
$\mathop{\mathrm{supp}} \nu_{N} = [-\alpha_{N}, \alpha_{N}]$, 
\eqref{eq:char_opt_eq_Prob}, and \eqref{eq:char_opt_ineq_Prob}.

\item[SP2]
Let the solutions for $\alpha_{N}$, $K_{N}$, and $V_{\mathcal{D}_{d}}^{\mu[\nu_{N}]}$ of SP1 
be given by $\alpha_{N}^{\ast}$, $K_{N}^{\ast}$, and $v^{\ast}$, respectively. 
Then, solve the equation $V_{\mathcal{D}_{d}}^{\mu[\nu_{N}]} = v^{\ast}$ on $\mathbf{R}$, i.e.,
\begin{align}
\int_{-\alpha_{N}^{\ast}}^{\alpha_{N}^{\ast}} 
\log 
\left|
\tanh \left( \frac{\pi}{4d}(x - z) \right)
\right|
\nu_{N}(z)\, \d z
= v^{\ast}(x) \qquad x \in \mathbf{R}, 
\end{align}
and determine $\nu_{N}$\, (i.e., $\mu[\nu_{N}]$). 
\end{description}
\end{prob}
\noindent
We let $\nu_{N}^{\ast}$ denote the solution $\nu_{N}$ of SP2 in Problem~\ref{prob:2Steps}. 
In order to solve SP1 in Problem~\ref{prob:2Steps} approximately, 
we consider a procedure with some steps involving the Fourier transform. 
We will explain the basic ideas behind this procedure in Section~\ref{eq:reform_opt_harmonic},
and the full details of its steps are presented in Sections~\ref{sec:FT_IE} and~\ref{sec:Approx_para}. 
Furthermore, we obtain an approximate solution of SP2 in Problem~\ref{prob:2Steps} 
using a numerical computation based on the Fourier transform. 
Using these solutions for Problem~\ref{prob:2Steps}, 
we propose a procedure for obtaining the sampling points $a_{\ell}$ in Section~\ref{sec:Proc_formula}.

\subsection{A basic idea for SP1 in Problem~\ref{prob:2Steps}}
\label{eq:reform_opt_harmonic}

A key ingredient for solving SP1 in Problem~\ref{prob:2Steps} is provided by the following proposition. 

\begin{prop}
\label{prop:V_Dd_harmonic}
Let $\mu$ be a finite positive measure on $\mathbf{R}$ with compact support $[-\alpha, \alpha]$.
Then, the following hold: 
(i) The function $V_{\mathcal{D}_{d}}^{\mu}$ is harmonic on $\mathcal{D}_{d} \setminus [-\alpha, \alpha]$.
(ii) For any $\xi \in \partial \mathcal{D}_{d}$, 
\begin{align}
\lim_{\zeta \to \xi,\ \zeta \in \mathcal{D}_{d}} V_{\mathcal{D}_{d}}^{\mu}(\zeta) = 0.
\end{align}
\end{prop}

\noindent
\begin{proof}
The statement (i) is a straightforward consequence of the more general result of
%{\cite[Theorem II.5.1 (ii) (iv)]{bib:SaffTotik_LogPotExtField_1997}}, 
{Theorem II.5.1 (ii) (iv) in \cite{bib:SaffTotik_LogPotExtField_1997}}, 
and the statement (ii) immediately follows from the fact that 
$\log | \tanh ((\pi/(4d)) \xi)| = 0$ for any $\xi$ with $|\mathop{\mathrm{Im}} \xi| = d$.
\end{proof}

\vskip5pt

From Proposition \ref{prop:V_Dd_harmonic} and the equivalent expression of \eqref{eq:char_opt_eq_Prob} given by
\begin{align}
V_{\mathcal{D}_{d}}^{\mu}(x) = -\log w(x) - K_{N} \qquad (x \in [-\alpha_{N}, \alpha_{N}]), 
\end{align}
it follows that the function $\upsilon (x,y) = V_{\mathcal{D}_{d}}^{\mu}(x+ \mathrm{i}\, y)$ of real numbers $x$ and $y$
is the solution of the following Dirichlet problem on the doubly-connected region 
$\mathcal{D}_{d} \setminus [-\alpha_{N}, \alpha_{N}]$. 
\begin{align}
& \triangle \upsilon = 0 \quad \text{on} \quad \mathcal{D}_{d} \setminus [-\alpha_{N}, \alpha_{N}], \label{eq:Laplace} \\
& \upsilon(x, 0) = -\log w(x) - K_{N} \qquad (x \in [-\alpha_{N}, \alpha_{N}]), \label{eq:Laplace_inner_bnd} \\
& \upsilon(x, \pm d) = 0 \qquad (x \in \mathbf{R}). \label{eq:Laplace_outer_bnd}
\end{align}
Here we regard $\mathcal{D}_{d}$ as a region in $\mathbf{R}^{2}$. 
If we can determine the optimal values of the parameters $\alpha_{N}$ and $K_{N}$, and 
obtain the solution $\upsilon = \upsilon_{N}^{\ast}$ of \eqref{eq:Laplace}--\eqref{eq:Laplace_outer_bnd} 
satisfying the inequality equivalent to \eqref{eq:char_opt_ineq_Prob} given by
\begin{align}
\upsilon(x, 0) \leq -\log w(x) - K_{N} \qquad (x \in \mathbf{R} \setminus [-\alpha_{N}, \alpha_{N}]), 
\label{eq:Laplace_ineq}
\end{align}
then we only have to solve 
$V_{\mathcal{D}_{d}}^{\mu}(x) = \upsilon_{N}^{\ast}(x, 0) \ (x \in \mathbf{R})$
to obtain the optimal measure $\mu_{N}^{\ast}$. 
Therefore, we must carry out the following tasks. 
\begin{itemize}
\item 
Determine the optimal values of $\alpha_{N}$ and $K_{N}$, i.e., $\alpha_{N}^{\ast}$ and $K_{N}^{\ast}$. 
\item 
Solve the Dirichlet problem \eqref{eq:Laplace}--\eqref{eq:Laplace_outer_bnd} with condition \eqref{eq:Laplace_ineq} to obtain $\upsilon_{N}^{\ast}$.
\item 
Solve the equation 
$V_{\mathcal{D}_{d}}^{\mu}(x) = \upsilon_{N}^{\ast}(x, 0) \ (x \in \mathbf{R})$
to obtain $\mu_{N}^{\ast}$. 
\end{itemize}

The smoothness \eqref{eq:smooth_V}
and the total measure \eqref{eq:all_meas_b} conditions allow us to  
determine $\alpha_{N}^{\ast}$ and $K_{N}^{\ast}$ before we have obtained $\upsilon_{N}^{\ast}$. 
More precisely, we first derive an expression for $\upsilon$ containing the unknown parameters $\alpha_{N}$ and $K_{N}$, 
and then we apply these conditions to determine $\alpha_{N}^{\ast}$ and $K_{N}^{\ast}$. 
We first consider the Dirichlet problem defined by \eqref{eq:Laplace}--\eqref{eq:Laplace_outer_bnd}. 
In general, 
we can obtain a closed form of the solution of such a Dirichlet problem on a doubly-connected region
by conformally mapping the region to an annulus\footnote{
This mapping can be realized by combination of the $\tanh$ map and the inverse map of $f$ from (49) in 
%{\cite[p.~295]{bib:Nehari1975}}. 
{p.~295 in \cite{bib:Nehari1975}}. 
} 
and using the explicit solution of the Dirichlet problem on the annulus.
For example, see 
%{\cite[pp.~293--295]{bib:Nehari1975}} 
{pp.~293--295 in \cite{bib:Nehari1975}} 
and 
%{\cite[\S 17.4]{bib:Henrici1993}}.
{\S 17.4 in \cite{bib:Henrici1993}}.
However, the explicit solution is rather complicated. 
In order to obtain simple approximations for $\upsilon$, $\alpha_{N}^{\ast}$ and $K_{N}^{\ast}$, 
we derive a partial approximate solution of the Dirichlet problem \eqref{eq:Laplace}--\eqref{eq:Laplace_outer_bnd}. 
In fact, we only require the approximate solution on $\mathbf{R} \setminus [-\alpha_{N}, \alpha_{N}]$
in order to solve SP1 in Problem~\ref{prob:2Steps}.  
Next, we consider the Dirichlet problem on $\mathcal{D}_{d} \cap \{ z \mid | \mathop{\mathrm{Re}} z | > \alpha_{N} \}$. 

For this purpose, we perform a separation of variables $\upsilon(x,y) = X(x)Y(y)$ 
on the region $\mathcal{D}_{d} \cap \{ z \mid | \mathop{\mathrm{Re}} z | > \alpha_{N} \}$. 
For $x > \alpha_{N}$, 
we can derive an approximation of $\upsilon(x,y)$ 
from \eqref{eq:Laplace} and \eqref{eq:Laplace_outer_bnd} as
\begin{align}
\upsilon(x,y) \approx \sum_{n = 1}^{\infty} A_{n} \exp(-\sqrt{c_{n}} x)\, \cos(\sqrt{c_{n}} y), 
\label{eq:SepValApproxSol}
\end{align}
where $c_{n} = ( \pi (n-1/2)/d )^{2}$ for $n = 1,2, \ldots \, $. 
In this derivation, we have used the condition $\lim_{x \to \infty} X(x) = 0$, 
which is the boundary condition \eqref{eq:Laplace_outer_bnd} at infinity.
Here, 
we intuitively expect that the term for $n=1$ of \eqref{eq:SepValApproxSol} is a leading term, 
although we have not obtained a mathematical justification. 
Then, from \eqref{eq:SepValApproxSol} we have 
\begin{align}
\upsilon(x,y) \approx  
A_{1} \exp \left( - \frac{\pi}{2d}\, x \right)\, 
\cos \left( \frac{\pi}{2d}\, y \right).
\label{eq:SepValApproxSol_simple}
\end{align}
Furthermore, by applying the condition \eqref{eq:Laplace_inner_bnd} at $x=\alpha_{N}$ to \eqref{eq:SepValApproxSol_simple}, 
we can determine $A_{1}$ and obtain
\begin{align}
\upsilon(x,y) \approx  -(\log w(\alpha_{N}) + K_{N})\, 
\exp \left( - \frac{\pi}{2d}\, (x - \alpha_{N}) \right)\, 
\cos \left( \frac{\pi}{2d}\, y \right).
\label{eq:SepValApproxSol_simple_A1}
\end{align}
Then, from the smoothness condition \eqref{eq:smooth_V} at $x=\alpha_{N}$, 
we can obtain a relation between $\alpha_{N}$ and $K_{N}$ as
\begin{align}
\label{eq:approx_determ_K}
& \frac{\pi}{2d}(\log w(\alpha_{N}) + K_{N}) = - \frac{w'(\alpha_{N})}{w(\alpha_{N})} \notag \\
& \iff
K_{N} = K(\alpha_{N}) := -\log w(\alpha_{N}) - \frac{2d}{\pi} \frac{w'(\alpha_{N})}{w(\alpha_{N})}. 
\end{align}
Because of the symmetry of the problem with respect to the imaginary axis, 
we can apply a similar argument to the problem for $x < -\alpha_{N}$. 
Thus we obtain an approximation of $\upsilon$ on $\mathcal{D}_{d} \cap \{ z \mid |\mathop{\mathrm{Re}} z | > \alpha_{N} \}$
containing the unknown parameter $\alpha_{N}$ as 
\begin{align}
\upsilon(x,y) \approx \tilde{\upsilon}_{N}(x,y) :=
\frac{2d}{\pi} \frac{w'(\alpha_{N})}{w(\alpha_{N})} 
\exp \left( - \frac{\pi}{2d}\, (|x| - \alpha_{N}) \right)\, 
\cos \left( \frac{\pi}{2d}\, y \right).
\label{eq:SepValApproxSol_with_alpha}
\end{align}
By noting Assumption \ref{assump:w_convex} about the convexity of $-\log w(x)$, 
we can confirm that $\tilde{\upsilon}_{N}$ satisfies the condition \eqref{eq:Laplace_ineq}. That is, 
\begin{align}
\tilde{\upsilon}_{N}(x, 0) \leq -\log w(x) - K(\alpha_{N}) \qquad (x \in \mathbf{R} \setminus [-\alpha_{N}, \alpha_{N}]).
\label{eq:Laplace_ineq_approx_v_N}
\end{align}
This is because $\tilde{\upsilon}_{N}(x, 0)$ is concave on $\mathbf{R} \setminus [-\alpha_{N}, \alpha_{N}]$
and satisfies 
$\tilde{\upsilon}_{N}(\pm \alpha_{N}, 0) = -\log w(\pm \alpha_{N}) - K(\alpha_{N})$ and 
$(\tilde{\upsilon}_{N})_{x}(\pm \alpha_{N}, 0) = -w'(\pm \alpha_{N})/w(\pm \alpha_{N})$ 
owing to \eqref{eq:SepValApproxSol_simple_A1} and \eqref{eq:approx_determ_K}, respectively. 

From the above arguments, 
we can consider the following approximation of the equation 
$V_{\mathcal{D}_{d}}^{\mu}(x) = \upsilon_{N}^{\ast}(x, 0) \ (x \in \mathbf{R})$: 
\begin{align}
& V_{\mathcal{D}_{d}}^{\mu[\nu_{N}]}(x) = 
\begin{cases}
-\log w(x) - K(\alpha_{N}) & (x \in [-\alpha_{N}, \alpha_{N}]), \\
\tilde{\upsilon}_{N}(x,0) & (x \in \mathbf{R} \setminus [-\alpha_{N}, \alpha_{N}]).
\end{cases}
\label{eq:base_IE}
\end{align}
In fact, this is sufficient for obtaining an approximation for the solution $\nu_{N}^{\ast}$ as demonstrated below. 
Then, letting $\tilde{\nu}_{N}^{\ast}$ denote the solution of the equation \eqref{eq:base_IE}, 
we propose a procedure to solve SP1 in Problem~\ref{prob:2Steps} as follows. 
\begin{description}
\item[Step 1]
Derive the expression of $\mathcal{F}[\tilde{\nu}_{N}^{\ast}]$, which is 
the Fourier transform of $\tilde{\nu}_{N}^{\ast}$, from the equation \eqref{eq:base_IE}. 
\item[Step 2]
Obtain the approximate value of $\alpha_{N}^{\ast}$ by applying the condition \eqref{eq:all_meas_b} to $\mathcal{F}[\tilde{\nu}_{N}^{\ast}]$. 
Let $\tilde{\alpha}_{N}^{\ast}$ denote the approximate value. 
\end{description}
Once we have performed the above two steps, 
we obtain an approximate equation of $\nu_{N}$ on $\mathbf{R}$ by substituting $\tilde{\alpha}_{N}^{\ast}$ into \eqref{eq:base_IE}, 
which completes SP1 in Problem~\ref{prob:2Steps}. 
Therefore, following the two steps above, it only remains to solve SP2 in Problem~\ref{prob:2Steps} in order to obtain $\tilde{\nu}_{N}^{\ast}$. 
We present the procedure for carrying out this task as Step 3 below. 
\begin{description}
\item[Step 3]
Substitute $\tilde{\alpha}_{N}^{\ast}$ into the expression of $\mathcal{F}[\tilde{\nu}_{N}^{\ast}]$
and obtain $\tilde{\nu}_{N}^{\ast}$ by numerically computing 
the inverse Fourier transform of $\mathcal{F}[\tilde{\nu}_{N}^{\ast}]$.  
\end{description}
We present the full details of Steps 1 and 2 in Sections \ref{sec:FT_IE} and \ref{sec:Approx_para}, respectively. 
In Section \ref{sec:Proc_formula}, we present a procedure to design an accurate formula. 
Some notes relating to Step 3 are presented in Section \ref{sec:NumEx}, the section describing our numerical experiments.

\begin{rem}
Here, we explain the reasons why we have used the Fourier transform. 
In principle, 
we can determine $\tilde{\alpha}_{N}^{\ast}$, $\tilde{K}_{N}^{\ast}$, and $\tilde{\nu}_{N}^{\ast}$
by numerically solving a direct discretizaion of the equation \eqref{eq:base_IE}.
In that case, 
we may determine $\tilde{\alpha}_{N}^{\ast}$ by trial and error 
so that the numerical solution $\tilde{\nu}_{N}^{\ast}$ (approximately) satisfies the condition \eqref{eq:all_meas_b}.
Such a method is sufficient if we are only required to obtain numerical values for $\tilde{\alpha}_{N}^{\ast}$, $\tilde{K}_{N}^{\ast}$, 
and the sampling points given by $\tilde{\nu}_{N}^{\ast}$. 
However, we intend not only to derive a theoretical error estimate for our formula 
using closed expressions of $\tilde{\alpha}_{N}^{\ast}$ and $\tilde{K}_{N}^{\ast}$, 
but also to investigate how the error depends on the parameter $d$ and the weight $w$
which define the weighted Hardy space $\boldsymbol{H}^{\infty}(\mathcal{D}_{d}, w)$. 
These are the reasons why we employ the Fourier transform. 
See Sections \ref{sec:Approx_para} and \ref{sec:Err} for further details and examples.
Furthermore, as a by-product of our application of the Fourier transform, 
we can use the fast Fourier transform (FFT) in the execution of the above procedure. 
\end{rem}

\subsection{Step 1 for SP1 in Problem~\ref{prob:2Steps}: the Fourier transform of the solution of the integral equation \eqref{eq:base_IE}}
\label{sec:FT_IE}

We first note that the explicit form of the LHS of \eqref{eq:base_IE} is given by 
\begin{align}
V_{\mathcal{D}_{d}}^{\mu[\nu_{N}]}(x)
& = \int_{-\alpha_{N}}^{\alpha_{N}} \log \left| \tanh \left( \frac{\pi}{4d} (x-z) \right) \right| \, \d \mu[\nu_{N}](z) \notag \\
& = \int_{-\infty}^{\infty} \log \left| \tanh \left( \frac{\pi}{4d} (x-z) \right) \right| \, \nu_{N}(z) \, \d z. \notag 
\end{align}
Then, equation \eqref{eq:base_IE} with $\nu_{N} = \tilde{\nu}_{N}$ can be rewritten in the form
\begin{align}
& \int_{-\infty}^{\infty} \log \left| \tanh \left( \frac{\pi}{4d} (x-z) \right) \right| \, \tilde{\nu}_{N}(z) \, \d z \notag \\
& =
\begin{cases}
-\log w(x) - K(\alpha_{N}) & (x \in [-\alpha_{N}, \alpha_{N}]), \\
\frac{2d}{\pi} \frac{w'(\alpha_{N})}{w(\alpha_{N})} 
\exp \left( - \frac{\pi}{2d}\, (|x| - \alpha_{N}) \right) & (x \in \mathbf{R} \setminus [-\alpha_{N}, \alpha_{N}]).
\end{cases}
\label{eq:base_IE_RW}
\end{align}
For a function $f$ in a Lebesgue measurable space on $\mathbf{R}$, let $\mathcal{F}[f]$ be the Fourier transform of $f$ given by
\begin{align}
\mathcal{F}[f](\omega) = \int_{-\infty}^{\infty} f(x)\, \e^{-\i \, \omega \, x} \, \d x.
\end{align}
According to this definition, we can use the formula {\citep[p.~43, 7.112]{bib:Oberhettinger1990}}
\begin{align}
\mathcal{F}\left[ \log \left| \tanh \left( \frac{\pi}{4d} (\, \cdot \, ) \right) \right| \right](\omega) 
= - \frac{\pi}{\omega} \, \tanh \left( d\, \omega \right), 
\end{align}
to derive the Fourier transforms of both sides of \eqref{eq:base_IE_RW} as follows:
\begin{align}
\text{LHS}
= & 
\left[ - \frac{\pi}{\omega} \, \tanh \left( d\, \omega \right) \right] \, \cdot \, 
\mathcal{F}[\tilde{\nu}_{N}](\omega), \\
\text{RHS} 
= & 
\int_{-\alpha_{N}}^{\alpha_{N}} (-\log w(x) - K(\alpha_{N})) \, \mathrm{e}^{-\mathrm{i}\, \omega\, x} \mathrm{d} x \notag \\
& + 
2 \mathop{\mathrm{Re}} 
\left[
\int_{\alpha_{N}}^{\infty} 
\frac{2d}{\pi} \frac{w'(\alpha_{N})}{w(\alpha_{N})} 
\exp \left( - \frac{\pi}{2d}\, (x - \alpha_{N}) \right) \, \mathrm{e}^{-\mathrm{i}\, \omega\, x} 
\mathrm{d} x
\right] \notag \\
= & - \int_{-\alpha_{N}}^{\alpha_{N}} \log \frac{w(x)}{w(\alpha_{N})} \, \mathrm{e}^{-\mathrm{i}\, \omega\, x} \mathrm{d} x
+ \frac{4d}{\pi} \frac{w'(\alpha_{N})}{w(\alpha_{N})}\, \frac{\sin (\alpha_{N} \omega)}{\omega} \notag \\
& + \frac{4d}{\pi} \frac{w'(\alpha_{N})}{w(\alpha_{N})} 
\frac{2 \pi d \cos(\alpha_{N} \omega) - 4 d^{2} \omega \sin (\alpha_{N} \omega)}{\pi^{2} + 4 d^{2} \omega^{2}} \notag \\
= &
- \int_{-\alpha_{N}}^{\alpha_{N}} \log \frac{w(x)}{w(\alpha_{N})} \, \mathrm{e}^{-\mathrm{i}\, \omega\, x} \mathrm{d} x
+ \frac{4d}{\pi} \frac{w'(\alpha_{N})}{w(\alpha_{N})}\, 
\frac{ \pi^{2} \sin(\alpha_{N} \omega) + 2 \pi d\, \omega \cos (\alpha_{N} \omega)}{\omega (\pi^{2} + 4 d^{2} \omega^{2})}.
\label{eq:MainIE_ext_approx_R_Fourier}
\end{align}
Therefore, we have
\begin{align}
\mathcal{F}[\tilde{\nu}_{N}](\omega)
= &
\frac{\omega}{\pi\, \tanh (d\, \omega)}
\int_{-\alpha_{N}}^{\alpha_{N}} \log \frac{w(x)}{w(\alpha_{N})} \, \mathrm{e}^{-\mathrm{i}\, \omega\, x} \mathrm{d} x \notag \\
& - \frac{4d}{\pi} \frac{w'(\alpha_{N})}{w(\alpha_{N})}\, 
\frac{ \pi \sin(\alpha_{N} \omega) + 2 d\, \omega \cos (\alpha_{N} \omega)}{ (\pi^{2} + 4 d^{2} \omega^{2})\, \tanh(d\, \omega)} \notag \\
= &
\frac{1}{\i\, \pi\, \tanh (d\, \omega)}
\int_{-\alpha_{N}}^{\alpha_{N}} \frac{w'(x)}{w(x)} \, \mathrm{e}^{-\mathrm{i}\, \omega\, x} \mathrm{d} x \notag \\
& - \frac{4d}{\pi} \frac{w'(\alpha_{N})}{w(\alpha_{N})}\, 
\frac{ \pi \sin(\alpha_{N} \omega) + 2 d\, \omega \cos (\alpha_{N} \omega)}{ (\pi^{2} + 4 d^{2} \omega^{2})\, \tanh(d\, \omega)}. 
\label{eq:FT_nu_N}
\end{align}
We can confirm the existence of the inverse Fourier transform of $\mathcal{F}[\tilde{\nu}_{N}](\omega)$. 

\begin{thm}
\label{thm:FT_nu_N}
The function $\mathcal{F}[\tilde{\nu}_{N}]$ in \eqref{eq:FT_nu_N} belongs to $L^{2}(\mathbf{R})$ 
and has an inverse Fourier transform. 
\end{thm}

\noindent
\begin{proof}
We must confirm the square integrability on $\mathbf{R}$ of the first and second terms of the RHS in \eqref{eq:FT_nu_N}.
For the first term, we have 
\begin{align}
& \lim_{\omega \to 0} 
\frac{1}{\mathrm{i}\, \pi \, \tanh( d\, \omega )}
\int_{-\alpha_{N}}^{\alpha_{N}} \frac{w'(x)}{w(x)} \, \mathrm{e}^{-\mathrm{i}\, \omega\, x} \mathrm{d} x \notag \\
& =
\lim_{\omega \to 0} 
\frac{\cosh^{2}( d\, \omega )}{\mathrm{i}\, \pi\, d}
\int_{-\alpha_{N}}^{\alpha_{N}} \frac{-\mathrm{i}\, x\, w'(x)}{w(x)} \, \mathrm{e}^{-\mathrm{i}\, \omega \, x} \mathrm{d} x 
= 
- \frac{1}{\pi \, d} \int_{-\alpha_{N}}^{\alpha_{N}} \frac{x\, w'(x)}{w(x)} \, \mathrm{d}x, 
\label{eq:FTbp_first_zero}
\end{align}
which shows that the first term is bounded around the origin. 
Moreover, for $\omega$ with $|\omega| \gg 1$ the function $1/\tanh(d\, \omega)$ is bounded, 
and the function
\[
\int_{-\alpha_{N}}^{\alpha_{N}} \frac{w'(x)}{w(x)} \, \mathrm{e}^{-\mathrm{i}\, \omega\, x} \mathrm{d} x 
\]
is square integrable because this is the Fourier transform of the square integrable function 
$(w'(x)/w(x))\, \chi_{[-\alpha_{N}, \, \alpha_{N}]}(x)$. 
Thus, the first term is in $L^{2}(\mathbf{R})$. 
For the second term, we have
\begin{align}
& \lim_{\omega \to 0} 
\frac{4\, d}{\pi}
\frac{w'(\alpha_{N})}{w(\alpha_{N})}\, 
\frac{ \pi \sin(\alpha_{N} \omega) + 2d\, \omega \cos (\alpha_{N} \omega) }{(\pi^{2}+4d^{2}\, \omega^{2}) \, \tanh( d\, \omega )} \notag \\
& =
\frac{4\, d}{\pi^{3}} 
\frac{w'(\alpha_{N})}{w(\alpha_{N})}
\lim_{\omega \to 0} 
\frac{ \pi \sin(\alpha_{N} \omega) + 2d\, \omega \cos (\alpha_{N} \omega) }{\tanh( d\, \omega )}  \notag \\
& = 
\frac{4\, d}{\pi^{3}} 
\frac{w'(\alpha_{N})}{w(\alpha_{N})}
\lim_{\omega \to 0} 
\frac{(\pi \alpha_{N} + 2d) \cos(\alpha_{N} \omega) - 2d\, \alpha_{N}\, \omega \sin(\alpha_{N} \omega)}{d/\cosh^{2}(d\, \omega)}\notag \\
& = \frac{4(\pi \alpha_{N} + 2d)}{\pi^{3}} \frac{w'(\alpha_{N})}{w(\alpha_{N})}, 
\label{eq:FTbp_second_zero} 
\end{align}
which shows that the second term is bounded around the origin. 
Moreover, 
for $\omega$ with $|\omega| \gg 1$ the function $1/\tanh(d\, \omega)$ is bounded, 
and the function $\omega/(\pi^{2}+4d^{2}\, \omega^{2})$ is square integrable. 
Thus, the second term is in $L^{2}(\mathbf{R})$. 
\end{proof}

\subsection{Step 2 for SP1 in Problem~\ref{prob:2Steps}: approximation of the parameters $\alpha_{N}^{\ast}$ and $K_{N}^{\ast}$}
\label{sec:Approx_para}

We use the fact that the condition \eqref{eq:all_meas_b} 
can be described in terms of the Fourier transform of $\nu_{N}$ with 
$\mu_{N} = \mu[\nu_{N}]$ and 
$\mathop{\mathrm{supp}} \nu_{N} = [-\alpha_{N}, \alpha_{N}]$, as follows:
\begin{align}
\int_{-\alpha_{N}}^{\alpha_{N}} \nu_{N}(z) \, \mathrm{d}z = 2(N+1)
& \iff 
\int_{-\infty}^{\infty} \nu_{N}(z) \, \mathrm{d}z = 2(N+1) \notag \\
& \iff 
\lim_{\omega \to 0} \mathcal{F}[\nu_{N}](\omega)
= 2(N+1). \label{eq:b_const_equiv}
\end{align}
According to this condition \eqref{eq:b_const_equiv}, 
we assume that $\mathcal{F}[\tilde{\nu}_{N}]$ satisfies
\begin{align}
\label{eq:b_const_approx}
\lim_{\omega \to 0} \mathcal{F}[\tilde{\nu}_{N}](\omega) = 2(N+1)
\end{align}
in order to determine an approximate value of $\alpha_{N}^{\ast}$. 
It follows from \eqref{eq:FTbp_first_zero} and \eqref{eq:FTbp_second_zero} that 
the condition \eqref{eq:b_const_approx} is equivalent to 
\begin{align}
- \frac{1}{\pi \, d} \int_{-\alpha_{N}}^{\alpha_{N}} \frac{x\, w'(x)}{w(x)} \, \mathrm{d}x
- \frac{4(\pi \alpha_{N} + 2d)}{\pi^{3}} \frac{w'(\alpha_{N})}{w(\alpha_{N})} 
= 2(N+1).
\label{eq:det_alpha}
\end{align}
Equation \eqref{eq:det_alpha} has a unique solution
because the LHS of equation \eqref{eq:det_alpha} is a strictly increasing function of $\alpha_{N}$
that increases from $0$ (for $\alpha_{N} = 0$) to $+\infty$ (for $\alpha_{N} \to +\infty$). 
Then, let $\tilde{\alpha}_{N}^{\ast}$ be the unique solution of equation \eqref{eq:det_alpha}. 
For given $d$ and $w$, we can obtain a concrete value for $\tilde{\alpha}_{N}^{\ast}$ by solving \eqref{eq:det_alpha}. 
Then, by using the formula \eqref{eq:approx_determ_K}, 
we can also determine $\tilde{K}_{N}^{\ast}$, which is an approximate value of $K_{N}^{\ast}$, as follows:
\begin{align}
\label{eq:det_K}
\tilde{K}_{N}^{\ast} = K(\tilde{\alpha}_{N}^{\ast}) 
= - \log w(\tilde{\alpha}_{N}^{\ast}) - \frac{2d}{\pi} \frac{w'(\tilde{\alpha}_{N}^{\ast})}{w(\tilde{\alpha}_{N}^{\ast})}. 
\end{align}

\begin{ex}
\label{ex:SE}
Single exponential weight functions. 
For real numbers $\beta > 0$ and $\rho > 0$, we consider weight functions $w$ with 
\begin{align}
w(x) = \mathrm{O}(\exp(-(\beta |x|)^{\rho}))\quad (|x| \to \infty).
\end{align}
For simplicity, we consider the case that $w(x) = \exp(-(\beta |x|)^{\rho})$, 
although this $w$ does not always satisfy Assumption \ref{assump:w}%
\footnote{For example, if $\rho$ is an even integer, Assumption \ref{assump:w} is satisfied.}. 
Then, we pursue approximate formulas for $\tilde{\alpha}_{N}^{\ast}$ and $\tilde{K}_{N}^{\ast}$. 
It follows from 
\begin{align}
\frac{w'(x)}{w(x)} 
= 
\begin{cases}
- \beta^{\rho} \rho\, x^{\rho - 1} & (x > 0), \\
\beta^{\rho} \rho\, |x|^{\rho - 1} & (x < 0),
\end{cases}
\label{eq:SE_wp_w}
\end{align}
that 
\begin{align}
& -\frac{1}{\pi d} \int_{- \alpha_{N}}^{\alpha_{N}} \frac{x\, w'(x)}{w(x)} \d x
= \frac{2 \beta^{\rho} \rho}{\pi d (\rho + 1)} \alpha_{N}^{\rho + 1}, \\
& - \frac{4(\pi \alpha_{N} + 2d)}{\pi^{3}} \frac{w'(\alpha_{N})}{w(\alpha_{N})} 
= 
\frac{4\beta^{\rho} \rho\, \alpha_{N}^{\rho - 1} (\pi \alpha_{N} + 2d) }{\pi^{3}}. 
\end{align}
Using these expressions and \eqref{eq:det_alpha}, 
we obtain the equation 
\begin{align}
\frac{2 \beta^{\rho} \rho}{\pi d (\rho + 1)} \alpha_{N}^{\rho + 1}
+ 
\frac{4 \beta^{\rho} \rho\, \alpha_{N}^{\rho - 1} (\pi \alpha_{N} + 2d) }{\pi^{3}}
= 2(N+1)
\label{eq:SE_alpha_eq}
\end{align}
which determines $\tilde{\alpha}_{N}^{\ast}$. 
In order to obtain an asymptotic form of $\tilde{\alpha}_{N}^{\ast}$, 
we neglect the second term of the LHS in \eqref{eq:SE_alpha_eq}. 
Thus, we have
\begin{align}
\tilde{\alpha}_{N}^{\ast} 
\approx
\left(\frac{\pi d (\rho+1) (N+1)}{\beta^{\rho} \rho} \right)^{\frac{1}{\rho+1}} \qquad (N \to \infty).
\label{eq:SE_alpha_asymp}
\end{align}
By using this expression and \eqref{eq:det_K}, we obtain that
\begin{align}
\tilde{K}_{N}^{\ast}
\approx
\beta^{\rho} \left(\frac{\pi d (\rho+1) (N+1)}{\beta^{\rho} \rho} \right)^{\frac{\rho}{\rho+1}}
+
\frac{2d \beta^{\rho} \rho }{\pi} 
\left(\frac{\pi d (\rho+1) (N+1)}{\beta^{\rho} \rho} \right)^{\frac{\rho-1}{\rho+1}} \qquad (N \to \infty).
\label{eq:SE_K_asymp}
\end{align}
\end{ex}

\begin{ex}
\label{ex:DE}
Double exponential weight functions. 
For real numbers $\beta > 0$ and $\gamma > 0$, we consider weight functions $w$ with 
\begin{align}
w(x) = \mathrm{O}(\exp(-\beta \exp(\gamma |x|)))\quad (|x| \to \infty).
\end{align}
For simplicity, we consider the case that $w(x) = \exp(-\beta \exp(\gamma |x|))$, 
although this $w$ does not satisfy Assumption \ref{assump:w}. 
Then, we pursue approximate formulas for $\tilde{\alpha}_{N}^{\ast}$ and $\tilde{K}_{N}^{\ast}$. 
It follows from 
\begin{align}
\frac{w'(x)}{w(x)} 
= 
\begin{cases}
- \beta \gamma \exp(\gamma x) & (x > 0), \\
\beta \gamma \exp(-\gamma x)  & (x < 0),
\end{cases}
\label{eq:DE_wp_w}
\end{align}
that 
\begin{align}
& -\frac{1}{\pi d} \int_{- \alpha_{N}}^{\alpha_{N}} \frac{x\, w'(x)}{w(x)} \d x
= 
\frac{2 \beta}{\pi d \gamma} 
\left[ (\gamma \alpha_{N} - 1)\, \e^{\gamma \alpha_{N}} + 1 \right], \\
& - \frac{4(\pi \alpha_{N} + 2d)}{\pi^{3}} \frac{w'(\alpha_{N})}{w(\alpha_{N})} 
= \frac{4\beta \gamma (\pi \alpha_{N} + 2d)}{\pi^{3}}\, \e^{\gamma \alpha_{N}}. 
\end{align}
Using these expressions and \eqref{eq:det_alpha}, 
we obtain the equation 
\begin{align}
\left(\frac{2 \beta}{\pi d} + \frac{4\beta \gamma}{\pi^{2}} \right) \alpha_{N}\, \e^{\gamma \alpha_{N}}   
- \left( \frac{2 \beta}{\pi d \gamma} - \frac{8d \beta \gamma}{\pi^{3}} \right) \e^{\gamma \alpha_{N}} 
+ \frac{2 \beta}{\pi d \gamma} 
= 2(N+1)
\label{eq:DE_alpha_eq}
\end{align}
which determines $\tilde{\alpha}_{N}^{\ast}$. 
In order to obtain an asymptotic form of $\tilde{\alpha}_{N}^{\ast}$, 
we neglect the second and third terms of the LHS in \eqref{eq:DE_alpha_eq}. 
Thus, we have
\begin{align}
\tilde{\alpha}_{N}^{\ast} 
\approx
\frac{1}{\gamma}\, W\left( \frac{\pi^{2} d \gamma (N+1)}{(\pi + 2 d \gamma) \beta} \right)
\approx
\frac{1}{\gamma}\, \log \left( \frac{\pi^{2} d \gamma (N+1)}{(\pi + 2 d \gamma) \beta} \right) \qquad (N \to \infty),
\label{eq:DE_alpha_asymp}
\end{align}
where $W$ is Lambert's W function, i.e., the inverse function of $g(x) = x\, \e^{x}$. 
By using this expression and \eqref{eq:det_K}, we obtain that
\begin{align}
\tilde{K}_{N}^{\ast}
& \approx
\left( 1 + \frac{2d \gamma}{\pi} \right) \beta \exp \left( W\left( \frac{\pi^{2} d \gamma (N+1)}{(\pi + 2d \gamma) \beta} \right) \right) \notag \\
& =
\frac{\pi d \gamma (N+1)}{W\left( \frac{\pi^{2} d \gamma (N+1)}{(\pi + 2d \gamma) \beta} \right)}
\approx
\frac{\pi d \gamma (N+1)}{\log \left( \frac{\pi^{2} d \gamma (N+1)}{(\pi + 2d \gamma) \beta} \right)} \qquad (N \to \infty).
\label{eq:DE_K_asymp}
\end{align}
\end{ex}

\subsection{Procedure to design accurate formulas}
\label{sec:Proc_formula}

In view of the considerations discussed in Sections \ref{eq:reform_opt_harmonic}--\ref{sec:Approx_para}, 
we can establish a procedure to design an accurate formula for each $\boldsymbol{H}^{\infty}(\mathcal{D}_{d}, w)$ as follows.

\begin{enumerate}
\item
For given $d$, $w$, and $N$, determine $\tilde{\alpha}_{N}^{\ast}$
by solving the equation \eqref{eq:det_alpha}. 
Furthermore, determine $\tilde{K}_{N}^{\ast}$ using \eqref{eq:det_K}.

\item 
Compute $\tilde{\nu}_{N}(x)$ for $x \in [-\tilde{\alpha}_{N}, \tilde{\alpha}_{N}]$
using the inverse Fourier transform of $\mathcal{F}[\tilde{\nu}_{N}]$ in \eqref{eq:FT_nu_N}. 

\item
Compute the indefinite integral 
$\displaystyle I[\tilde{\nu}_{N}](x) := \int_{0}^{x} \tilde{\nu}_{N}(t)\, \d t$ for $x \in [-\tilde{\alpha}_{N}, \tilde{\alpha}_{N}]$.

\item
Compute $I[\tilde{\nu}_{N}]^{-1}$, which is the inverse function of $I[\tilde{\nu}_{N}]$. 

\item
Generate sampling points $a_{i}$ as $a_{i} = I[\tilde{\nu}_{N}]^{-1}(i)\ (i=-N, \ldots, N)$. 

\item
Obtain an approximation formula $\tilde{f}_{N}$ of $f$ for $x \in \mathbf{R}$ as 
\begin{align}
\tilde{f}_{N}(x) 
& := 
\sum_{j = -N}^{N} f(a_{j}) \, 
\frac{B_{N; j}(x; \{ a_{i} \}, \mathcal{D}_{d})\, w(x)}{B_{N; j}(a_{j}; \{ a_{i} \}, \mathcal{D}_{d})\, w(a_{j})} \, 
\frac{4d}{\pi} T_{d}'(a_{j} - x)\notag \\
& =
\sum_{j = -N}^{N} f(a_{j})\,
\frac{ w(x) 
\displaystyle
\prod_{\begin{subarray}{c} -N\leq i \leq N,\\ i \neq j \end{subarray} } 
\tanh \left( \frac{\pi}{4d} (x - a_{i}) \right)}
{\displaystyle
w(a_{j})
\prod_{\begin{subarray}{c} -N\leq i \leq N,\\ i \neq j \end{subarray} } 
\tanh \left( \frac{\pi}{4d} (a_{j} - a_{i}) \right)} 
\, \mathrm{sech}^{2} \left( \frac{\pi}{4d} (a_{j} - x) \right). 
\label{eq:final_approx_formula}
\end{align}
\end{enumerate}

%------------------------------
\section{Error estimate}
\label{sec:Err}

\subsection{General error estimate}

Here, we estimate $\sup_{x \in \mathbf{R}} | f(x) - \tilde{f}_{N}(x) |$, 
which is the error of the approximation of $f$ by $\tilde{f}_{N}$ in \eqref{eq:final_approx_formula}.  
For this purpose, we first note that 
\begin{align}
\sup_{x \in \mathbf{R}} 
\left| f(x) - \tilde{f}_{N}(x) \right|
& \leq
\| f \| \, \sup_{x \in \mathbf{R}} 
\left|
w(x)\, B_{N} (x; \{ a_{i} \}, \mathcal{D}_{d})
\right| \notag \\
& = 
\| f \| \, \sup_{x \in \mathbf{R}} 
\left|
w(x)\, \prod_{j = -N}^{N} \tanh \left( \frac{\pi}{4d} (x - a_{j}) \right) 
\right|, \label{eq:ErrorUpper}
\end{align}
where 
\begin{align}
\| f \| = \sup_{z \in \mathcal{D}_{d}} \left| \frac{f(z)}{w(z)} \right|. 
\end{align}
This fact is derived in the proof of 
%\cite[Lemma 4.3]{bib:Sugihara_NearOpt_2003}. 
Lemma 4.3 in \cite{bib:Sugihara_NearOpt_2003}. 
Furthermore, we will prove the following lemma, which provides an estimate of 
the difference between the discrete potential \eqref{eq:minus_G_pot_rewrited}
and its continuous counterpart \eqref{eq:minus_G_pot_approx}. 

\begin{lem}
\label{lem:GrPotDiscError_mod}
Let $a: [-N-1, N+1] \to \mathbf{R}$ be a $C^{1}$ strictly monotone increasing function. 
Assume that its inverse function $b = a^{-1}$ satisfies
\begin{align}
\max_{x \in [a(-N-1), a(N+1)]} |b'(x)| \leq c\, N^{\lambda}
\label{eq:b_prime_max}
\end{align}
for some constants $c > 0$ and $\lambda > 0$ that are independent of $N$. 
Then, for sufficiently large $N$, we have
\begin{align}
& \sum_{j = -N}^{N} \log \left| \tanh \left( \frac{\pi}{4d} (x - a(j)) \right) \right| \notag \\
& \leq 
\int_{a(-N-1)}^{a(N+1)} \log \left| \tanh \left( \frac{\pi}{4d} (x - z) \right) \right| b'(z) \, \d z + c' \log N + c'',
\label{eq:DiscPotUpper} 
\end{align}
where $c'$ and $c''$ are constants depending only on $c$, $d$, and $\lambda$.
\end{lem}

\noindent
We prove this lemma in Section \ref{sec:proof_GrPotDiscError_mod}. 
Now, under some assumptions, we can estimate the error $\sup_{x \in \mathbf{R}} | f(x) - \tilde{f}_{N}(x) |$. 
\begin{thm}
\label{thm:general_error_estimate}
Let the weight function $w$ satisfy Assumptions \ref{assump:w}--\ref{assump:w_convex}. 
Furthermore, suppose that for each positive integer $N$ 
the sampling points $a_{i}$ are obtained using the procedure presented in Section \ref{sec:Proc_formula}, 
and that the function $\tilde{\nu}_{N}$ obtained in that procedure satisfies a counterpart of \eqref{eq:b_prime_max}, i.e., 
\begin{align}
\max_{x \in [-\tilde{\alpha}_{N}^{\ast}, \tilde{\alpha}_{N}^{\ast}]} |\tilde{\nu}_{N}(x)| \leq c\, N^{\lambda}. 
\label{eq:nu_max}
\end{align}
Then, for any $f \in \boldsymbol{H}^{\infty}(\mathcal{D}_{d}, w)$, 
the approximation $\tilde{f}_{N}$ given by \eqref{eq:final_approx_formula} satisfies
\begin{align}
\sup_{x \in \mathbf{R}} 
\left| f(x) - \tilde{f}_{N}(x) \right|
\leq 
C\, \| f \| \, N^{c'}
\exp \left( -\tilde{K}_{N}^{\ast} \right)
\end{align}
for sufficiently large $N$, where $C$ and $c'$ are constants that are independent of $f$ and $N$.
\end{thm}

\noindent
\begin{proof}
From \eqref{eq:nu_max} and Lemma \ref{lem:GrPotDiscError_mod}, it follows that 
\begin{align}
& \sum_{j = -N}^{N} \log \left| \tanh \left( \frac{\pi}{4d} (x - a_{j}) \right) \right| \notag \\
& \leq 
\int_{-\tilde{\alpha}_{N}^{\ast}}^{\tilde{\alpha}_{N}^{\ast}} 
\log \left| \tanh \left( \frac{\pi}{4d} (x - z) \right) \right|\, \d \tilde{\nu}_{N}(z) + c' \log N + c'', \notag \\
& \leq -\log w(x) - \tilde{K}_{N}^{\ast} + c' \log N + c''. 
\end{align}
By combining this estimate and \eqref{eq:ErrorUpper}, we obtain that
\begin{align}
& \sup_{x \in \mathbf{R}} 
\left| f(x) - \tilde{f}_{N}(x) \right| \notag \\
& \leq 
\| f \|
\sup_{x \in \mathbf{R}}
\left[
\exp
\left(
\log w(x) 
+
\int_{-\tilde{\alpha}_{N}^{\ast}}^{\tilde{\alpha}_{N}^{\ast}} 
\log \left| \tanh \left( \frac{\pi}{4d} (x - z) \right) \right|\, \d \tilde{\nu}_{N}(z) + c' \log N + c''
\right)
\right] \notag \\
& \leq
\| f \|
\sup_{x \in \mathbf{R}}
\left[
\exp
\left(
- \tilde{K}_{N}^{\ast} + c' \log N + c''
\right)
\right] 
=
C \| f \| \, 
N^{c'}
\exp
\left(
- \tilde{K}_{N}^{\ast}
\right), 
\notag 
\end{align}
which concludes our proof. 
\end{proof}

\subsection{Examples of error estimate}

Now, we present the explicit forms of the error estimate presented in Theorem \ref{eq:nu_max} 
for the weight functions $w$ given in Examples \ref{ex:SE}--\ref{ex:DE}.
For this, we require one additional assumption to confirm the condition \eqref{eq:nu_max}. 

\begin{assump}
\label{assump:nu_max_x_0}
The function $| \tilde{\nu}_{N}(x) |$ takes its maximum value at $x = 0$, 
where $\tilde{\nu}_{N}$ is obtained in the procedure presented in Section \ref{sec:Proc_formula}. 
\end{assump}

\begin{rem}
According to the conditions \eqref{eq:FT_nu_N} and \eqref{eq:det_alpha}, 
the validity of Assumption \ref{assump:nu_max_x_0} depends on the real constant $d$ and the weight $w$. 
Therefore, it is preferable to formulate a sufficient condition for the statement of Assumption \ref{assump:nu_max_x_0} in terms of $d$ and $w$. 
However, we do not derive such a condition here, and leave this as a theme for a future study. 
In addition, as shown in Section \ref{sec:CompSample}, 
we numerically confirm the validity of Assumption \ref{assump:nu_max_x_0} in practical applications. 
\end{rem}

\noindent
Under Assumption \ref{assump:nu_max_x_0}, we can provide an estimate to confirm the condition \eqref{eq:nu_max}. 

\begin{lem}
\label{lem:mu_prime_estim}
Suppose that the weight function $w$ satisfies Assumptions \ref{assump:w}--\ref{assump:w_convex}. 
Then, letting $v$ be defined by $v(x) = w'(x)/w(x)$, we have
\begin{align}
& \max_{x \in [-\tilde{\alpha}_{N}^{\ast}, \tilde{\alpha}_{N}^{\ast}]} |\tilde{\nu}_{N}(x)| \notag \\
& \leq 
\frac{1}{\tanh(\pi d/(2 \tilde{\alpha}_{N}^{\ast}))}
\left[
- \frac{1}{4 (\tilde{\alpha}_{N}^{\ast})^{2}} \int_{- \tilde{\alpha}_{N}^{\ast}}^{ \tilde{\alpha}_{N}^{\ast} } x\, v(x)\, \mathrm{d}x 
+ \left( \frac{1}{2} + \frac{8}{\pi^{3}} \right) |v(\tilde{\alpha}_{N}^{\ast})| 
\right. \notag \\ 
& \quad \left.
+ 
\frac{4}{\pi^{3}}
\left(
\tilde{\alpha}_{N}^{\ast}\, | v'(\tilde{\alpha}_{N}^{\ast}) | 
+ (\tilde{\alpha}_{N}^{\ast})^{2} \max_{x \in [-\tilde{\alpha}_{N}^{\ast}, \tilde{\alpha}_{N}^{\ast}]} |v''(x)|
\right) 
\right] \notag \\
& \quad + 
\left( \frac{2}{\pi^{2}} + \frac{8 d}{\pi^{3} \tilde{\alpha}_{N}^{\ast}} \right) \frac{|v(\tilde{\alpha}_{N}^{\ast})|}{\tanh(\pi/2)}.
\end{align}
\end{lem}

\noindent
We prove this lemma in Section \ref{sec:proof_mu_prime_estim}.

\begin{ex}
\label{ex:SE_error}
Error estimate in the case of the single exponential weight functions given in Example \ref{ex:SE}. 
Using \eqref{eq:SE_wp_w} and \eqref{eq:SE_alpha_asymp}, 
we can deduce from Lemma \ref{lem:mu_prime_estim} that 
\begin{align}
\max_{x \in [-\tilde{\alpha}_{N}^{\ast}, \tilde{\alpha}_{N}^{\ast}]} |\tilde{\nu}_{N}(x)| 
\lesssim  (\tilde{\alpha}_{N}^{\ast})^{\rho} 
\lesssim  N^{\frac{\rho}{\rho + 1}}
\end{align}
for sufficiently large $N$. 
Therefore, it follows from \eqref{eq:SE_K_asymp} and Theorem \ref{thm:general_error_estimate} that
\begin{align}
\sup_{x \in \mathbf{R}} 
\left| f(x) - \tilde{f}_{N}(x) \right|
\lesssim
C\, \| f \| \, N^{c'}
\exp \left[ -\beta^{\rho} \left(\frac{\pi d (\rho+1) (N+1)}{\beta^{\rho} \rho} \right)^{\frac{\rho}{\rho+1}} \right].
\label{eq:SE_rough_error}
\end{align}
In particular, in the case that $d=\pi/4$ and $\rho = 1$, 
the argument of the exponential in \eqref{eq:SE_rough_error} is equal to
\begin{align}
-\beta^{\rho} \left(\frac{\pi d (\rho+1) (N+1)}{\beta^{\rho} \rho} \right)^{\frac{\rho}{\rho+1}}
=
-\pi \sqrt{\frac{\beta (N+1)}{2}}. 
\end{align}
This expression is almost the same as the arguments of the exponentials in \eqref{eq:ExactEminSech}, 
which provides an error estimate for Ganelius's formula, which is described in Section~\ref{sec:DefOptApprox}. 
\end{ex}

\begin{ex}
\label{ex:DE_error}
Error estimate in the case of the double exponential weight functions given in Example \ref{ex:DE}. 
Using \eqref{eq:DE_wp_w} and \eqref{eq:DE_alpha_asymp}, 
we deduce from Lemma \ref{lem:mu_prime_estim} that 
\begin{align}
\max_{x \in [-\tilde{\alpha}_{N}^{\ast}, \tilde{\alpha}_{N}^{\ast}]} |\tilde{\nu}_{N}(x)| 
\lesssim  (\tilde{\alpha}_{N}^{\ast})^{3} \, \e^{\gamma \tilde{\alpha}_{N}^{\ast}} 
\lesssim (\tilde{\alpha}_{N}^{\ast})^{2} \, N
\lesssim (\log N)^{2} \, N
\end{align}
for sufficiently large $N$. 
Therefore, it follows from \eqref{eq:DE_K_asymp} and Theorem \ref{thm:general_error_estimate} that
\begin{align}
\sup_{x \in \mathbf{R}} 
\left| f(x) - \tilde{f}_{N}(x) \right|
\lesssim
C\, \| f \| \, N^{c'}
\exp \left[ - 
\frac{\pi d \gamma (N+1)}{\log \left( \frac{\pi^{2} d \gamma (N+1)}{(\pi + 2d \gamma) \beta} \right)} \right].
\label{eq:DE_rough_error}
\end{align}
The rate of this bound is close to that of the upper bound
\begin{align}
C_{d, w} \exp 
\left[ 
-\frac{\pi d \gamma N}{\log(\pi d \gamma N/\beta)}
\right]
\label{eq:Sugihara2003DE-Sinc}
\end{align}
of the error of the DE-Sinc formula \citep[Theorem 3.2 (3.3)]{bib:Sugihara_NearOpt_2003}%
\footnote{We need to note that \cite{bib:Sugihara_NearOpt_2003} uses $N$ 
as the total number of the sampling points. Therefore, 
in \eqref{eq:Sugihara2003DE-Sinc} and \eqref{eq:Sugihara2003min} 
we rewrite the bounds by setting $N = (n-1)/2$, 
where $n$ is the total number of the sampling points.}, 
and worse than that of the lower bound 
\begin{align}
C_{d, w}' \exp 
\left[ 
-\frac{2 \pi d \gamma N}{\log(\pi d \gamma N/\beta)}
\right]
\label{eq:Sugihara2003min}
\end{align}
of the minimum error norm \citep[Theorem 3.2 (3.4)]{bib:Sugihara_NearOpt_2003}.
These facts do not seem to agree with our expectation 
that the bound in \eqref{eq:DE_rough_error} is close to the minimum error norm. 
However, because the computation to derive \eqref{eq:DE_rough_error} is somewhat rough, 
we need to pursue the exact value of the minimum error norm in a more rigorous manner. 
\end{ex}

%------------------------------
\section{Numerical experiments}
\label{sec:NumEx}

For numerical experiments on approximations using our formula $\tilde{f}_{N}(x)$ given by \eqref{eq:final_approx_formula}, 
we chose the functions listed in Table \ref{tab:func_weight}. 
For simplicity, we set $d = \pi/4$. 
The Gaussian weight in row (2) of Table \ref{tab:func_weight} is of a single exponential weight type. 
In fact, by letting $\rho = 2$ in Example \ref{ex:SE}, we have Gaussian weights. 

\begin{table}[H]
\begin{center}
\caption{Functions $f$ in $\boldsymbol{H}^{\infty}(\mathcal{D}_{d}, w)$ 
and weight functions $w$ corresponding to them}
\label{tab:func_weight}
\begin{tabular}{l | l l | l }
$f \in \boldsymbol{H}^{\infty}(\mathcal{D}_{\pi/4}, w)$ & $w$ 
& \begin{minipage}{0.10\linewidth} {\small decay \\ type} \end{minipage} 
& \begin{minipage}{0.28\linewidth} {\small Approximation formula \\ for comparison} \end{minipage} \\
\hline
(1)\ $f(x) = \mathop{\mathrm{sech}}(2x)$ & $f(x)$ & SE & SE-Sinc, Ganelius \\
(2)\ $f(x) = \dfrac{x^2}{(\pi/4)^2 + x^2}\, \e^{-x^{2}} $ & $ \e^{-x^{2}}$ & Gaussian & SE-Sinc, Ganelius \\
(3)\ $f(x) = \mathop{\mathrm{sech}}((\pi/2) \sinh(2x))$ & $f(x)$ & DE & DE-Sinc \\ 
\end{tabular}
\end{center}
\end{table}

We will explain the numerical algorithms for producing the sampling points for our formulas in Section~\ref{sec:NumAlgo}, and 
will present the results of the computations of them in Section~\ref{sec:CompSample}. 
Then, we will present the results of the approximations provided by our formulas in Section~\ref{sec:approx_results}. 

%--------------------------------------------------
\subsection{Numerical algorithms}
\label{sec:NumAlgo}

In this section, we present the numerical algorithms for executing the procedure proposed in Section~\ref{sec:Proc_formula}. 
In Step 1, we apply the Newton method to solve the equation \eqref{eq:det_alpha} for $\tilde{\alpha}_{N}^{\ast}$
using the approximate expression \eqref{eq:SE_alpha_asymp} or \eqref{eq:DE_alpha_asymp} as an initial value. 
To compute the integral 
$\int_{-\alpha_{N}}^{\alpha_{N}} (x\, w'(x)/w(x)) \d x$, 
we apply the mid-point rule with the grid $\{ x_{i} \}$, where
\begin{align}
x_{i} = i\, h_{x} \qquad i = -M+1, -M+2,\ldots, M-1, M
\label{eq:x_grid}
\end{align}
with $M=2^{12}$ and $h_{x} = \tilde{\alpha}_{N}^{\ast}/M$.
In Step 2, 
for the discretization of the inverse Fourier transform of $\mathcal{F}[\tilde{\nu}_{N}](\omega)$, 
we apply the equispaced grid $\{ \omega_{i} \}$ with
\begin{align}
\omega_{i} = i\, h_{\omega} \qquad i = -M+1, -M+2,\ldots, M-1, M, 
\end{align}
where $h_{\omega} = \pi/\tilde{\alpha}_{N}^{\ast}$. 
Then, using the mid-point rule with this grid, 
we compute the approximate values of $\tilde{\nu}_{N}(x)$ for $x = x_{i}$ given by \eqref{eq:x_grid}. 
In addition, we apply a fractional FFT \citep{bib:BailSwarFRFT1991} to speed up the computation of 
the discrete Fourier transform (DFT) without the Nyquist condition about $h_{x}$ and $h_{\omega}$. 
In Step 3, we compute the approximate values of the integral $I[\tilde{\nu}_{N}](x_{i})$
using the standard Euler scheme, i.e., 
\begin{align}
\tilde{I}[\tilde{\nu}_{N}](x_{i+1}) = \tilde{I}[\tilde{\nu}_{N}](x_{i}) + \tilde{\nu}_{N}(x_{i}) \, h_{x}
\qquad i=0,1,\ldots, M-1,
\label{eq:Indefint1stEuler}
\end{align}
where $\tilde{I}[\tilde{\nu}_{N}]$ denotes the approximation of $I[\tilde{\nu}_{N}]$. 
Furthermore, because $\tilde{\nu}_{N}$ is even, 
we compute $\tilde{I}[\tilde{\nu}_{N}](x_{i})$ for $i<0$ as 
\begin{align}
\tilde{I}[\tilde{\nu}_{N}](x_{-j}) = -\tilde{I}[\tilde{\nu}_{N}](x_{j})\qquad j=1,2,\ldots, M.
\label{eq:Inverse3rd}
\end{align}
In Step 4, we apply piecewise third order interpolation\footnote{
We used the Matlab function \texttt{interp1( , , , 'pchip')} of Matlab R2015a.}
to the data $\{ (\tilde{I}[\tilde{\nu}_{N}](x_{i}), x_{i}) \}$
to obtain the inverse of $\tilde{I}[\tilde{\nu}_{N}]$. 

\begin{rem}
\label{rem:orders}
It may seem inconsistent to use 
a first-order Euler scheme for the integral in \eqref{eq:Indefint1stEuler} and 
a third order interpolant for the inverse. 
However, 
we do not care about that because the accuracy of the computed sampling points
does not seem to make so much difference in the performance of the resulting approximation formula
as far as we judge from the numerical results in Section \ref{sec:approx_results}. 
Theoretical investigation about the robustness of the formula against the errors in the sampling points
is an important and interesting issue, which we leave as future work. 
\end{rem}

%--------------------------------------------------
\subsection{Computed sampling points}
\label{sec:CompSample}

We present the computed sampling points for $d = \pi/4$ and the weight functions in Table \ref{tab:func_weight}. 
First, in order to confirm Assumption~\ref{assump:nu_max_x_0} numerically, 
we plot the computed values of $\tilde{\nu}_{N}$. 
Next, we show the computed sampling points and the discrete weighted potential %\eqref{eq:DWP_revisited} 
\begin{align}
\log_{10} w(x) + \sum_{j=-N}^{N} \log_{10} \left| \tanh \left( \frac{\pi}{4d} (x-a_{j}) \right) \right|
\label{eq:DWP_revisited}
\end{align}
for each weight function $w$. 
Because the function~\eqref{eq:DWP_revisited} is the approximation of $\log w(x) + V_{\mathcal{D}_{\pi/4}}^{\mu[\tilde{\nu}_{N}^{\ast}]}(x)$, 
we can expect it to be almost ``flat'' on the interval $[-\tilde{\alpha}_{N}^{\ast}, \tilde{\alpha}_{N}^{\ast}]$ and to decay outside of it. 
The computations of the sampling points were performed 
using Matlab R2015a programs with double precision floating point numbers.
These programs used for the computations are available
on the web page \cite{bib:TanaMatlab2015}. 

The results for the weights in (1), (2), and (3) are presented in 
Figures \ref{fig:samp_SE}, \ref{fig:samp_Gauss}, and \ref{fig:samp_DE}, respectively. 
In each graph (a) from Figures \ref{fig:samp_SE}--\ref{fig:samp_DE}, 
we can observe that the functions $\tilde{\nu}_{N}$ are unimodal and take their maximums at the origin, 
although some outliers appear around the endpoints, particularly in the case of the DE weight. 
Therefore, in the case of the SE and the Gaussian weight, we can confirm Assumption~\ref{assump:nu_max_x_0}. 
We suspect that the outliers are the result of numerical errors. 
As for the discrete weighted functions, 
in graphs (c) and (d) in Figures \ref{fig:samp_SE}--\ref{fig:samp_DE}, 
we can observe that the results are consistent with our expectation for small $N$. 
On the other hand, 
the discrete weighted potentials are warped particularly in the case of the DE weight for large $N$. 
We leave the investigation of these phenomena as a topic for future work. 

%-----
\begin{figure}[H]
\noindent
\begin{minipage}[t]{0.49\linewidth}
\begin{center}
\includegraphics[width=\linewidth]{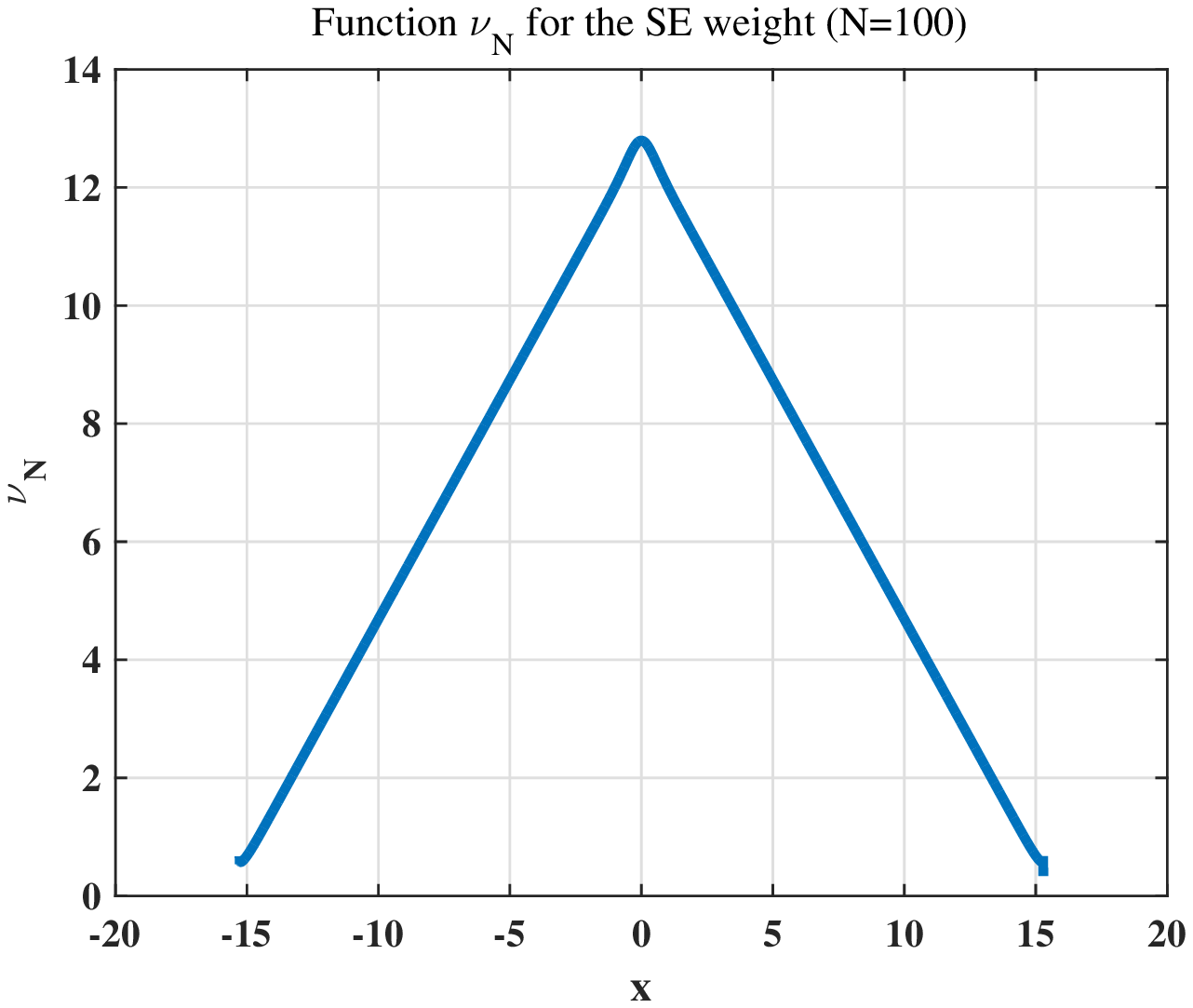}
(a) Function $\tilde{\nu}_{N}$ for $N=100$
\end{center}
\end{minipage}
\begin{minipage}[t]{0.49\linewidth}
\begin{center}
\includegraphics[width=\linewidth]{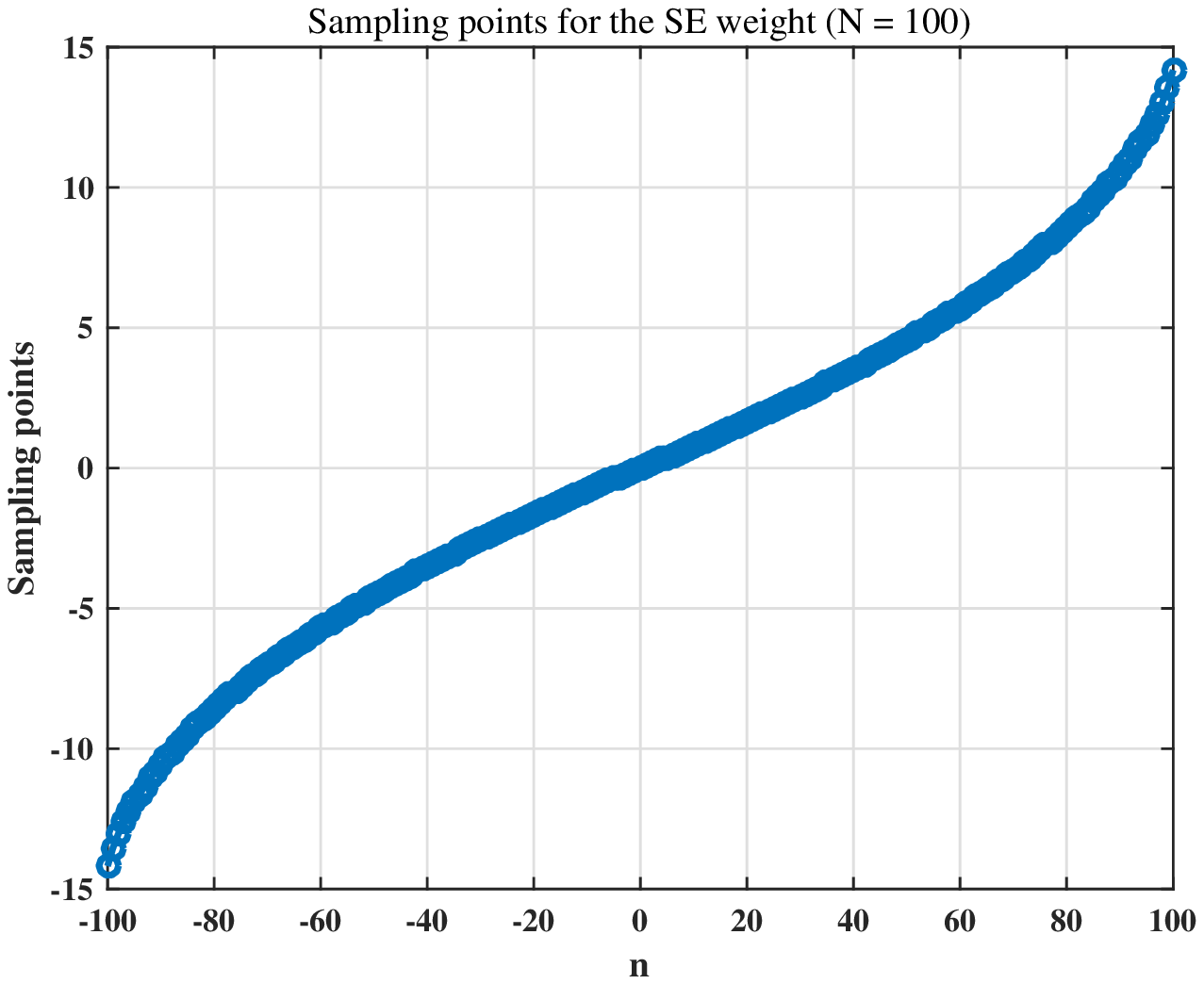}
(b) Sampling points for $N=100$
\end{center}
\end{minipage}

\hspace{5mm}

\noindent
\begin{minipage}[t]{0.49\linewidth}
\begin{center}
\includegraphics[width=\linewidth]{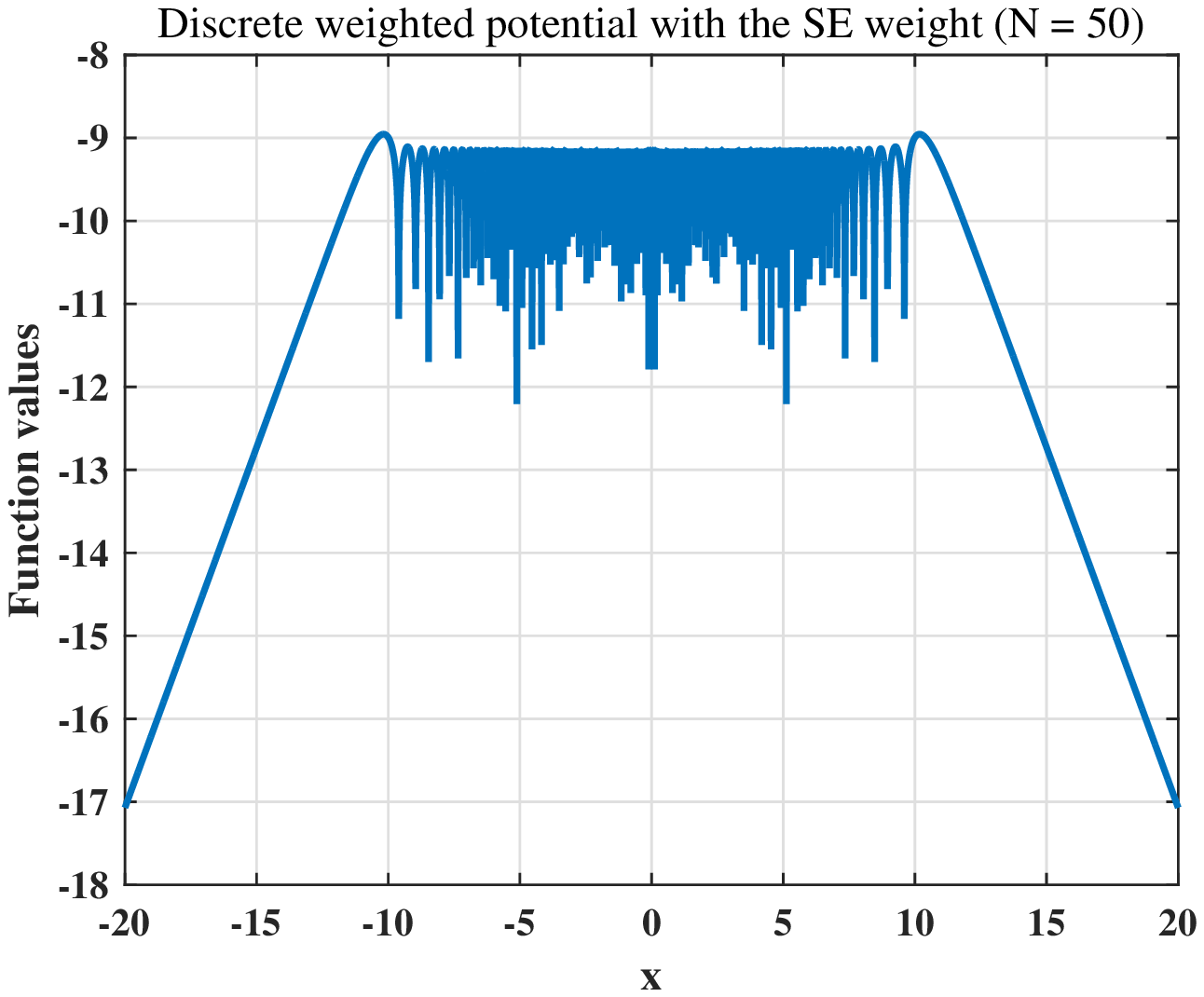}
(c) Discrete weighted potential \eqref{eq:DWP_revisited} for $N=50$
\end{center}
\end{minipage}
\begin{minipage}[t]{0.49\linewidth}
\begin{center}
\includegraphics[width=\linewidth]{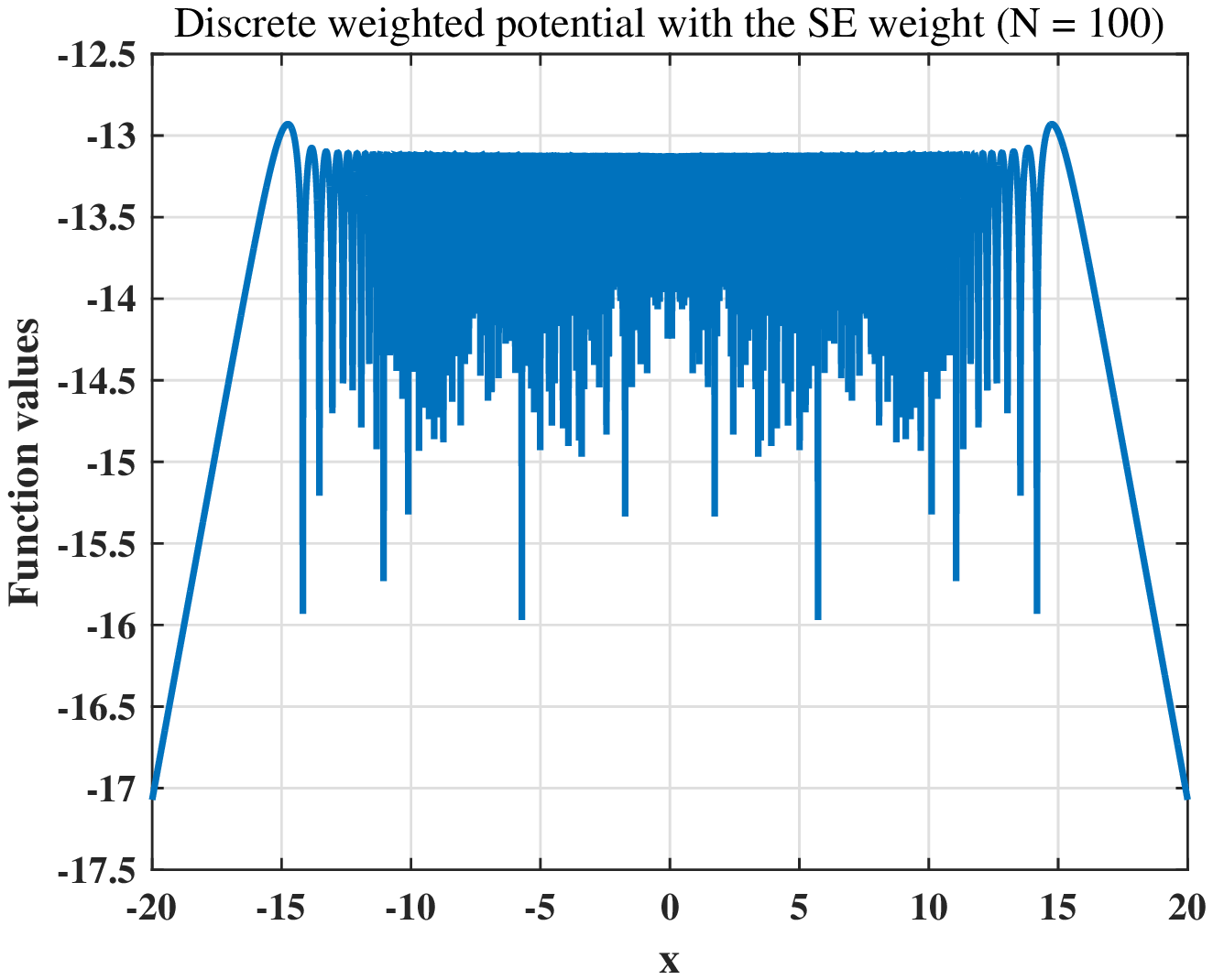}
(d) Discrete weighted potential \eqref{eq:DWP_revisited} for $N=100$
\end{center}
\end{minipage}
\bigskip
\caption{Results for the sampling points for weight (1) with the SE decay}
\label{fig:samp_SE}
\end{figure}

%-----
\begin{figure}[H]
\noindent
\begin{minipage}[t]{0.49\linewidth}
\begin{center}
\includegraphics[width=\linewidth]{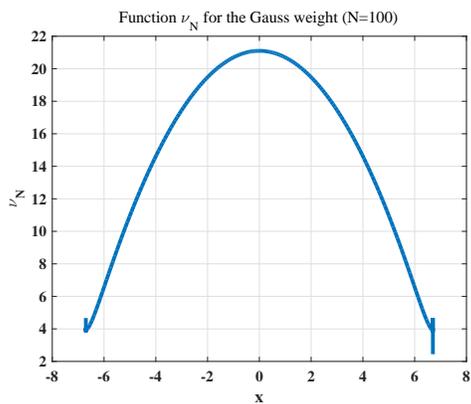}
(a) Function $\tilde{\nu}_{N}$ for $N=100$
\end{center}
\end{minipage}
\begin{minipage}[t]{0.49\linewidth}
\begin{center}
\includegraphics[width=\linewidth]{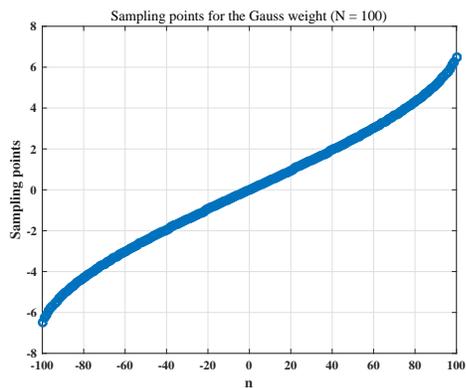}
(b) Sampling points for $N=100$
\end{center}
\end{minipage}

\hspace{5mm}

\noindent
\begin{minipage}[t]{0.49\linewidth}
\begin{center}
\includegraphics[width=\linewidth]{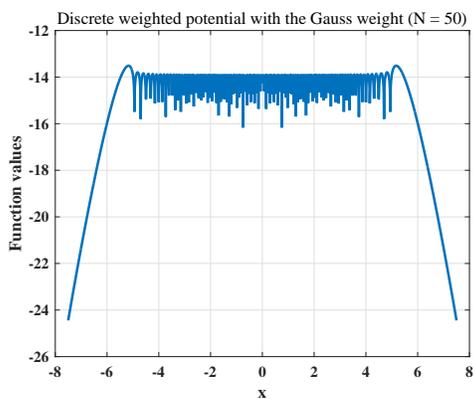}
(c) Discrete weighted potential \eqref{eq:DWP_revisited} for $N=50$
\end{center}
\end{minipage}
\begin{minipage}[t]{0.49\linewidth}
\begin{center}
\includegraphics[width=\linewidth]{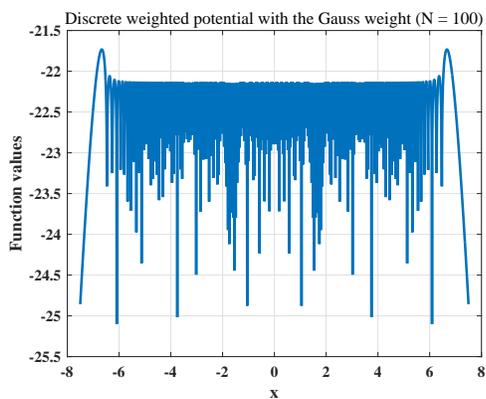}
(d) Discrete weighted potential \eqref{eq:DWP_revisited} for $N=100$
\end{center}
\end{minipage}
\bigskip
\caption{Results for the sampling points for weight (2) with the Gauss decay}
\label{fig:samp_Gauss}
\end{figure}

%-----
\begin{figure}[H]
\noindent
\begin{minipage}[t]{0.49\linewidth}
\begin{center}
\includegraphics[width=\linewidth]{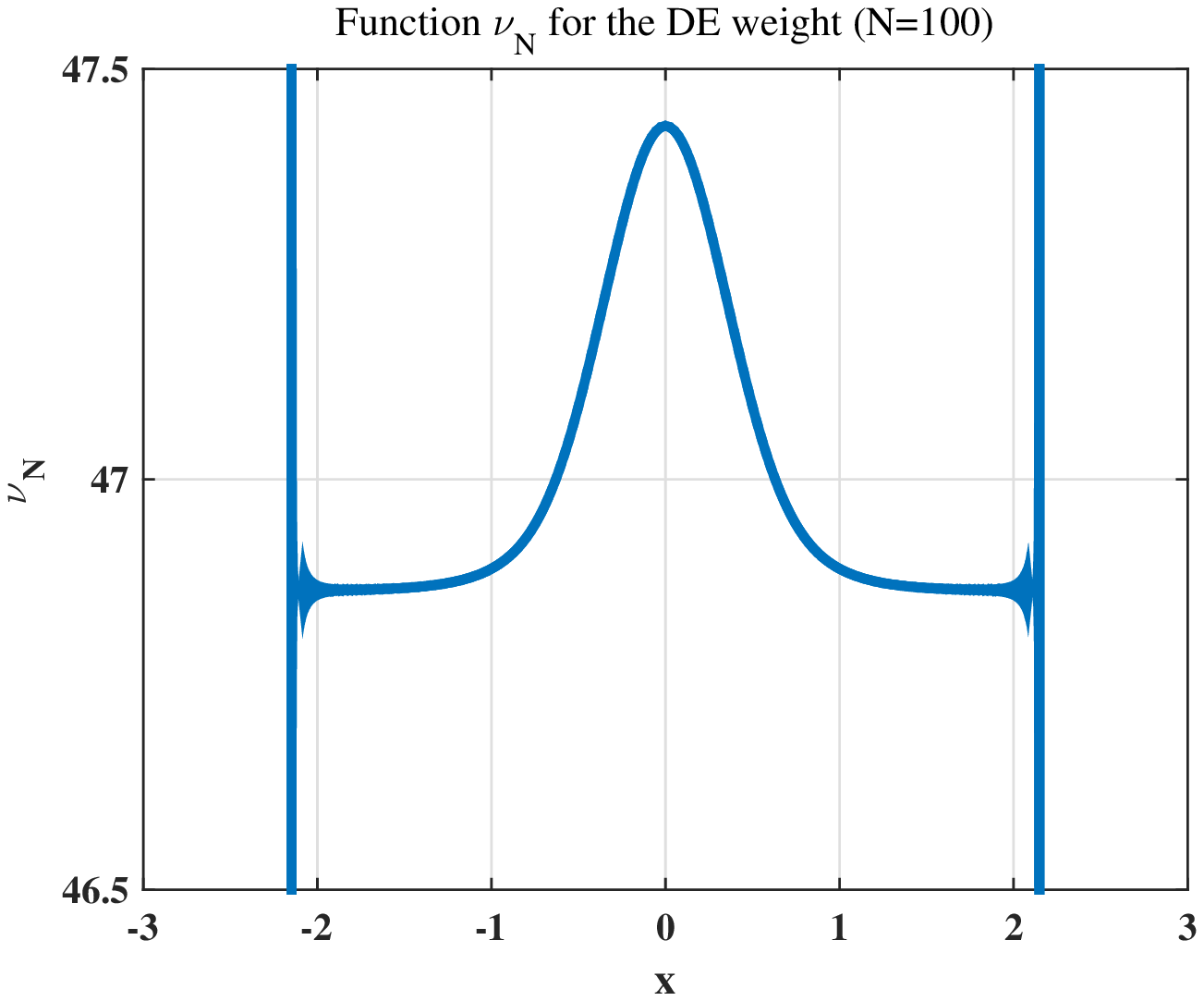}
(a) Function $\tilde{\nu}_{N}$ for $N=100$
\end{center}
\end{minipage}
\begin{minipage}[t]{0.49\linewidth}
\begin{center}
\includegraphics[width=\linewidth]{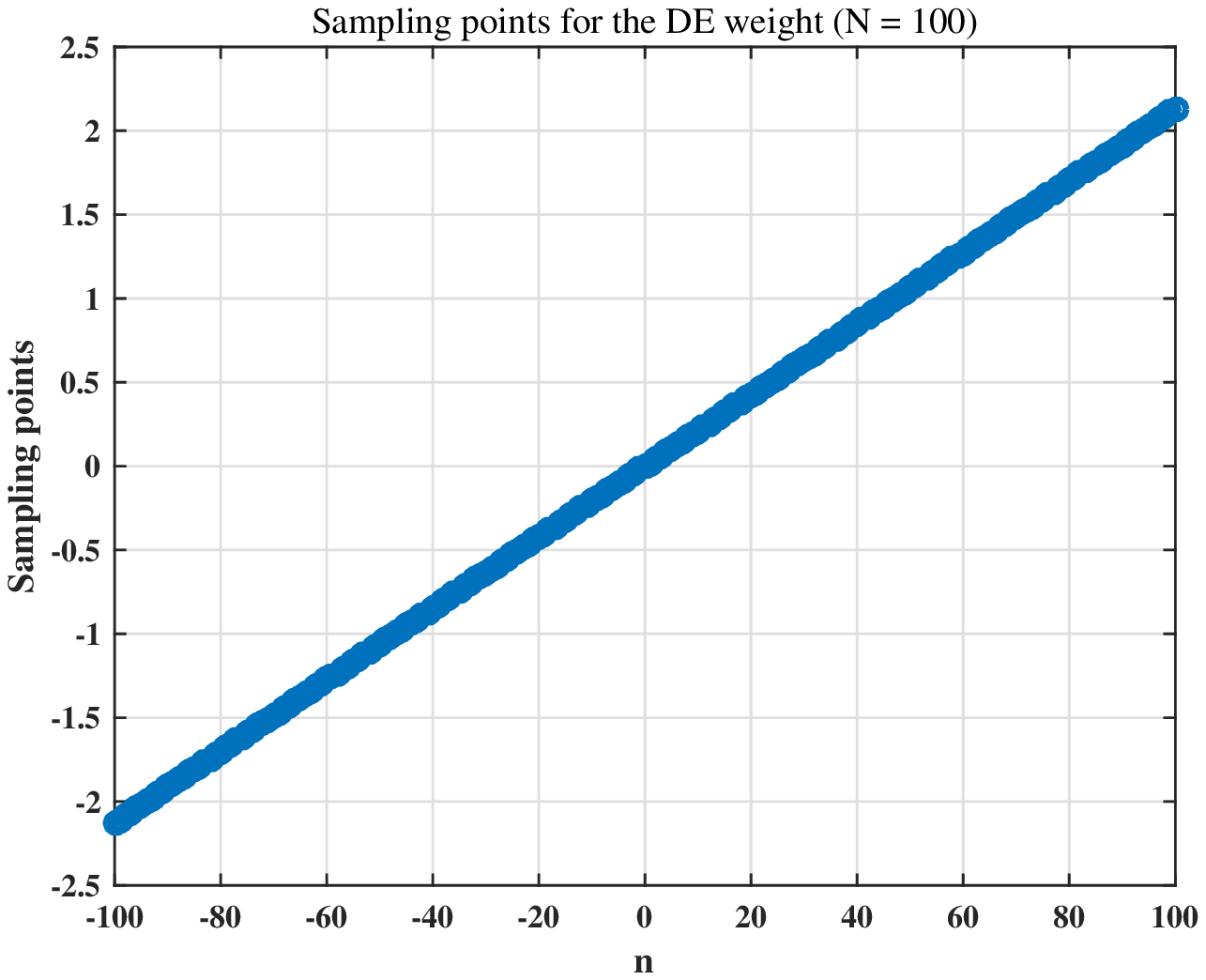}
(b) Sampling points for $N=100$
\end{center}
\end{minipage}

\hspace{5mm}

\noindent
\begin{minipage}[t]{0.49\linewidth}
\begin{center}
\includegraphics[width=\linewidth]{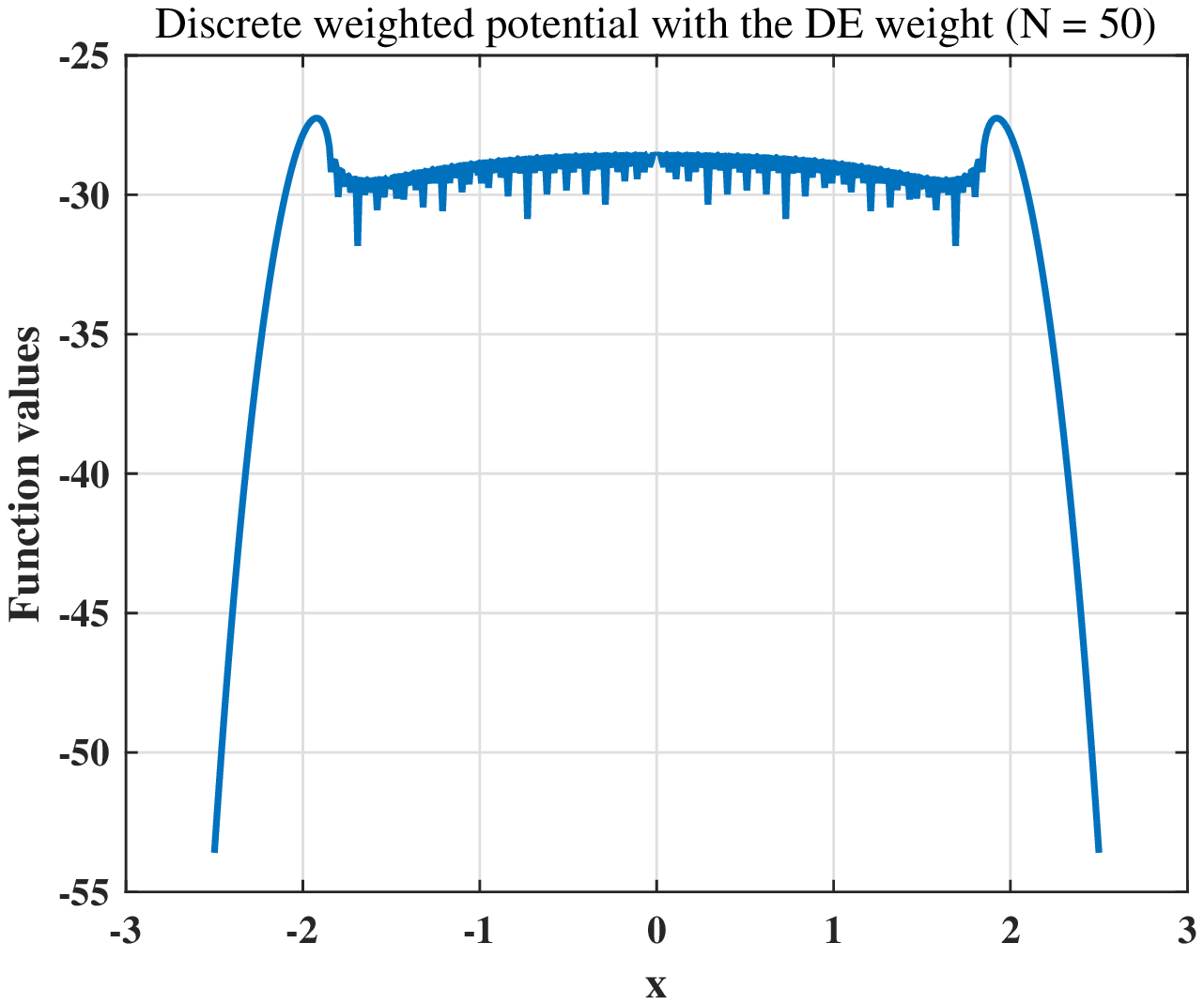}
(c) Discrete weighted potential \eqref{eq:DWP_revisited} for $N=50$
\end{center}
\end{minipage}
\begin{minipage}[t]{0.49\linewidth}
\begin{center}
\includegraphics[width=\linewidth]{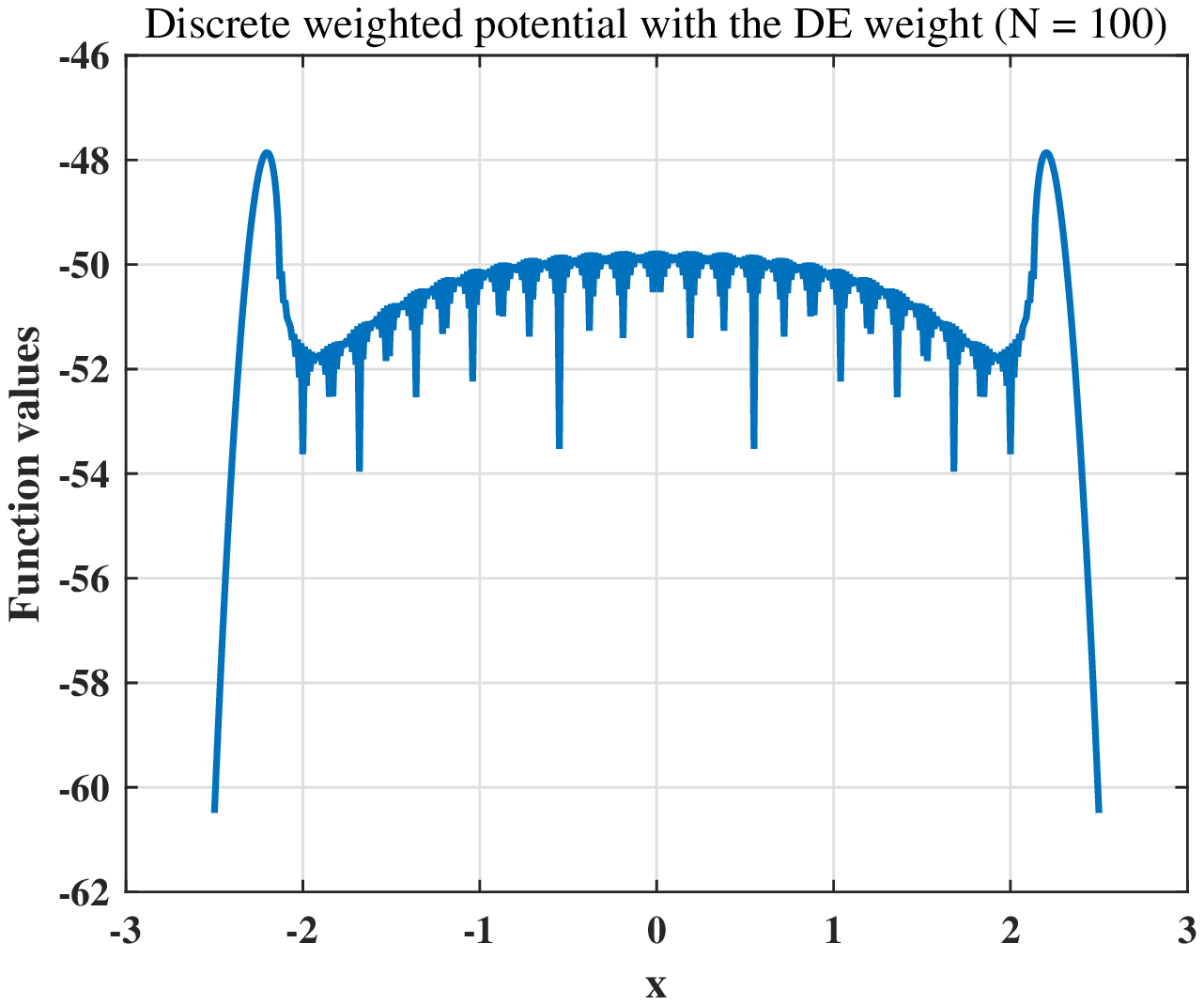}
(d) Discrete weighted potential \eqref{eq:DWP_revisited} for $N=100$
\end{center}
\end{minipage}
\bigskip
\caption{Results for the sampling points for weight (3) with the DE decay}
\label{fig:samp_DE}
\end{figure}

%--------------------------------------------------
\subsection{Results of function approximations}
\label{sec:approx_results}

%In order to compare with our formula, 
For comparison with our formula, 
we also computed the errors of the approximations 
using the SE-Sinc formulas and Ganelius's formula in the case of the (1) SE and (2) Gaussian functions, 
and those using the DE-Sinc formula in the case of the (3) DE function. 
For the sinc formulas, 
we follow the convention in \cite{bib:Sugihara_NearOpt_2003}, 
and use formula~\eqref{eq:ex_f_sinc} with 
\(
h = \pi/(2\sqrt{2N})
\),
\(
h = (\pi/(2N))^{2/3}
\), and
\(
h = \log(2 \pi N)/(2N)
\)
for (1) SE, (2) Gaussian, and (3) DE functions, respectively. 
Furthermore, as Ganelius's formula, we use formula~\eqref{eq:Ganelius_formula} with $\beta = 2$. 

We adopted ten different values for $N$ as $N = 10, 20, \ldots, 100$, 
and computed the approximations of the functions in Table~\ref{tab:func_weight} for $x = x_{\ell}$ given by
\begin{align}
x_{\ell} =
\begin{cases}
-20+0.04 \ell & ((1)\ \text{SE}), \\
-10+0.02 \ell & ((2)\ \text{Gaussian}), \\
-2.5+0.005 \ell & ((3)\ \text{DE})
\end{cases}
\end{align}
for $\ell = 0,1,\ldots, 1000$. 
Then, we computed the values $\max_{\ell} |f(x_{\ell}) - \tilde{f}_{N}(x_{\ell})|/\|f\|$,  
which are presented in Figures \ref{fig:func_SE}, \ref{fig:func_Ga}, and \ref{fig:func_DE} 
for the functions with SE, Gaussian, and DE decays, respectively. 
The computations of the approximations were performed 
using Matlab R2015a programs with multi-precision numbers, 
whose digits were 30, 50, and 90 for the functions with SE, Gaussian, and DE decays, respectively. 
For the multi-precision numbers, 
we used the Multiprecision Computing Toolbox for Matlab, produced by Advanpix (\url{http://www.advanpix.com}).

In the case of the SE weight, 
Ganelius's formula and our formula achieve almost the same accuracy, which surpass that of the SE-Sinc formula. 
The former result supports the observation in Example~\ref{ex:SE_error} that the error estimates of these formulas almost coincide. 
In the other cases, our formula outperforms the other formulas.

\begin{figure}[H]
\begin{center}
\includegraphics[width=0.6\linewidth]{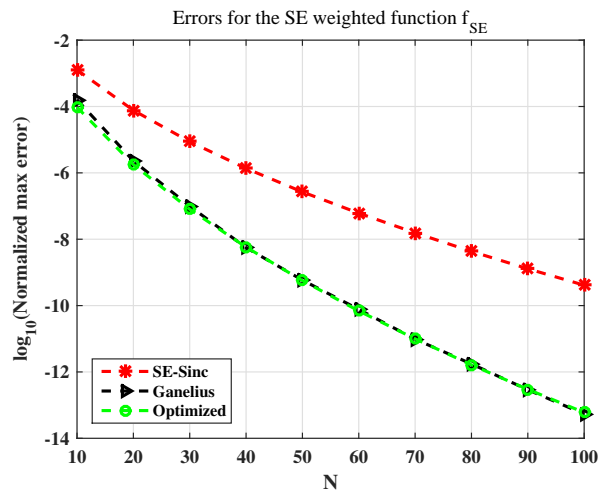}
\end{center}
\caption{Errors for the function with SE decay (1). }
\label{fig:func_SE}
\end{figure}

\begin{figure}[H]
\begin{center}
\includegraphics[width=0.6\linewidth]{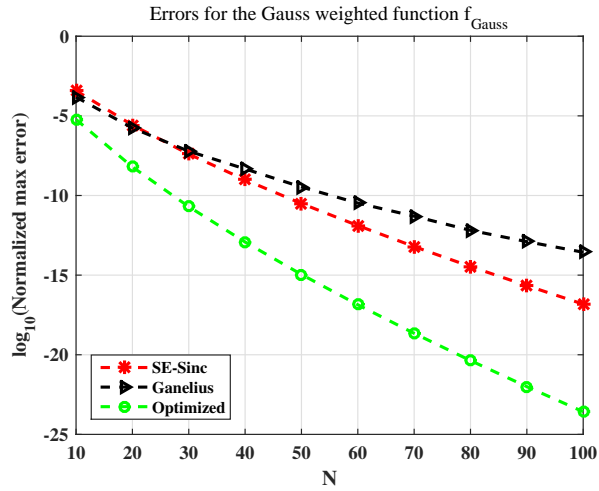}
\end{center}
\caption{Errors for the function with Gaussian decay (2). }
\label{fig:func_Ga}
\end{figure}

\begin{figure}[H]
\begin{center}
\includegraphics[width=0.6\linewidth]{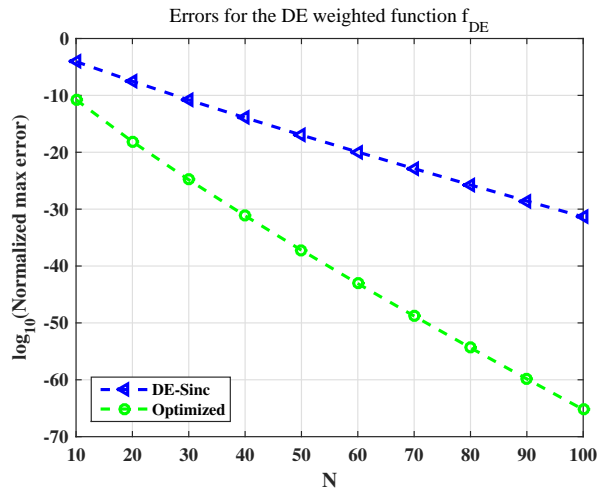}
\end{center}
\caption{Errors for the function with DE decay (3). }
\label{fig:func_DE}
\end{figure}

%------------------------------
\section{Concluding remarks}
\label{sec:Concl}

In this paper, we proposed a method for designing 
accurate function approximation formulas on weighted Hardy spaces $\boldsymbol{H}^{\infty}(\mathcal{D}_{d}, w)$
for weight functions fulfilling Assumptions \ref{assump:w}--\ref{assump:w_convex}.
We began with Problem~\ref{prob:Original}, which is 
the worst error minimization problem given by \eqref{eq:main_opt_equiv_log} 
to determine sampling points for the formulas. 
We approximately reduced this to 
Problem~\ref{prob:ContMeas} with a general measure $\mu_{N}$ for analytical tractability. 
According to potential theory,  
solutions of Problem~\ref{prob:ContMeas} are characterized by the system consisting of 
the integral equation \eqref{eq:char_opt_eq} and the integral inequality \eqref{eq:char_opt_ineq}. 
Next, we considered Problem~\ref{prob:IE_IINEQ} 
by introducing the measure $\mu[\nu_{N}]$ with the smooth density $\nu_{N}$
in place of the measure $\mu_{N}$ again for analytical tractability. 
Finally, using the harmonic property of the Green potential $U_{\mathcal{D}_{d}}^{\mu[\nu_{N}]}$, 
we considered Problem~\ref{prob:2Steps} as a reformulation of Problem~\ref{prob:IE_IINEQ} 
and obtained the Fourier transform \eqref{eq:FT_nu_N} of the approximate solution for the density $\nu_{N}$ of the measure $\mu[\nu_{N}]$. 
After determining the unknown parameters $\alpha_{N}$ and $K_{N}$ 
in the Fourier transform by \eqref{eq:det_alpha} and \eqref{eq:det_K}, 
we obtained an approximation of the density $\nu_{N}$. 
Then, using its discretization, we generated the sampling points and proposed the approximation formula~\eqref{eq:final_approx_formula}
for each space $\boldsymbol{H}^{\infty}(\mathcal{D}_{d}, w)$. 
Furthermore, 
we provided an error estimate for the proposed formulas in Theorem~\ref{thm:general_error_estimate} and 
observed that in numerical experiments our formulas outperformed the existing formulas. 

However, our procedure for generating the sampling points contains approximations 
in the reduction of Problem~\ref{prob:Original} to Problem~\ref{prob:ContMeas} and
in the approximate solution of the Dirichlet problem given by \eqref{eq:Laplace}--\eqref{eq:Laplace_outer_bnd}
for SP1 in Problem~\ref{prob:2Steps}. 
Therefore, we cannot guarantee that each of our formulas is precisely optimal 
in the corresponding space $\boldsymbol{H}^{\infty}(\mathcal{D}_{d}, w)$. 
Then, one possible direction for future work is the improvement of the procedure to obtain exactly optimal formulas.
Furthermore, another future direction may be generalization of the proposed method
such as a generalization of the domain $G = \mathcal{D}_{d}$ and the closed set $E = \mathbf{R}$
and/or 
a generalization of the weight function $w$. 
In particular, as the latter generalization, 
weight functions with complex singularities will be of our interest. 

Finally, we make two remarks about the computational aspects of the proposed formula in \eqref{eq:final_approx_formula}. 
First, 
the form of formula \eqref{eq:final_approx_formula} suggests
an $\mathrm{O}(N^{2})$ computation for evaluation at a fixed $x$, 
which is also the case of Lagrange interpolation of a polynomial. 
However, 
the complexity of Lagrange interpolation is reduced to $\mathrm{O}(N)$ by the barycentric formula 
\citep{
bib:barycent_BerrutTrefethen2004, 
bib:barycent_Higham2004}. 
Since formula \eqref{eq:final_approx_formula} has a similar form to a Lagrange interpolant, 
a certain analogue of the barycentric formula may reduce its complexity. 
Second, 
there remains the issue of the numerical stability of formula \eqref{eq:final_approx_formula}. 
The numerical stability of computing the product in $B_{N}(x; \{ a_{\ell} \} , \mathcal{D}_{d})$ 
is not clear because we used multi-precision arithmetic for the computations in Section \ref{sec:approx_results}.
We need to investigate the stability and to modify the proposed formula to stabilize it if necessary. 
A new barycentric formula mentioned above may be a possible candidate that gives such modification. 
We regard these computational issues as important topics of future work. 

%----------
\section*{Acknowledgements}

The authors would like to give thanks to the editors and anonymous referees 
for their valuable comments and suggestions.

%------------------------------
% Bibliography

%------------------------------
% Appendix
\appendix

\section{Proofs}
\label{sec:Proofs}

\subsection{Sketch of the proof of Proposition \ref{prop:Sugihara2003}}
\label{sec:proof_opt_err_norm}

First, we show the inequalities
\begin{align}
& E_{N}^{\mathrm{min}}(\boldsymbol{H}^{\infty}(\mathcal{D}_{d}, w)) \notag \\
& \leq
\inf_{\begin{subarray}{c} a_{\ell} \in \mathbf{R} \\ \text{distinct} \end{subarray}}
\left[
\sup_{\| f \| \leq 1}
\sup_{x \in \mathbf{R}} 
\left|
f(x) - \hspace{-1mm} \sum_{k = -N}^{N} \hspace{-1mm} f(a_{k}) 
\frac{ B_{N;k}(x; \{a_{\ell}\}, \mathcal{D}_{d}) w(x) }{B_{N;k}(a_{k}; \{a_{\ell}\}, \mathcal{D}_{d}) w(a_{k})} \frac{4d}{\pi} T_{d}'(a_{k} - x)
\right|
\right]
\notag \\
& \leq 
\inf_{a_{\ell} \in \mathbf{R}}
\left[
\sup_{x \in \mathbf{R}} 
\left|
B_{N}(x; \{a_{\ell}\} , \mathcal{D}_{d})\, w(x)
\right|
\right]. 
\label{eq:MinErrBlaschke_upper} 
\end{align}
The first inequality is trivial by the definition of $ E_{N}^{\mathrm{min}}(\boldsymbol{H}^{\infty}(\mathcal{D}_{d}, w))$
in \eqref{eq:def_E_min}. 
For the second inequality, 
we use residue analysis to obtain that 
\begin{align}
& 
\left|
f(x) - \hspace{-1mm} \sum_{k = -N}^{N} \hspace{-1mm} f(a_{k}) 
\frac{ B_{N;k}(x; \{a_{\ell}\}, \mathcal{D}_{d}) w(x) }{B_{N;k}(a_{k}; \{a_{\ell}\}, \mathcal{D}_{d}) w(a_{k})} \frac{4d}{\pi} T_{d}'(a_{k} - x) 
\right| \notag \\
& = 
\left|
\frac{1}{2\pi\, \i}
\int_{\partial \mathcal{D}_{d}} 
f(z)
\frac{B_{N}(x; \{a_{\ell}\}, \mathcal{D}_{d}) w(x)}{B_{N}(z; \{a_{\ell}\}, \mathcal{D}_{d}) w(z)} \, 
\frac{T_{d}'(z - x)}{T_{d}(z - x)} \, \d z 
\right| \notag \\
& \leq 
\left|
B_{N}(x; \{a_{\ell}\}, \mathcal{D}_{d}) w(x)
\right| \, 
\frac{1}{2\pi}
\int_{\partial \mathcal{D}_{d}} 
\left|
\frac{f(z)}{w(z)}
\right| \, 
\frac{1}{| B_{N}(z; \{a_{\ell}\}, \mathcal{D}_{d}) |}\, 
\left| 
\frac{T_{d}'(z - x)}{T_{d}(z - x)}  
\right|\, |\d z| \notag \\
& \leq 
\left|
B_{N}(x; \{a_{\ell}\}, \mathcal{D}_{d}) w(x)
\right| \, 
\frac{1}{2\pi}
\int_{\partial \mathcal{D}_{d}} 
\left| 
\frac{T_{d}'(z - x)}{T_{d}(z - x)}  
\right|\, |\d z| \notag \\
& = 
\left|
B_{N}(x; \{a_{\ell}\}, \mathcal{D}_{d}) w(x)
\right| 
\end{align}
for $f \in \boldsymbol{H}^{\infty}(\mathcal{D}_{d}, w)$ with $\| f \| \leq 1$ and $x \in \mathbf{R}$. 
Then, we have the second inequality. 

Next, we show the inequality
\begin{align}
E_{N}^{\mathrm{min}}(\boldsymbol{H}^{\infty}(\mathcal{D}_{d}, w))
\geq
\inf_{a_{\ell} \in \mathbf{R}}
\left[
\sup_{x \in \mathbf{R}} 
\left|
B_{N}(x; \{a_{\ell}\} , \mathcal{D}_{d})\, w(x)
\right|
\right]. 
\label{eq:MinErrBlaschke_lower} 
\end{align}
In order to show this inequality, 
we consider the subspace $\boldsymbol{F}_{0}(\{ a_{\ell} \}, \{ m_{\ell} \})$ defined by 
\begin{align}
& \boldsymbol{F}_{0}(\{ a_{\ell} \}, \{ m_{\ell} \}) 
=
\{ 
f \in \boldsymbol{H}^{\infty} (\mathcal{D}_{d}, w) 
\mid
\| f \| \leq 1, \ 
\forall \ell: \ f^{(k)}(a_{\ell}) = 0\ (k = 0,1,\ldots, m_{\ell} - 1) 
\}. 
\end{align}
By the definition of $\boldsymbol{F}_{0}(\{ a_{\ell} \}, \{ m_{\ell} \})$, 
any interpolant of $f$ at $\{ a_{\ell} \}$ vanishes if $f \in \boldsymbol{F}_{0}(\{ a_{\ell} \}, \{ m_{\ell} \})$. 
Therefore, we have that 
\begin{align}
E_{N}^{\mathrm{min}}(\boldsymbol{H}^{\infty}(\mathcal{D}_{d}, w))
& \geq 
\inf_{1 \leq l \leq N} 
\inf_{\begin{subarray}{c} m_{-l}, \ldots , m_{l} \\ m_{-l}+\cdots+m_{l} = 2N+1 \end{subarray}}
\inf_{\begin{subarray}{c} a_{\ell} \in \mathcal{D}_{d} \\ \text{distinct} \end{subarray}}
\left[ 
\sup_{f \in \boldsymbol{F}_{0}(\{ a_{\ell} \}, \{ m_{\ell} \})} 
\sup_{x \in \mathbf{R}}|f(x)| 
\right] \notag \\
& \geq 
\inf_{a_{\ell} \in \mathcal{D}_{d}}
\left[
\sup_{x \in \mathbf{R}} 
\left|
B_{N}(x; \{a_{\ell}\} , \mathcal{D}_{d})\, w(x)
\right|
\right] \notag \\
& \geq 
\inf_{a_{\ell} \in \mathbf{R}}
\left[
\sup_{x \in \mathbf{R}} 
\left|
B_{N}(x; \{a_{\ell}\} , \mathcal{D}_{d})\, w(x)
\right|
\right]
\end{align}
because 
$B_{N}(\, \cdot\, ; \{a_{\ell}\} , \mathcal{D}_{d})\, w(\, \cdot\, ) \in \boldsymbol{F}_{0}(\{ a_{\ell} \}, \{ m_{\ell} \})$
and 
\begin{align}
| \tanh[C_{d} (x - a)] | \geq | \tanh[C_{d} (x - (\mathop{\mathrm{Re}} a))] |
\end{align}
for any $a \in \mathcal{D}_{d}$, where $C_{d} = \pi/(4d)$. 
Hence we have \eqref{eq:MinErrBlaschke_lower}. 

%-----
\subsection{Proof of Proposition \ref{prop:smooth_V}}
\label{sec:proof_smooth_V}

For simplicity, 
we set $C_{d} = \pi/(4d)$, and 
we use $\alpha$, $\nu$, and $V$ in place of $\alpha_{N}$, $\nu_{N}$, and $V_{\mathcal{D}_{d}}^{\mu[\nu_{N}]}$, respectively. 
First, we investigate the function $V$ given by
\begin{align}
V(x) = \int_{-\alpha}^{\alpha} \log |\tanh(C_{d}(x-z))|\, \nu(z)\, \d z, 
\end{align}
where $\nu \in C^{1}(-\alpha, \alpha)$ and $\nu(\pm \alpha) = 0$. 
In the following, we will prove that
\begin{align}
V'(x) 
=
\int_{-\alpha}^{\alpha} \log |\tanh(C_{d}(x-z))|\, \nu'(z)\, \d z
\label{eq:V_prime}
\end{align}
for $x$ with $|x| \neq \alpha$. 
It suffices to show that formula \eqref{eq:V_prime} holds,
because this also guarantees the relation $V'(\pm \alpha - 0) = V'(\pm \alpha + 0)$. 

In the case $|x| > \alpha$, 
we can derive formula~\eqref{eq:V_prime} using a standard argument from calculus and integration by parts. 
In the case $|x| < \alpha$, 
we consider the interval $I_{\delta} = [-\alpha + \delta, \alpha - \delta]$ for $\delta$ with $0 < \delta < \alpha$, 
and define $V_{\varepsilon}(x)$ for $x \in I_{\delta}$ by
\begin{align}
V_{\varepsilon}(x)
= 
\left(\int_{-\alpha}^{x - \varepsilon} + \int_{x + \varepsilon}^{\alpha} \right) 
\log |\tanh(C_{d}(x-z))|\, \nu(z)\, \d z, 
\end{align}
where $0 < \varepsilon < \delta$. 
Clearly, $\lim_{\varepsilon \to 0} V_{\varepsilon}(x) = V(x)$ holds for any $x \in I_{\delta}$. 
Then, if we can show that $V_{\varepsilon}'$ is continuous on $I_{\delta}$ and 
\begin{align}
\lim_{\varepsilon \to 0} V_{\varepsilon}'(x) = \int_{-\alpha}^{\alpha} \log |\tanh(C_{d}(x-z))|\, \nu'(z)\, \d z
\label{eq:V_eps_prime_limit}
\end{align}
uniformly with respect to $x \in I_{\delta}$, 
we have that \eqref{eq:V_prime} holds for  $x \in I_{\delta}$. 
The derivative of $V_{\varepsilon}$ is given by
\begin{align}
V_{\varepsilon}'(x) 
= &
\left(\int_{-\alpha}^{x - \varepsilon} + \int_{x + \varepsilon}^{\alpha} \right) 
\frac{2C_{d}}{\sinh(2C_{d}(x-z))}\, \nu(z)\, \d z \notag \\
& + 
\log |\tanh(C_{d} \varepsilon )|\, ( \nu(x - \varepsilon) - \nu(x + \varepsilon) ) \notag \\
= &
\log \left| \tanh(C_{d}(x+\alpha)) \right| \nu(-\alpha) 
-\log \left| \tanh(C_{d}(x-\alpha)) \right| \nu(\alpha) \notag \\
& + 
\left(\int_{-\alpha}^{x - \varepsilon} + \int_{x + \varepsilon}^{\alpha} \right) 
\log \left| \tanh(C_{d}(x-z)) \right| \nu'(z)\, \d z \notag \\
= &
\left(\int_{-\alpha}^{\alpha} - \int_{x - \varepsilon}^{x + \varepsilon} \right) 
\log \left| \tanh(C_{d}(x-z)) \right| \nu'(z)\, \d z, 
\label{eq:V_eps_prime}
\end{align}
which is continuous on $I_{\delta}$. 
We have used integration by parts for the second equality in \eqref{eq:V_eps_prime}. 
Because $\nu'$ is continuous on $I_{\delta}$, 
it has a minimum value $\Phi_{\mathrm{min}}$ and the maximum value $\Phi_{\mathrm{max}}$ on $I_{\delta}$. 
Using these, we have
\begin{align}
& - \Phi_{\mathrm{min}} 
\int_{x - \varepsilon}^{x + \varepsilon} \log \left| \tanh(C_{d}(x-z)) \right| \, \d z \notag \\
& \leq
- \int_{x - \varepsilon}^{x + \varepsilon} 
\log \left| \tanh(C_{d}(x-z)) \right| \nu'(z)\, \d z \notag \\
& \leq
- \Phi_{\mathrm{max}} 
\int_{x - \varepsilon}^{x + \varepsilon} \log \left| \tanh(C_{d}(x-z)) \right| \, \d z. 
\label{eq:log_tanh_ineq}
\end{align}
Furthermore, we can show that 
\begin{align}
\int_{x - \varepsilon}^{x + \varepsilon} \log \left| \tanh(C_{d}(x-z)) \right| \, \d z
= 
\int_{- \varepsilon}^{\varepsilon} \log \left| \tanh(C_{d} t) \right| \, \d t
\end{align}
is independent of $x$ and tends to zero as $\varepsilon \to 0$. 
Then, it follows from \eqref{eq:V_eps_prime} and \eqref{eq:log_tanh_ineq} 
that the uniform convergence \eqref{eq:V_eps_prime_limit} on $I_{\delta}$ holds. 
Because $\delta$ is arbitrary in $(0, \alpha)$, we have that \eqref{eq:V_prime} holds for $x$ with $|x| < \alpha$. 

%-----
\subsection{Proof of Lemma \ref{lem:GrPotDiscError_mod}}
\label{sec:proof_GrPotDiscError_mod}

In this proof as well, we will use $C_{d} = \pi/(4d)$ for simplicity.

\bigskip

\noindent
\renewcommand{\proof}{{\it Proof in the case $x \in [a(-N), a(N)]$.\ }}
\begin{proof}%[\text{\it Proof in the case $x \in [a(-N), a(N)]$}]
Let $m$ be the integer with $x \in [a(m), a(m+1))$ and 
let $n$ be an integer with $n \leq m$.
Noting that $\log | \tanh(C_{d}(x - z)) |$ is a monotone decreasing function of $z$ with $z < x$, 
for $z \in [a(n-1), a(n))$, we have that
\begin{align}
\log | \tanh(C_{d}(x - a(n))) | \leq \log | \tanh(C_{d}(x - z)) |. \label{eq:LTz_L}
\end{align}
Because we have that 
\begin{align}
\int_{a(l)}^{a(l+1)} b'(z)\, \d z = b(a(l+1)) - b(a(l)) = (l+1) - l = 1,
\label{eq:int_b_1}
\end{align}
we can multiply both sides of \eqref{eq:LTz_L} by $b'(z)$ and 
integrate them with respect to $z$, to obtain
\begin{align}
\log | \tanh(C_{d}(x - a(n))) | 
\leq 
\int_{a(n-1)}^{a(n)} \log | \tanh(C_{d}(x - z)) | \, b'(z)\, \d z.
\end{align}
Summing these terms for $n = -N, \ldots, m$, we have that
\begin{align}
\sum_{n=-N}^{m} \log | \tanh(C_{d}(x - a(n))) | 
\leq 
\int_{a(-N-1)}^{a(m)} \log | \tanh(C_{d}(x - z)) | \, b'(z)\, \d z.
\label{eq:sum_Left}
\end{align}
By noting that $\log | \tanh(C_{d}(x - z)) |$ is a monotone increasing function of $z$ with $x < z$, 
and applying similar arguments to those used above, we also have that
\begin{align}
\sum_{n=m+1}^{N} \log | \tanh(C_{d}(x - a(n))) | 
\leq 
\int_{a(m+1)}^{a(N+1)} \log | \tanh(C_{d}(x - z)) | \, b'(z)\, \d z. 
\label{eq:sum_Right}
\end{align}
Therefore, by combining \eqref{eq:sum_Left} and \eqref{eq:sum_Right}, we obtain
\begin{align}
& \sum_{n=-N}^{N} \log | \tanh(C_{d}(x - a(n))) | \notag \\
& \leq 
\int_{a(-N-1)}^{a(N+1)} \log | \tanh(C_{d}(x - z)) | \, b'(z)\, \d z 
- \left( \int_{a(m)}^{a(m+1)} \right). 
\label{eq:SumLT_Upper_2}
\end{align}
Here, we partition the second term of the RHS in \eqref{eq:SumLT_Upper_2} as
\begin{align}
& \int_{a(m)}^{a(m+1)} \log | \tanh(C_{d}(x - z)) | \, b'(z)\, \d z \notag \\
& = 
\left(
\int_{
\begin{subarray}{c}
z: |x-z|<N^{-\lambda} \\
z \in [a(m), a(m+1))
\end{subarray}
} 
+
\int_{
\begin{subarray}{c}
z: |x-z| \geq N^{-\lambda} \\
z \in [a(m), a(m+1))
\end{subarray}
} \right) \log | \tanh(C_{d}(x - z)) | \, b'(z)\, \d z.
\label{eq:SumLT_Upper_2_sep}
\end{align}
To estimate the first term in \eqref{eq:SumLT_Upper_2_sep}, 
we let $t_{d}$ denote the positive solution of $(C_{d}/2) t = \tanh (C_{d} t)$
and set $N_{d} = \lceil (1/t_{d})^{1/\lambda} \rceil$. 
Then, $0 \leq (C_{d}/2) t \leq \tanh (C_{d} t)$ holds for $t$ with $0 \leq t \leq t_{d}$. 
Therefore, for $N$ with $N \geq N_{d}$, we have 
\begin{align}
& \int_{
\begin{subarray}{c}
z: |x-z|<N^{-\lambda} \\
z \in [a(m), a(m+1)]
\end{subarray}
} 
\log | \tanh(C_{d}(x - z)) | \, b'(z)\, \d z \notag \\
& \geq 
\int_{z: |x-z|<N^{-\lambda}} 
\log | (C_{d}/2) (x - z) | \, b'(z)\, \d z.  
\label{eq:ineq_tanht_t}
\end{align}
Then, from the inequality \eqref{eq:ineq_tanht_t} and the assumption \eqref{eq:b_prime_max}, 
we have that
\begin{align}
& \int_{
\begin{subarray}{c}
z: |x-z|<N^{-\lambda} \\
z \in [a(m), a(m+1)]
\end{subarray}
} 
\log | \tanh(C_{d}(x - z)) | \, b'(z)\, \d z \notag \\
& \geq 
c N^{\lambda} 
\int_{z: |x-z|<N^{-\lambda}} 
\log | (C_{d}/2) (x - z) | \, \d z \notag \\
& =
c N^{\lambda} 
\left(
\int_{z: |x-z|<N^{-\lambda}} 
\log | x - z | \, \d z 
+ \frac{2 \log (C_{d}/2)}{N^{\lambda}} 
\right) \notag \\
& = 
c N^{\lambda} 
\left(
2 \int_{0}^{N^{-\lambda}} 
\log z' \, \d z' 
+ \frac{2 \log (C_{d}/2)}{N^{\lambda}} 
\right) \notag \\
& =
c N^{\lambda} 
\left(
\frac{2 \log N^{-\lambda}}{N^{\lambda}}  - \frac{2}{N^{\lambda}}
+ \frac{2 \log (C_{d}/2)}{N^{\lambda}} 
\right) \notag \\
& = 
- c\, ( 2 \lambda \log N + 2 - 2 \log (C_{d}/2)).  
\label{eq:SumLT_Upper_2_sep_first}
\end{align}
Furthermore, 
for $N$ with $N \geq N_{d}$, 
the second term in \eqref{eq:SumLT_Upper_2_sep} is estimated as
\begin{align}
& \int_{
\begin{subarray}{c}
z: |x-z| \geq N^{-\lambda} \\
z \in [a(m), a(m+1)]
\end{subarray}
} \log | \tanh(C_{d}(x - z)) | \, b'(z)\, \d z \notag \\
& \geq
\log | \tanh(C_{d} N^{-\lambda}) |
\int_{
\begin{subarray}{c}
z: |x-z| \geq N^{-\lambda} \\
z \in [a(m), a(m+1)]
\end{subarray}
} b'(z)\, \d z \notag \\
& \geq
\log ((C_{d}/2) N^{-\lambda})
\int_{
\begin{subarray}{c}
z: |x-z| \geq N^{-\lambda} \\
z \in [a(m), a(m+1)]
\end{subarray}
} b'(z)\, \d z \notag \\
& \geq 
\log ((C_{d}/2) N^{-\lambda})
\int_{a(m)}^{a(m+1)} b'(z)\, \d z 
= - (\lambda \log N - \log (C_{d}/2)). 
\label{eq:SumLT_Upper_2_sep_second}
\end{align}
From
\eqref{eq:SumLT_Upper_2},
\eqref{eq:SumLT_Upper_2_sep},
\eqref{eq:SumLT_Upper_2_sep_first}, and 
\eqref{eq:SumLT_Upper_2_sep_second}, 
we have that
\begin{align}
\sum_{n=-N}^{N} \log | \tanh(C_{d}(x - a(n))) | 
\leq 
& \int_{a(-N-1)}^{a(N+1)} \log | \tanh(C_{d}(x - z)) | \, b'(z)\, \d z \notag \\
& + (2c+1) (\lambda \log N - \log (C_{d}/2)) + 2c, 
\label{eq:SumLT_Upper_final}
\end{align}
which implies \eqref{eq:DiscPotUpper}. 
\end{proof}

\bigskip

\noindent
\renewcommand{\proof}{{\it Proof in the case $x \not \in [a(-N), a(N)]$.\ }}
\begin{proof}%[\text{\it Proof in the case $x \not \in [a(-N), a(N)]$}]
We consider the case that $x > a(N)$ and omit the case that $x < a(-N)$
because the latter can be proved similarly. 
First, we deduce that
\begin{align}
\sum_{n=-N}^{N} \log | \tanh(C_{d}(x - a(n))) | 
\leq 
\int_{a(-N-1)}^{a(N)} \log | \tanh(C_{d}(x - z)) | \, b'(z)\, \d z 
\label{eq:SumLT_Upper_xlarge}
\end{align}
in a similar manner to \eqref{eq:sum_Left}.
Then, in the case that $x \in [a(N), a(N+1)]$, 
we deduce \eqref{eq:SumLT_Upper_final} from \eqref{eq:SumLT_Upper_xlarge}
in a similar manner to 
\eqref{eq:SumLT_Upper_2_sep},
\eqref{eq:SumLT_Upper_2_sep_first}, and
\eqref{eq:SumLT_Upper_2_sep_second}. 
Furthermore, in the case that $x > a(N+1)$, 
we can use the inequality
\begin{align}
& - \int_{a(N)}^{a(N+1)} \log | \tanh(C_{d}(x - z)) | \, b'(z)\, \d z \notag \\
& \leq 
- \int_{a(N)}^{a(N+1)} \log | \tanh(C_{d}(a(N+1) - z)) | \, b'(z)\, \d z,
\end{align}
to reduce this case to the case that $x \in [a(N), a(N+1)]$.
Then, we can deduce \eqref{eq:SumLT_Upper_final} from \eqref{eq:SumLT_Upper_xlarge}. 
Thus, we have shown that \eqref{eq:DiscPotUpper} holds. 
\end{proof}
\renewcommand{\proof}{{\it Proof.\ }}

%-----
\subsection{Proof of Lemma \ref{lem:mu_prime_estim}}
\label{sec:proof_mu_prime_estim}

In this proof, we use $\alpha$ in place of $\tilde{\alpha}_{N}^{\ast}$ for simplicity. 
According to Assumption \ref{assump:nu_max_x_0}, 
it suffices to estimate $|\tilde{\nu}_{N}(0)|$. 
Let $p_{w}$ and $q_{w}$ be defined by
\begin{align}
p_{w}(\omega) & = \frac{1}{\mathrm{i}\, \pi} \int_{-\alpha}^{\alpha} v(x)\, \mathrm{e}^{-\mathrm{i}\, \omega\, x}\, \mathrm{d}x, \\
q_{w}(\omega) & = -\frac{4d\, v(\alpha)}{\pi} \frac{\pi \sin (\alpha \omega) + 2d \omega \cos(\alpha \omega)}{(\pi^{2} + 4d^{2} \omega^{2})}. 
\end{align}
Then, it follows from \eqref{eq:FT_nu_N} that 
\begin{align}
2\pi \, |\tilde{\nu}_{N}(0)|
& = 
\left|
\int_{-\infty}^{\infty} \frac{p_{w}(\omega) + q_{w}(\omega)}{\tanh(d \omega)}\, \mathrm{d}\omega
\right| \notag \\
& \leq
\left|
\int_{-\infty}^{\infty} \frac{p_{w}(\omega)}{\tanh(d \omega)}\, \mathrm{d}\omega
\right|
+
\left|
\int_{-\infty}^{\infty} \frac{q_{w}(\omega)}{\tanh(d \omega)}\, \mathrm{d}\omega
\right|. 
\label{eq:bp_zero_triangle}
\end{align}
Because $v(x)$ is odd, $p_{w}(\omega)$ must be even. 
Therefore, for the first term in \eqref{eq:bp_zero_triangle}, we have 
\begin{align}
\left|
\int_{-\infty}^{\infty} \frac{p_{w}(\omega)}{\tanh(d\, \omega)}\, \mathrm{d}\omega
\right|
& \leq 
2 \left|
\int_{0}^{\infty} \frac{p_{w}(\omega)}{\tanh(d\, \omega)}\, \mathrm{d}\omega
\right|
\leq 
2 (B_{p, 1} + B_{p, 2} ), 
\label{eq:first}
\end{align}
where 
\begin{align}
B_{p, 1}
& = 
\int_{0}^{\pi/(2\alpha)} 
\left|
\frac{p_{w}(\omega)}{\tanh(d\, \omega)}
\right|
\, \mathrm{d}\omega, \label{eq:B_p_1} \\
B_{p, 2}
& = 
\left|
\int_{\pi/(2\alpha)}^{\infty} \frac{p_{w}(\omega)}{\tanh(d\, \omega)}\, \mathrm{d}\omega
\right|. \label{eq:B_p_2}
\end{align}
For the second term in \eqref{eq:bp_zero_triangle}, we have
\begin{align}
& \left|
\int_{-\infty}^{\infty} \frac{q_{w}(\omega)}{\tanh(d\, \omega)}\, \mathrm{d}\omega
\right| 
\leq
\left|
\frac{4d\, v(\alpha)}{\pi} 
\right| ( B_{q,1} + B_{q,2} )
\label{eq:second}
\end{align}
where 
\begin{align}
B_{q,1} & =  
\int_{-\infty}^{\infty}
\frac{\pi}{\pi^{2} + 4d^{2} \omega^{2}}
\left|
\frac{\sin (\alpha \omega) }{\tanh(d\, \omega)}
\right|
\, \mathrm{d}\omega, \label{eq:B_q_1} \\
B_{q,2} & =  
\left|
\int_{-\infty}^{\infty}
\frac{2d}{\pi^{2} + 4d^{2} \omega^{2}}
\frac{\omega \cos(\alpha \omega)}{\tanh(d\, \omega)}
\, \mathrm{d}\omega
\right|. \label{eq:B_q_2}
\end{align}
In the following, we provide estimates of 
$B_{p,1}$, 
$B_{p,2}$, 
$B_{q,1}$, and
$B_{q,2}$. 

\bigskip
\noindent
{\it Estimate of $B_{p,1}$ in \eqref{eq:B_p_1}.}
Because $v$ is odd, we have that
\begin{align}
p_{w}(\omega)
=
- \frac{1}{\pi} \int_{-\alpha}^{\alpha} v(x) \sin(\omega\, x)\, \mathrm{d}x. 
\end{align}
Furthermore, 
from Assumptions~\ref{assump:w_even} and~\ref{assump:w_convex},  
we have $v(x) \leq 0$ for $x \geq 0$, and $v(x) \geq 0$ for $x \leq 0$. 
Then, we have that
\begin{align}
p_{w}'(\omega)
& =
- \frac{1}{\pi} \int_{-\alpha}^{\alpha} x\, v(x) \cos(\omega\, x)\, \mathrm{d}x
\geq 0 \quad (0 \leq \omega \leq \pi/(2\alpha)),
\notag \\
p_{w}''(\omega)
& =
+ \frac{1}{\pi} \int_{-\alpha}^{\alpha} x^{2} v(x) \sin(\omega\, x)\, \mathrm{d}x
\leq 0 \quad (0 \leq \omega \leq \pi/\alpha).
\notag 
\end{align}
From these, we have that
$0 \leq p_{w}(\omega) \leq p_{w}'(0)\, \omega$ for $0 \leq \omega \leq \pi/(2\alpha)$. 
Therefore, we have that
\begin{align}
0 \leq \frac{p_{w}(\omega)}{\tanh(d\, \omega)} \leq \frac{p_{w}'(0)\, \omega}{\tanh(d\, \omega)}
\label{eq:ineq_linear}
\end{align}
for $0 < \omega \leq \pi/(2\alpha)$. 
Because it holds that
\begin{align}
\frac{\mathrm{d}}{\mathrm{d} \omega} \left( \frac{\omega}{\tanh(d\, \omega)} \right)
=
\frac{\sinh(2d\, \omega) - 2 d\, \omega}{2 \sinh^{2}(d\, \omega)} \geq 0
\label{eq:tanh_monotone}
\end{align}
for $\omega > 0$, the RHS of \eqref{eq:ineq_linear} is monotone increasing. 
Then, letting $\omega = \pi/(2\alpha)$ in the RHS of \eqref{eq:ineq_linear}, 
we obtain that
\begin{align}
0 \leq \frac{p_{w}(\omega)}{\tanh(d\, \omega)} \leq \frac{p_{w}'(0)\, (\pi/(2\alpha))}{\tanh(\pi d/(2 \alpha))}
\end{align}
for $0 \leq \omega \leq \pi/(2\alpha)$. 
Therefore, we have that
\begin{align}
B_{p,1} 
\leq 
\frac{ (\pi/(2\alpha))^{2} }{\tanh(\pi d/(2 \alpha))} p_{w}'(0) 
=
- \frac{ \pi/(2\alpha)^{2} }{\tanh(\pi d/(2 \alpha))}
\int_{-\alpha}^{\alpha} x\, v(x)\, \mathrm{d}x. 
\label{eq:B_p_1_ineq}
\end{align}

\bigskip
\noindent
{\it Estimate of $B_{p,2}$ in \eqref{eq:B_p_2}.}
Using integration by parts, we have that
\begin{align}
p_{w}(\omega) 
& = 
\frac{1}{\mathrm{i}\, \pi}
\int_{-\alpha}^{\alpha} v(x)\, \mathrm{e}^{-\mathrm{i}\, \omega\, x}\, \mathrm{d}x \notag \\
& =
\frac{1}{\mathrm{i}\, \pi}
\left(
\left[
\frac{v(x)}{-\mathrm{i}\, \omega}\, \mathrm{e}^{-\mathrm{i}\, \omega\, x}
\right]_{x = -\alpha}^{x = \alpha}
-
\int_{-\alpha}^{\alpha} 
\frac{v'(x)}{-\mathrm{i}\, \omega}\, \mathrm{e}^{-\mathrm{i}\, \omega\, x}
\, \mathrm{d}x 
\right)
\notag \\
& =
\frac{1}{\mathrm{i}\, \pi}
\left(
\frac{\mathrm{e}^{-\mathrm{i}\, \omega\, \alpha} + \mathrm{e}^{\mathrm{i}\, \omega\, \alpha}}{-\mathrm{i}\, \omega}\, 
v(\alpha)
-
\int_{-\alpha}^{\alpha} 
\frac{v'(x)}{-\mathrm{i}\, \omega}\, \mathrm{e}^{-\mathrm{i}\, \omega\, x}
\, \mathrm{d}x \right) 
\notag \\
& =
\frac{2 \cos( \alpha \omega)}{\pi \, \omega}\, 
v(\alpha)
-
\frac{1}{\mathrm{i}\, \pi\, \omega^{2}}
\left(
2\i \sin( \alpha \omega)\, v'(\alpha)
+
\int_{-\alpha}^{\alpha} 
v''(x)\, \mathrm{e}^{-\mathrm{i}\, \omega\, x}
\, \mathrm{d}x 
\right),
\label{eq:p_w_int_by_parts}
\end{align}
where we have used the relations $v(-\alpha) = -v(\alpha)$ and $v'(-\alpha) = v'(\alpha)$. 
We consider the integral of the first term of \eqref{eq:p_w_int_by_parts} divided by $\tanh(d\, \omega)$
on $[\pi/(2\alpha), \infty)$. 
The integral can be written in the form
\begin{align}
& \int_{\pi/(2\alpha)}^{\infty}
\frac{2 \cos( \alpha \omega)}{\pi\, \omega\, \tanh(d\, \omega)}\, 
v(\alpha)\, \mathrm{d} \omega
=
v(\alpha)
\sum_{k=1}^{\infty} t_{l}
\end{align}
where
\begin{align}
t_{l} 
= 
\int_{\pi (2l+1)/(2\alpha)}^{\pi (2l+3)/(2\alpha)}
\frac{2 \cos( \alpha \omega)}{\pi\, \omega\, \tanh(d\, \omega)}
\, \mathrm{d} \omega
\qquad 
l = 0,1,2,\ldots.
\end{align}
Because the sequence $\{ t_{l} \}$ satisfies
\begin{align}
& (-1)^{l+1} t_{l} > 0 \qquad l = 0,1,2,\ldots, \notag \\
& | t_{0} | > | t_{1} | > | t_{2} | > \cdots, \notag 
\end{align}
we can apply a well-known method for estimating an alternating series to obtain
\begin{align}
& \left|
\int_{\pi/(2\alpha)}^{\infty}
\frac{2 \cos( \alpha \omega)}{\pi\, \omega\, \tanh(d\, \omega)}\, 
v(\alpha)\, \mathrm{d} \omega
\right|
\leq
| v(\alpha) |\, |t_{0}| \notag \\
& \leq
| v(\alpha) |\, \frac{2}{\pi}\, \frac{1}{(\pi/(2\alpha)) \tanh (\pi d/(2\alpha))} 
\int_{\pi/(2\alpha)}^{3 \pi/(2\alpha)} |\cos (\alpha \omega)|\, \d \omega \notag \\
& = 
\left|
v(\alpha)
\right| 
\frac{2}{(\pi^{2}/(2\alpha)) \tanh(\pi d/(2\alpha))} \frac{2}{\alpha} 
= 
\frac{8\, |v(\alpha)| }{\pi^{2} \, \tanh(\pi d/(2\alpha))}. 
\label{eq:p_w_int_by_parts_term1}
\end{align}
For the second term of \eqref{eq:p_w_int_by_parts}, we have
\begin{align}
& \left|
\int_{\pi/(2\alpha)}^{\infty}
\left[
\frac{-1}{\mathrm{i}\, \pi\, \omega^{2} \tanh(d\, \omega)}
\left(
2\i \sin( \alpha \omega)\, v'(\alpha)
+
\int_{-\alpha}^{\alpha} 
v''(x)\, \mathrm{e}^{-\mathrm{i}\, \omega\, x}
\, \mathrm{d}x 
\right)
\right] \mathrm{d}\omega
\right| \notag \\
& \leq
\frac{2}{\pi \tanh(\pi d/(2\alpha))} 
\left(
|v'(\alpha)| + \alpha \max_{x \in [-\alpha, \alpha] } |v''(x)|
\right)
\int_{\pi/(2\alpha)}^{\infty} \frac{1}{\omega^{2}}\, \mathrm{d}\omega \notag \\
& =
\frac{4 \alpha}{\pi^{2} \tanh(\pi d/(2\alpha))} 
\left(
|v'(\alpha)| + \alpha \max_{x \in [-\alpha, \alpha] } |v''(x)|
\right).
\label{eq:p_w_int_by_parts_term2}
\end{align}
By combining \eqref{eq:p_w_int_by_parts_term1} and \eqref{eq:p_w_int_by_parts_term2}, 
we obtain 
\begin{align}
B_{p,2}
\leq
\frac{4}{\pi^{2} \tanh(\pi d/(2\alpha))} 
\left(
2\, |v(\alpha)| + \alpha \, |v'(\alpha)| + \alpha^{2} \max_{x \in [-\alpha, \alpha] } |v''(x)|
\right).
\label{eq:B_p_2_ineq}
\end{align}

\bigskip
\noindent
{\it Estimate of $B_{q,1}$ in \eqref{eq:B_q_1}.}
For $\omega$ with $0 \leq |\omega| \leq \pi/(2 \alpha)$, it follows from \eqref{eq:tanh_monotone} that
\begin{align}
& \left|
\frac{\sin (\alpha \omega) }{\tanh(d\, \omega)}
\right|
=
\left|
\frac{\sin (\alpha \omega) }{\omega}
\right|
\left|
\frac{\omega}{\tanh(d\, \omega)}
\right| \notag \\
& \leq 
\alpha\, \frac{\pi/(2 \alpha)}{\tanh(\pi d/(2 \alpha))}
= 
\frac{\pi}{2 \tanh(\pi d/(2 \alpha))}. 
\end{align}
Furthermore, for $\omega$ with $|\omega| > \pi/(2 \alpha) $, we have that
\begin{align}
\left|
\frac{\sin (\alpha \omega) }{\tanh(d\, \omega)}
\right|
\leq
\left|
\frac{1}{\tanh(d\, \omega)}
\right|
\leq 
\frac{1}{\tanh(\pi d/(2 \alpha))}. 
\end{align}
Therefore, we have that
\begin{align}
B_{q,1}
\leq 
\frac{\pi}{2 \tanh(\pi d/(2 \alpha))}
\int_{-\infty}^{\infty}
\frac{\pi}{\pi^{2} + 4d^{2} \omega^{2}}
\, \mathrm{d}\omega
=
\frac{\pi^{2}}{4d \tanh(\pi d/(2 \alpha))}
\label{eq:B_q_1_ineq}
\end{align}

\bigskip
\noindent
{\it Estimate of $B_{q,2}$ in \eqref{eq:B_q_2}.} 
First, we consider the following estimate:
\begin{align}
& \left|
\int_{-\infty}^{\infty}
\frac{2d}{\pi^{2}+4d^{2} \omega^{2}}
\frac{\omega \cos(\alpha \omega)}{\tanh(d\, \omega)}
\, \mathrm{d}\omega
\right| 
= 
2 \left|
\int_{0}^{\infty}
\frac{2d}{\pi^{2}+4d^{2} \omega^{2}}
\frac{\omega \cos(\alpha \omega)}{\tanh(d\, \omega)}
\, \mathrm{d}\omega
\right| \notag \\
& \leq
2 \left(
\int_{0}^{\pi/(2d)}
\left|
\frac{2d}{\pi^{2}+4d^{2} \omega^{2}}
\frac{\omega \cos(\alpha \omega)}{\tanh(d\, \omega)}
\right|
\, \mathrm{d}\omega
+
\left|
\int_{\pi/(2d)}^{\infty}
\frac{2d}{\pi^{2}+4d^{2} \omega^{2}}
\frac{\omega \cos(\alpha \omega)}{\tanh(d\, \omega)}
\, \mathrm{d}\omega
\right| 
\right).
\label{eq:B_q_2_div}
\end{align}
Then, in order to estimate the first term in the parenthesis in \eqref{eq:B_q_2_div}, 
we derive the following inequality for $\omega$ with $0 \leq \omega \leq \pi/(2d)$:
\begin{align}
\frac{2d}{\pi^{2}+4d^{2} \omega^{2}}
\frac{\omega}{\tanh(d\, \omega)}
\leq 
\frac{2d}{\pi^{2} + 0}\, \frac{\pi/(2d)}{\tanh(d\, \pi/(2 d))}
=
\frac{1}{\pi \tanh(\pi/2)}.
\end{align}
Using this inequality, we have that
\begin{align}
\int_{0}^{\pi/(2d)}
\left|
\frac{2d}{\pi^{2}+4d^{2} \omega^{2}}
\frac{\omega \cos(\alpha \omega)}{\tanh(d\, \omega)}
\right|
\, \mathrm{d}\omega
& \leq 
\frac{1}{2d \tanh(\pi/2)}.
\label{eq:B_q_2_first}
\end{align}
To estimate the second term in the parenthesis in \eqref{eq:B_q_2_div},
we note that  
\begin{align}
& \frac{\mathrm{d}}{\mathrm{d} \omega}
\left(
\frac{2d}{\pi^{2}+4d^{2} \omega^{2}}
\frac{\omega}{\tanh(d\, \omega)}
\right) \notag \\
& =
\frac{d\, [ -2d \omega (\pi^{2} + 4d^{2} \omega^{2}) + (\pi^{2} - 4d^{2} \omega^{2}) \sinh(d\, \omega)]}
{(\pi^{2} + 4d^{2} \omega^{2})^{2} \sinh^{2}(d\, \omega)}
\leq 0
\end{align}
holds for $\omega$ with $\omega > \pi/(2d)$. Then, we have that
\begin{align}
& \left|
\int_{\pi/(2d)}^{\infty}
\frac{2d}{\pi^{2}+4d^{2} \omega^{2}}
\frac{\omega \cos(\alpha \omega)}{\tanh(d\, \omega)}
\, \mathrm{d}\omega
\right| \notag \\
& \leq 
\left|
\left(
\int_{\pi/(2d)}^{\pi (2l_{d, \alpha}+1)/(2\alpha)}
+
\int_{\pi (2l_{d, \alpha}+1)/(2\alpha)}^{\infty}
\right)
\frac{2d}{\pi^{2}+4d^{2} \omega^{2}}
\frac{\omega \cos(\alpha \omega)}{\tanh(d\, \omega)}
\, \mathrm{d}\omega
\right| \label{eq:B_q_2_second_line2} \\
& \leq 
\frac{2d}{2\pi^{2}}
\frac{\pi/(2d)}{\tanh(\pi/2)}
\frac{2}{\alpha}
+
\frac{2d}{\pi^{2}+4d^{2} \omega_{d, \alpha}^{2}}
\frac{\omega_{d, \alpha}}{\tanh(d\, \omega_{d, \alpha})}
\frac{2}{\alpha} \notag \\
& \leq 
\frac{2d}{2\pi^{2}}
\frac{\pi/(2d)}{\tanh(\pi/2)}
\frac{2}{\alpha}
+
\frac{2d}{2\pi^{2}}
\frac{\pi/(2d)}{\tanh(\pi/2)}
\frac{2}{\alpha}
=
\frac{2}{\pi \alpha \tanh(\pi/2)}, 
\label{eq:B_q_2_second}
\end{align}
where $l_{d, \alpha}$ is the minimum integer $l$ such that $\pi/(2d) \leq \pi (2l+1)/(2\alpha)$
and $\omega_{d, \alpha} = \pi (2l_{d, \alpha}+1)/(2\alpha)$.
For the first and second terms in \eqref{eq:B_q_2_second_line2}, we applied the inequality
\begin{align}
\int_{\pi/(2d)}^{\pi (2l_{d, \alpha}+1)/(2\alpha)} |\cos (\alpha \omega)| \, \d \omega 
\leq 
\int_{\pi (2l_{d, \alpha}-1)/(2\alpha)}^{\pi (2l_{d, \alpha}+1)/(2\alpha)} |\cos (\alpha \omega)| \, \d \omega 
= \frac{2}{\alpha}
\end{align}
and the same method for alternating series we used for \eqref{eq:p_w_int_by_parts_term1}, respectively. 
From the inequalities \eqref{eq:B_q_2_div}, \eqref{eq:B_q_2_first} and \eqref{eq:B_q_2_second}, 
we obtain
\begin{align}
B_{q,2} 
\leq \left( \frac{1}{d} + \frac{4}{\pi \alpha} \right) \frac{1}{\tanh(\pi/2)}.
\label{eq:B_q_2_ineq}
\end{align}

Finally, by combining 
\eqref{eq:bp_zero_triangle}, 
\eqref{eq:first}, 
\eqref{eq:second}, 
\eqref{eq:B_p_1_ineq}, 
\eqref{eq:B_p_2_ineq}, 
\eqref{eq:B_q_1_ineq}, and 
\eqref{eq:B_q_2_ineq}, 
we obtain the desired result.

\end{document}